\documentclass[12pt]{article}
\usepackage[utf8]{inputenc}
\usepackage{textcomp}
\usepackage{graphicx, amssymb, amsfonts,amsmath,amsthm,latexsym,mathrsfs}
\usepackage{dsfont}
\usepackage[title]{appendix}
\numberwithin{equation}{section}
\usepackage{xcolor}
\usepackage{cite}
\usepackage[colorlinks=true]{hyperref}
\allowdisplaybreaks

\newcommand{\ignore}[1]{}

\theoremstyle{plain}

\newtheorem{theorem}{Theorem}[section]
\newtheorem{lemma}[theorem]{Lemma}

\newtheorem{corollary}[theorem]{Corollary}

\newtheorem{proposition}[theorem]{Proposition}

\newtheorem{rmkk}[theorem]{Remark}
\newtheorem{hypp}[theorem]{Hypotheses}

\newcommand{\E}{\mathbb{E}}
\newcommand{\PP}{\mathbb{P}}
\newcommand{\F}{\mathcal{F}}
\newcommand{\e}{\varepsilon}
\newcommand{\LL}{\mathscr{L}}
\newcommand{\M}{\mathcal{M}}
\newcommand{\MM}{\mathsf{M}}
\newcommand{\ind}{\mathds{1}}
\newcommand{\R}{\mathbb{R}}
\newcommand{\N}{\mathbb{N}}
\newcommand{\NN}{\mathcal{N}}
\newcommand{\cc}{\boldsymbol{\mathsf{c}}}
\newcommand{\J}{\boldsymbol{\mathrm {J}}}
\newcommand{\I}{\boldsymbol{\mathrm {I}}}

\title{Fast random sampling and small noise analysis for stochastic control models}

\author{
	Sarvesh R. Iyer \\ 
	Department of Mathematics, Ashoka University,\\ Sonipat, Haryana, India-131029 \\ 
	\texttt{sarveshiyer@gmail.com}\and
	Vivek Kumar$^*$\\ 
	Department of  Mathematics  and Statistics, IIT Kanpur, \\Kanpur, Uttar Pradesh, India-208016 \\ 
	\texttt{vivekmsc118@gmail.com, vivekkumar@iitk.ac.in}}
\date{}

\begin{document}
	
	\maketitle
	\footnotetext[0]{\hspace{-0.6cm}\textsuperscript{*}Corresponding author.
		}

	
	\begin{abstract}
		In this paper, we study  a linear control system with a given state feedback law. The system is influenced by rapid \textit{random} sampling occurring at frequency $\frac1n, n\in \N$, as well as by \textit{white noise} of small intensity $\e\in (0,1]$. We study the behavior of the system as $n \to \infty$ and $\e \searrow 0$ jointly, and prove that it converges to its ideal deterministic analogue. For the random fluctuations around its analogous deterministic trajectory, we obtain either stochastic differential equations or an ordinary differential equation depending on the joint behavior of $\e$ and $n$. Further, we extend this problem to a nonlinear system driven by multiplicative white noise, where the noise intensity is scaled by a small parameter. In this case, we again perform a similar analysis as in the linear case.
	\end{abstract}

	\section{Introduction} 
	In any model based on ODE, it is important to keep systems stable, safe, and efficient. Control theory provides the tools to do this, and for this reason, it has become a very popular subject, with applications in robotics, aircraft, cars, power plants, and many other fields.  It helps design systems that work reliably under different conditions; see \cite{chen2012optimal,sontag2013mathematical, yuz2014sampled}. In practice, control is not always applied continuously. Most modern controllers are digital, and they compute and update the control only at discrete time instants. Digital control uses a computer to control a system and it checks the system at fixed time intervals and makes adjustments based on those readings. This is different from continuous control, where the system is monitored and adjusted all the time. For instance, in an inverted pendulum, the computer checks the pendulum’s angle and speed every second, decides how much force the motor should apply to keep it upright, and holds that force steady until the next check. Between two updates, the control input is kept fixed. This is known as the sample and hold method, ensuring the pendulum stays balanced (see \cite[Chapter 1]{sontag2013mathematical}). There are several applications of sample and hold control, for examples, see \cite{chen2012optimal, sontag2013mathematical, wang2023overview} and references therein.

	In many applications, such as seismic data acquisition or networked control, sampling is influenced by noise, delays, and scheduling uncertainties, making it inherently random rather than strictly periodic. This motivates modeling the sampling period as a random variable with a given probability distribution, which allows for more realistic analysis and reconstruction. For instance, in the case of an inverted pendulum controlled by a computer, the controller ideally samples the pendulum’s position at fixed intervals to keep it upright. In practice, however, network delays, sensor glitches, or energy saving strategies can cause measurements to arrive irregularly, so the controller holds the input over random time lengths while the plant evolves open-loop. Even with a high average rate, these fluctuations affect performance and stability, highlighting the need to study random sampling.
	The study of control systems with random sampling has a long and interesting history \cite{kalman1957analysis,leneman1968random, kushner1969stability}. Kalman \cite{kalman1957analysis}, was the first to explore this idea in the 1950s. Later, Leneman \cite{leneman1968random} and also Kushner  and Tobias \cite{kushner1969stability}, extended his work during the 1960s. A helpful review can be found in chapter \cite{tanwani2018stabilization}, which gives a clear overview of these developments.  For more studies and recent advances on random sampling in control systems, see \cite{dhama2021asymptotic,gao2009robust,nilsson1998stochastic, shen2012sampled, tanwani2019performance, tanwani2018stabilization, tanwani2020error, wang2023overview} and references therein.

	Further, real systems are often affected by some small external effects and uncertainties. To describe these effects in dynamical systems, stochastic models of differential equations are used. In continuous time, this leads to stochastic differential equations (SDEs), where a small random noise term represents these external influences in the system, \cite{dhama2023fluctuation, dhama2025asymptotic, dhama2020approximation, evans2012introduction,oksendal2003sde}.
	
	In this article, we study a framework with both  random sampling mechanisms and external small stochastic effects. The important question in  such situations is whether such sampled systems keep the important properties of their continuous deterministic counterparts. For instance, does stability remain intact$?$  Does performance degrade$?$ These issues are central in the study of computer controlled and networked control systems. There has been limited research on the interaction between sampling and noise \cite{dhama2023fluctuation, dhama2025asymptotic, dhama2020approximation,dhama2021asymptotic}. The present work aims to address these questions for a given dynamical problem.


	As mentioned earlier, several works on sample and hold methods are available in the literature ( see \cite{chen2012optimal, dhama2023fluctuation, dhama2025asymptotic, dhama2020approximation, dhama2021asymptotic,  evans2012introduction, gao2009robust, nilsson1998stochastic, shen2012sampled, sontag2013mathematical, yuz2014sampled, kalman1957analysis,leneman1968random, kushner1969stability,  tanwani2019performance, tanwani2018stabilization, tanwani2020error}). For broader survey discussions, we refer the reader to \cite{tanwani2018stabilization, wang2023overview, zhang2023sampled}. However, in the present paper, we focus only on the studies that are most closely related to our work. In particular, Pahlajani and Dhama have recently obtained several noteworthy results on sampled data and hold control systems under small noise perturbations \cite{dhama2023fluctuation, dhama2025asymptotic, dhama2020approximation}, as well as on systems with random sampling \cite{dhama2021asymptotic}.
	In their works  \cite{dhama2023fluctuation, dhama2025asymptotic, dhama2020approximation}, both authors mainly consider a controlled dynamical system governed by the deterministic ODE
	$\dot y_t = c(y_t,u_t)$, with initial condition $y_0\in\mathbb{R}^d$,
	where $c$ is the drift and $u_t$ is the control.
	When the state is continuously observed, a feedback law $u_t=\kappa(y_t)$ is applied,
	leading to the closed loop system $\dot y_t = c(y_t,y_t)$,
	where $c(x,z)=c(x,\kappa(z))$, which describes the ideal deterministic behavior.
	To model digital controllers, the control is updated only at discrete times and held
	constant between updates, resulting in the sample and hold system
	$\dot y_t^\delta = c\big(y_t^\delta, y_{\delta\lfloor t/\delta\rfloor}^\delta\big)$
	with sampling period $\delta>0$.
	Finally, small external noise perturbations are incorporated, yielding
	$\dot y_t^{\e,\delta}
	= c\big(y_t^{\e,\delta}, y_{\delta\lfloor t/\delta\rfloor}^{\e,\delta}\big)
	+ \e\,\dot {\mathsf{F}}_t$,
	where $\e>0$ measures the noise intensity and $\dot {\mathsf{F}}_t$ is a general
	stochastic forcing.  Since their approaches and analytical settings are closely connected to ours, we review these works individually to highlight their main contributions. Subsequently, in the next subsection,  we clarify how our results differ from and extend their findings.
	
	Pahlajani and Dhama  first  studied the linear systems in this direction. In their first work \cite{dhama2020approximation}, they considered a linear system given by
	\begin{equation*}\label{DhamaSDE}
		dY^{\e,\delta}_t = \big[AY^{\e,\delta}_t + BU^{\e,\delta}_k \big] dt + \e  dW_t, 
		\quad U^{\e,\delta}_k = -K (X^{\e,\delta}_{k\delta-} + \e V_{k\delta}), \quad ,
	\end{equation*}
	for $t \in [k\delta,(k+1)\delta)$ and where $W_t$ and $V_{k\delta}$ are independent Brownian motions representing system and measurement noise. For positive integers $m$ and $d,$  $A\in \R^{d\times d}, B\in  \R^{d\times m}$ and $K\in \R^{m\times d}$ are constant matrices.
	The main objective of this paper  \cite{dhama2020approximation} was to characterize the limiting behavior of $X^{\e,\delta}_t$ as $\e,\delta \to 0$, both in terms of the mean dynamics, which follow the deterministic ODE, and the fluctuations, which are governed by SDEs. 
	
	In the  work \cite{dhama2023fluctuation}, both the authors have generalized  the previous work \cite{dhama2020approximation} by considering nonlinear drift with multiplicative noise as
	\begin{equation*}\label{stochastic}
		dX_t^{\e,\delta} = c(X_t^{\e,\delta}, X^{\e,\delta}_{\delta\lfloor t/\delta\rfloor})dt 
		+ \e\sigma(X_t^{\e,\delta})dW_t.
	\end{equation*}
	Here $c: \R^d\times \R^d\to \R^d$ is a sufficiently regular mapping and
	takes the specific form $c(x,y)=f(x)+g(x)\kappa(y), x,y\in \R^d$, with the functions $f:\R^d \to \R^d$, $g:\R^d \to \R^{d\times m}$, $\kappa:\R^d \to \R^m$ satisfying certain regularity conditions.
	The authors obtain asymptotic approximations for both the mean dynamics and the fluctuations 
	of $X_t^{\e,\delta}$ as $\e,\delta \searrow 0$. The mean behavior is described by a 
	limiting ODE of the form 
	\begin{equation*}\label{DhamalimitingODE}
		x_t  = x_0 + \int_0^t [f(x_s)+g(x_s)\kappa(x_s)] ds
	\end{equation*}
	on the interval $[0,T]$ while the fluctuations are captured by a linear SDE whose form depends on the relative 
	rates at which $\e$ and $\delta$ vanish.

	Further, in the article \cite{dhama2025asymptotic},  Dhama studied the following nonlinear sampled data systems perturbed by both Brownian noise and small jumps via Poisson random measures
	\begin{align*}\label{DhamaJumpProblem} dY_t^{\e,\delta}=c(Y_{t-},Y^{\e,\delta}_{\pi_\delta(t)-})dt+\e\sigma(Y_{t-})dW_t+\e\int_{0<|x|<1}F(Y_{t-},x),\tilde N(dt,dx)
	\end{align*} with $\pi_\delta(t)=\delta\lfloor t/\delta\rfloor$. Methodologically, the work performs a joint small parameter expansion under three regimes for $\delta/\e$ and proves pathwise versions of Law of Large Numbers ( LLN) and Central Limit Theorem (CLT) type results.

	The paper  \cite{dhama2021asymptotic},  differs from the above three studies. It considers a linear state feedback system  implemented using a sample and hold mechanism (as mentioned in \cite{dhama2020approximation}), where the sampling times are random and follow a renewal process. Unlike the earlier works, randomness enters only through the sampling times, and no extra noise is added to the system.  The primary objective is to understand how fast, yet finite rate, random sampling alters the system dynamics when compared to the ideal continuous time model. Using the law of large numbers and the central limit theorem results for random matrix products, the study describes both the average behavior and the typical fluctuations caused purely by the randomness in sampling.
	\subsection{ The Novelty and Methods}
	This work primarily deals with a linear control model subject to both rapidly increasing random sampling at rate $n\in\N$ and an external noise perturbation of very small intensity $\e$. The primary goal is to investigate law of large numbers (LLN) and central limit theorem (CLT) type results in the regime where the effects of random sampling and noise perturbations both vanish, i.e., when $(\frac1n, \e)\searrow 0.$
	One of the main differences between our work and that of \cite{dhama2023fluctuation, dhama2021asymptotic, dhama2025asymptotic, dhama2020approximation} is the simultaneous consideration of both random sampling and an external noise term. In particular, our work is an extension of the work \cite{dhama2020approximation, dhama2021asymptotic} by updating sampling method in \cite{dhama2020approximation} while  it  extend \cite{dhama2021asymptotic}  by introducing an external  noise forcing term with very  small intensity $\e$.    To our best knowledge, the present work has not been done yet in the literature.

	Further, we also generalize the results obtained in our linear setting to a more broader framework that includes nonlinear systems and  multiplicative white noise. This extension is inspired by the work \cite{dhama2025asymptotic}, where author has worked under deterministic sampling. In contrast, along with the nonlinearity, we also consider random sampling, which makes the problem more complex and interesting. In this context, Lemma \ref{G4decomposition} of linear case, plays a key role in establishing the CLT and makes calculation more simple. To the best of our knowledge, this setting is also new and has not yet been investigated in the existing literature.  This work also covers the results obtained in \cite{dhama2023fluctuation} in the context of random sampling.
	
	Earlier, we mentioned that this paper deals with two interacting sources of randomness, which makes the analysis considerably more challenging.
	This can be considered as a multiscale problems \cite{has1966stochastic, khasminskij1968principle,cerrai2009averaging, givon2007strong, brehier2022averaging, cerrai2025averaging}, where asymptotic behavior plays a key role. Multiscale systems of such kind were first investigated by Khasminskii in his seminal works \cite{has1966stochastic,khasminskij1968principle}, where he developed the averaging method to derive simplified dynamics for the slow component by averaging over the fast one. This idea has since inspired extensive research; for example, see \cite{cerrai2009averaging, givon2007strong, brehier2022averaging, cerrai2025averaging} and references therein. The other techniques to solve multiscale problems is the idea used by  Freidlin and Sowers \cite{freidlin1999comparison}  where they have solved Large deviation principle by using first homogenization and then averaging principle.

	In the present article, in the case of the linear system, our analysis is inspired by and partially follows the framework introduced in \cite{dhama2020approximation}, where authors have shown the LLN and CLT type results for their linear control problems with periodic sampling.  For the linear setting in our case,  in the LLN result, we have shown that as $\e\searrow 0$ and $n\to \infty,$ the the sampled stochastic system behaves close to the ideal deterministic system.
	In the CLT part, we have analyzed  the behavior of  the fluctuations of the sampled stochastic system around the deterministic controlled trajectory. Depending on the balance between noise intensity and sampling frequency, these fluctuations converge either to a stochastic differential equation or to a deterministic equation. To derive these results, we encounter technical difficulties due to the presence of combination of random sampling together with external noise. This goal cannot be achieved by simply applying the techniques from \cite{dhama2020approximation}. Therefore, we have developed an extended framework beyond that method.
	To overcome these challenges, we first perform a careful decomposition of the main expression into several components, each of which can be analyzed separately. In order to do this, well established methods from probability theory are employed thoughtfully and in a sophisticated manner. The Wald’s identity to handle sums of random variables, the elementary renewal theorem to control renewal processes. Further, we use the Donsker’s theorem for approximating scaled processes by Brownian motion, and Doob’s maximal inequality to bound the supremum of martingales. The contribution of the small noise term is treated using the Burkholder-Davis-Gundy (BDG) inequality in combination with Doob’s maximal inequality, ensuring precise control over stochastic fluctuations. This systematic approach allows us to rigorously establish the limiting behavior of the system, despite the complexity introduced by the simultaneous presence of sampling and noise.

	\subsection{Plan of the paper} In the next section, we introduce the overall framework of our study and present the main results. We first formulate the problem rigorously and set out the notation and assumptions that will be used throughout the paper. We then outline the key ideas and techniques underlying our analysis. Section \ref{section3} contains the necessary definitions and auxiliary results that will be used in the subsequent sections. In Section \ref{section4}, we derive LLN type results, which describe the average, or deterministic, behavior of the system. Section \ref{section5} is devoted to CLT type results. This section is divided into two subsections, in which the detailed proofs of the CLT are provided. Subsection \ref{subsection5.1} focuses on the decomposition of the random sampling component, whereas Subsection \ref{subsection5.2} is concerned with the analysis of the noise component. Finally, in Section \ref{section6}, we extend our framework to a more general class of linear problems and obtain analogous results for this broader setting.
	
	\subsection{Notations}Let $(\Omega, \mathcal{F}, \mathbb{P})$ be a probability space, and let 
	$\{\xi_i\}_{i=1}^{\infty}$ be a sequence of i.i.d.\ positive real-valued random variables defined on this space. 
	Throughout the paper, $|\cdot|$ denotes the induced matrix norm, and for two matrices $A$ and $B$, 
	we write $|A|\cdot|B|$ to denote the product of their norms. The symbol $C$ denotes a generic positive constant whose value may change from line to line. Whenever necessary, its specific dependence will be indicated explicitly in the relevant statement. 
	If no dependence is specified, $C$ represents a universal positive constant.
	We denote the set of natural numbers by $\mathbb{N}$ and the set of positive integers by $\mathbb{Z}_+$. 
	The notation $\textit{Var}$ is used to represent variance.  
	Further, pth power over expectation is denoted by $ \E[\cdot]^p:=\left(\int_{\Omega} \cdot~ d \PP\right)^p.$

	\section{ Problem Formulation,   Assumptions and Main Results}\label{section2}
	Before going to the introduce the dynamical system of our problem, let us first discuss about the random sampling setup which we are going to use in our paper. In our sampling and hold  setup, we  are going to replace deterministic discretizations by a renewal process, using the prescription from \cite[ page 360]{dhama2021asymptotic}. For this, we define  time process associated to  random variables $\{\xi_i\}_{i=1}^{\infty}$  as  
	\begin{equation*}
		\tau_k := \sum_{i=1}^{k} \xi_i\quad ,\quad \tau_0 =0,
	\end{equation*}
	and the renewal process 
	\begin{equation*}
		N_t := \sup\{k \in \mathbb Z_+ : \tau_k \leq t\}.
	\end{equation*}
	Throughout the paper we are assuming  that the sequence $\{\xi_i\}_{i=1}^{\infty}$ has finite moment generating functions. The appropriate ``fast sampling" takes place via a parameter $n \in \mathbb N$. That is, for $k\in \mathbb Z_+, n\in \mathbb{N},$ let $\xi^n_k := \frac 1n \xi_k$  and let $\{\tau^n_k\}_{k \in \mathbb Z_+}, \{N^n_t\}_{t \geq 0}$ be the associated time and renewal processes respectively given by \begin{equation}\label{taunk}
		\tau^n_k := \sum_{i=1}^{k} \xi^n_i\quad ,\quad \tau^n_0=0,
	\end{equation}
	and \begin{equation}\label{Nnt}
		N^n_t := \sup\{k \in \mathbb Z_+ : \tau^n_k \leq t\}.
	\end{equation}	
	By definitions, $\xi^1_k = \xi_k$, $\tau^1_k = \tau_k$ for all $k \geq 1$, and $N^1_t = N_t$ for all $t \geq 0$. These will be used interchangeably. Given $n \in \mathbb{N}$, the appropriate discretizations function at level $n$ in this setting is the function $t \to \tau^n_{N^n_t}$, which is the last sampling point smaller than $t$ at the scale $\frac 1n$. Observe that
	\begin{equation}\label{nnt}
		N^n_T = \sup\{k : \tau^n_k \leq T\} = \sup\{k : \tau^1_k \leq nT\} = N^1_{nT}.
	\end{equation}
	Let us now turn our attention to the dynamical feedback control framework. We consider the linear differential equation given below.
	\begin{align}\label{ode1}
		\dot{x}_t=Ax_t+Bu_t; ~~~x(0)=x_0\in \R^d
	\end{align}
	where $A\in \R^{d\times d}, B\in  \R^{d\times m}$,  $m, d $ are fixed integers and $t \in [0,T]$ for some finite $T>0$.
	We apply the feedback control law $u=-Kx$, where $K\in  \R^{m\times d}$ is a suitable  matrix. Throughout the paper, we assume that the matrix $A$ is invertible, matrix $A-BK$ generates a strongly continuous semigroup $e^{t(A-BK)}$ and matrices $A$ and $BK$ commutes. Consequently, one can get following an equivalent integral form of equation \eqref{ode1}
	\begin{equation}\label{contrlUsoln}
		x_t=x_0+\int_0^t (A-BK)x_s ds=e^{t(A-BK)}x_0.
	\end{equation}
	For the sake of simplicity, we assume that the initial data $x_0$ is bounded and its bound is incorporated into a generic constant $C$. Now, consider implementing the feedback control law described earlier using a sample and hold strategy in dynamics \eqref{ode1}. For any fix $n\in \mathbb{N}$ and any $k\in \mathbb Z_+$, suppose that the system is sampled at time instant $t=\tau_k^n$.  At each sampling instant $t$, the current state $x_{t}=x_{\tau_k^n}$ is measured, and the control input is determined according to  
	$$u_{\tau_k^n}=-Kx_{\tau_k^n},$$
	and this control value is then kept constant over the interval $[\tau_k^n, \tau_{k+1}^n).$  The trajectory $x_{t}^n$ over time $t \in [\tau_k^n, \tau_{k+1}^n)$ can be obtained by solving the differential equation
	\begin{equation} \label{randomcontrol}
		\dot{x}^n_t = Ax^n_t + B(-K x^n_{\tau^n_k}) = Ax^n_t - BK x^n_{\tau^n_k}, 
	\end{equation}	
	with initial condition $x_{\tau_k^n}^n=x_{\tau_k^n-}^n$. In this interval $t\in [\tau_k^n, \tau_{k+1}^n), x^n_{\tau^n_k}=x^{ n}_{\tau^n_{N^n_t}}$ is fixed and the system is linear with constant coefficients on each interval $[\tau^n_k,\tau^n_{k+1}), ~ k\in \mathbb Z_+.$ Hence, the solution can be  given by
	\begin{equation}\label{randomcontrolsoln}
		x_t^n
		= \Big( e^{A(t-\tau^n_k)}
		- \int_{\tau^n_k}^{t} e^{A (t-s)} BK ds
		\Big) x^n_{\tau^n_k},
		~~~t \in [\tau^n_k,\tau^n_{k+1}).
	\end{equation}
	As we can see that, the trajectory depends on $n$ also, and we can expect that as $n\to\infty$, it should converge to the continuous time solution \eqref{contrlUsoln}.
	
	We now consider a situation when the system is influenced by some  small external random force. For this, let us consider a  filtration $\{ \F_t: t\geq 0\}$ on the  probability space  $(\Omega, \F,\PP)$  with  satisfying standard conditions (see \cite{oksendal2003sde}). In this probability space, consider an  n-dimensional Brownian motion, $W=\{W_t\}_{t\geq 0}$ which represent the external noise in the  system and and is  independent of the sequence $\{\xi_i\}_{i=1}^{\infty}$.  We suppose the intensity of the noise is very small, say of order $\e\in (0,1)$. Now, the system  \eqref{randomcontrol} evolves according to the following stochastic hybrid system with additive noise 
	\begin{equation} \label{controlwithnoise}
		dX_t^{\e, n} = \left( AX_t^{\e, n} - BK X^{\e, n}_{\tau^n_k} \right) dt + \e  dW_t, \quad t \in [\tau_k^n, \tau_{k+1}^n),
	\end{equation}
	and with initial data $X_0^{\e, n}=x_0.$ It can be noticed that in the interval $t\in [\tau_k^n, \tau_{k+1}^n), X_{\tau_k^n}^{\e,n}=X^{\e, n}_{\tau^n_{N^n_t}}$ is fixed.  Following the  representation \eqref{randomcontrolsoln}, we have 
	\begin{align*}
		X_t^{\e,n}=\left[e^{(t-{\tau_k^n})A}-\int_{{\tau_k^n}}^t e^{(t-s)A}BK ds\right] X_{\tau_k^n}^{\e,n}+\e\int_{{\tau_k^n}}^te^{(t-s)A} dW_s
	\end{align*}
	and thus for any $t\geq 0,$ we have 
	\begin{align*}
		X_t^{\e,n}=\sum_{k \ge 0} \ind_{[\tau_k^n,\tau_{k+1}^n)} (t) \Bigg\{\Bigg[e^{(t-{\tau_k^n})A}-\int_{{\tau_k^n}}^t e^{(t-s)A}BK ds\Bigg] X_{\tau_k^n}^{\e,n}
		+\e\int_{{\tau_k^n}}^te^{(t-s)A} dW_s\Bigg\}.
	\end{align*}
	\textit{In \eqref{controlwithnoise}, if we fix $\e\in (0,1)$ and take $n\to \infty$, we can expect the limiting dynamics  are described by the process $X_t^\e$  which has following form
		\begin{align}\label{nlimit}
			dX_t^\e=[(A-BK)X_t^\e]dt+\e dW_t, ~~~~~X_0^\e=x_0.
		\end{align}
		Further, we are expecting that by taking limit $\e\searrow 0,$ the convergence of the solution $X^\e_t$ given by \eqref{nlimit} to the deterministic trajectory $x(t)$ solving \eqref{contrlUsoln} is straightforward.
		Our primary objective is to determine how the combined presence of random sampling and Brownian motion affects these classical limits. More specifically, we are interested in how the relative rates at which $\e \searrow 0$ and $n \to \infty$ simultaneously  modify the convergence of $X^{\e, n}_t$ to $x(t)$.} 
	
	Considering the joint convergence of $\e \searrow 0$ and $n \to \infty$, we identify the following three asymptotic regimes:
	\begin{equation}\label{Regimes}
		\cc :=
		\lim_{\e\searrow 0, n\to \infty} \frac{1}{n\e}\begin{cases}
			=0 & \text{Regime 1}\\
			\in (0,\infty) & \text{Regime 2}\\
			=\infty & \text{Regime 3}.
		\end{cases}       
	\end{equation}
	Here, in the first Regime, the sampling  process is very fast while the noise decreases slowly in comparison. Consequently, the system is already well sampled while the noise is still present. In this case, the main effect comes from small noise acting in a rapidly varying environment. In the second Regime, sampling and noise evolve at almost the same speed, i.e., neither one dominates the other. Finally, in the third Regime,  the noise intensity $\e$ decreases much faster than the sampling parameter $n$ increases. As a result, the effect of external noise become negligible at an early stage of the system's evolution.  For the cases $\cc=0$ and $\cc \in (0,\infty)$, we set
	\begin{equation}\label{kappa}
		\varkappa(\e) := \left|\frac{1}{ n\e}-\cc\right|,
	\end{equation}
	Note that $\displaystyle\lim_{\e \searrow 0, n\to \infty} \varkappa(\e)=0$. 
	
	Now,  for each $n \in \mathbb{N}$, define  a time–discretizations map $\pi^n: [0, \infty) \to \{ \tau^n_k : k \in \mathbb{Z}_+ \}$ by
	$$
	\pi^n(t) := \tau^n_{N^n_t} \quad \text{for } t \in [0, \infty).
	$$
	That is, for each time $t \in [0,\infty)$, ${\pi}^n$ picks the closest previous sampling time in the random discretizations grid. In particular, if $t \in [\tau^n_k, \tau^n_{k+1})$ for some $k \in \mathbb Z_+$, then
	$N^n_t  = k$ by definition \ref{Nnt}, and hence ${\pi}^n (t) = \tau^n_k$. With the help of $\pi^n$, we can rewrite equations \eqref{randomcontrol} and \eqref{controlwithnoise} for $t\in [0,\infty)$ as follows:
	\begin{equation}\label{M2withoutnoiseforallt}
		\dot{x}^n_t= A x^n_t-BK x^n_{\pi^n(t)} 
	\end{equation}
	with $x^n_0=x_0$ and 
	
	\begin{equation}\label{LSDEwRS}
		dX_t^{\e, n} = \left( AX_t^{\e, n} - BK X^{\e,n}_{\pi^n(t)} \right) dt + \e  dW_t, 
	\end{equation} with $X^{\e, n}_0=x_0$ or equivalently 
	\begin{equation}\label{M2withnoiseforallt}
		X_t^{\e,n}=x_0+\int_0^{t}(AX_s^{\e,n}-BKX^{\e,n}_{\pi^n(s)}) ds +\e\int_{0}^{t} dW_s.
	\end{equation}
	We can now state our first  main results of this paper. The first result is a ``law of large numbers", guaranteeing that as $\e \searrow 0, n \to \infty$, $X^{\e,n}_t \to x_t$ uniformly in $L^p, p \geq 1$ in all the Regime.

	\begin{theorem}\label{LLN}
		Let $x(t)$ denotes the solution of \eqref{contrlUsoln} and let $X^{\e,n}_t$ be the solution to SDE \eqref{controlwithnoise}. Let $T\geq 0$ be fixed. Then, for arbitrary $\e>0, n\in \mathbb N, p\geq 1,$ there exists $ C_{ABKTp} > 0, $ depending only on $A, B, K$,$T$ and $p$ such that
		\begin{equation*}
			\E \left[ \sup_{0 \leq s \leq T} |X_s^{\e, n} - x_s|^p \right] \leq C_{ABKTp}  \left(\E[\NN_p]+\e^p\right),
		\end{equation*}
		where $\NN_p = \int_0^T (s-\pi^n(s))^p ds$ satisfies $\E[\NN_p]\le \frac{1}{n^p}C_{T\xi_1}$  and so $\to 0$ as $n\to \infty$ by Corollary~\ref{Np} below.
	\end{theorem}	
	Our next main aim is to study the fluctuation of $ X_t^{\e,n}$ about $x(t)$ with respect to parameters $\e$ and $n$. For this, we introduce the fluctuation processes under different scaling regimes defined in \eqref{Regimes}. For Regimes 1 and 2, we define the  fluctuations scaled  by $\e$ as
	\begin{equation*}
		Z^{\e,n}_t := \frac{X_t^{\e,n}-x_t}{\e},
		\label{eq:fluct-Z}
	\end{equation*}
	while for Regime 3  we consider fluctuations under the scaling $1/n$ 
	\begin{equation*}
		Q_t^{\e,n} := \frac{X_t^{\e,n}-x_t}{1/n}.
		\label{eq:fluct-Q}
	\end{equation*}
	With the help of  \eqref{contrlUsoln} and \eqref{M2withnoiseforallt}, we obtain
	$$X_t^{\e,n}-x_t= \int_{0}^{t}(A-BK)(X_s^{\e,n}-x_s) ds+BK\int_{0}^{t}\left(X_s^{\e,n}-X^{\e,n}_{\pi^n(s)}\right)ds+\e W_t$$
	Therefore, 
	\begin{equation}\label{eq:mainterm}
		Z^{\e,n}_t=\int_{0}^{t}(A-BK)Z_s^{\e,n} ds+BK\int_{0}^{t}\frac{X_s^{\e,n}-X^{\e,n}_{\pi^n(s)}}{\e}ds+ W_t,
	\end{equation}
	and similarly
	\begin{equation}\label{eq:maintermQ}
		Q_t^{\e,n}=\int_{0}^{t}(A-BK)Q_s^{\e,n} ds+BK\int_{0}^{t}\frac{X_s^{\e,n}-X^{\e,n}_{\pi^n(s)}}{1/n}ds+ \e nW_t.
	\end{equation}
	We expect that as $\epsilon \searrow 0, n \to \infty$, the random fluctuations $Z^{\e,n}_t$ (in Regimes 1,2) and $Q^{\e,n}_t$ (in Regime 3) will converge to some effective fluctuation processes $Z_t$ and $Q_t$ respectively, which are independent of $\e,n$ and $X^{\e,n}_t$. If such an approximation   holds, then we can approximate $X^{\e,n}_t = x_t + \e Z_t + \gamma_{\e, n}$ in Regimes $1,2$ and $X^{\e,n}_t = x_t + \frac 1n Q_t + \lambda_{\e,n}$ in Regime $3$, where $\gamma_{\e, n}$ and $\lambda_{\e,n}$ vanish as $\e\searrow 0$ and $n\to \infty.$
	
	We analyze the two cases separately, beginning with Regime 1 and 2. 
	In  Regimes 1 and 2, since  $\displaystyle \lim_{\substack{\e \searrow 0 \\ n \to \infty}} \frac{1}{n\e}=\cc\in [0,\infty),$ there exists $\e_0\in (0,1)$ such that $\varkappa(\e)<1$ whenever $0<\e<\e_0.$
	In the both Regimes, one of the main challenges comes from the term
	$\displaystyle\int_{0}^{t}\frac{X_s^{\e,n}-X^{\e,n}_{\pi^n(s)}}{\e}ds$
	which is central to understand the behavior of the system $Z_t^{\e,n}$. It is not immediately clear how it evolves over time when $\e$ is very small and $n$ is very large. To get a handle on this, let us suppose that we can show that, in both regimes, the integral converges to a well defined function. Keeping this idea, we introduce a function $\ell(t)$ such that
	\begin{equation}\label{ellDefn}
		\ell(t):=\cc\int_0^t M (A-BK)x(s) ds, 
	\end{equation}
	where $M:=\frac{\E[\xi_1^2]}{2\E[\xi_1]}$ . Further, we define the process $Z=\{ Z_t:t\geq 0\}$ as the unique strong solution of
	\begin{equation}\label{ZtDefn}
		Z_t:=\int_0^t(A-BK)Z_s ds- \cc M BK\int_0^t (A-BK)x(s) ds +W_t,
	\end{equation}
	Now, suppose we can establish that the   the function $\ell(t)$  is such that 
	\begin{align*}
		\lim_{\substack{\e,1/n \to 0\\ 1/{n\e} \to c}} \int_{0}^{t}\frac{X_s^{\e,n}-X^{\e,n}_{\pi^n(s)}}{\e} ds=\ell(t).
	\end{align*}
	Then, in the  Regimes $1$ and $2$, the fluctuation  process $Z^{\e,n}_t$ can be approximated by $Z_t$ for small $\e$ and large $n$. This is the content of our next main result, as presented below. 
	\begin{theorem}[Central Limit Type Theorem]\label{CLTresult}
		Let $x(t)$ denotes the solution of \eqref{contrlUsoln} and let $X^{\e,n}_t$ be the solution to SDE \eqref{controlwithnoise}. Assume that the scaling parameters  fall into Regime $i\in \{1,2\},$ i.e., $\cc\in [0,\infty).$    Then, there exists a number $\e_0\in (0,1)$ such that for every fixed $T>0$ and for all $0<\e<\e_0,$ 
		\begin{align}\label{CLTExpression}
			\E\Big[\sup_{0\leq t\leq T} \left|X_t^{\e,n}-x(t)-\e Z_t\right|\Big] \leq \left[ \cc(n^{-1/2}(n^{-1}+\e)+n^{-1/4})
			+ M\varkappa(\e)+n^{-1/2}
			\right]C_{ABKT\xi_1}.\nonumber
		\end{align}
	\end{theorem}
	\begin{rmkk}
		As we know that in Regime 1, $\cc=0$ and speed of $\e$ is very slow compare to $1/n.$ Therefore, if we take $\e \approx \frac{1}{n^{1-\delta}}$ with $0<\delta<1$, then the rate becomes  $M\varkappa + n^{-1/2}=M/(n\e)+n^{-1/2}\approx M n^{\delta-2}+n^{-1/2}\approx n^{-1/2}$.  In term of $\e,$ the rate will be $\e^{\min\left(\frac{1}{2(1-\delta)},\frac{\delta}{1-\delta}\right)}.$ 
		In Regime 2, $\cc$ is a fixed constant, and $\e$ and $1/n$ both converge to $0$  at the approximately  same rate. Consequently, the dominant term in the overall rate is $(M\varkappa + n^{-1/4})$.  Here, the convergence rate is determined by how $\varkappa \to 0$, i.e., by the rate at which $n\e \to 1/\cc$.
	\end{rmkk} 
	
	Further, in Regime 3, a similar result holds; however, in this case,  the process $Q_t^{\e,n}$ converges to a deterministic limit. This occurs because, in this regime, $\e$ decays far more rapidly than $n$ grows. The result can be formulated as follows.
	\begin{theorem}\label{CLTresult2}
		Let $x(t)$ denotes the solution of \eqref{contrlUsoln} and let $X^{\e,n}_t$ be the solution to SDE \eqref{controlwithnoise}. Assume that the scaling parameters  fall into Regime 3, i.e.,  $\cc=\infty.$  Define the process $Q=\{ Q_t:t\geq 0\}$ as the unique strong solution of
		\begin{equation*}
			Q_t=\int_0^t(A-BK)Q_s ds- BK M\int_0^t (A-BK)x(s) ds.
		\end{equation*}
		Then, there exists a number $n_0$ such that for every fixed $T>0$ and for all $0<n_0<n,$ 
		\begin{equation*}
			\E\left[\sup_{0\leq t\leq T} \left|X_t^{\e,n}-x(t)-\frac1n Q_t\right|\right]\leq   \left(\frac{1}{n^{1/4}}+\e \sqrt{n}\right)C_{ABKT\xi_1}.		\end{equation*}		
	\end{theorem}
	The proof of this theorem follows from the proof of Theorem \ref{CLTresult}; therefore, we omit the proof here. We observe that, in this regime, $\cc=\infty$, meaning that $\e n\to 0$ as $\e\searrow 0$ and $n\to\infty.$
	\begin{rmkk} 
		The term M contributing in the effective drift in function \eqref{ellDefn} is the long run time averaged age or long run residual life of the renewal process $\{N_t^n\}_{t \geq 0}$ created by $\{\xi_i^n\}_{i=1}^{\infty}$ (see \cite[Chapter 5]{gallager2013stochastic}).  In  the present renewal process, the time between consecutive sampling is random and given by $\xi_1^n, \xi_2^n, \xi_3^n, \dots$. The \emph{age} at any time is how long it has been since the last sampling, and the \emph{residual life} is how long until the next sampling. Within each interval between two sampling, the age increases linearly from $0$ to the length of the interval, while the residual life decreases linearly from the interval length to $0$. If we plot age or residual life versus time, each interval forms a right angled triangle, with the area under the triangle representing the total accumulated age or residual life in that interval. Dividing this area by the interval length gives the average age or residual life in that interval. Summing over all intervals and averaging over time $[0,T]$, the long run time average of age or residual life converges to
		${\displaystyle
			\frac{\mathbb{E}[X_1^2]}{2\,\mathbb{E}[X_1]}}.
		$
		The triangle picture is just a simple visual tool to understand how the average over time comes from the linear increase or decrease of age and residual life within each interval. For further details on the term $M$, the reader is referred to \cite{gallager2013stochastic}.
	\end{rmkk}
	\begin{rmkk}
		The trajectories $\{Z_t\}_{t \geq 0}$ and $\{Q_t\}_{t \geq 0}$ can be explicitly calculated.  In fact, we  can rewrite $Z_t$ in \eqref{ZtDefn} as 
		\begin{align*}
			dZ_t= (A-BK)Z_tdt-\cc MBK(A-BK)x(t) dt+dW_t, \qquad Z_0=0.
		\end{align*}
		Define $Y_t:=e^{-(A-BK)t}Z_t$ and applying It\^{o} formula, we get
		\begin{align}\label{ZtSoln}
			Z_t=-~\cc MBK\int_0^t e^{(A-BK)(t-s)}(A-BK)x(s) ds+\int_0^t e^{{(A-BK)(t-s)}}dW_s.
		\end{align}
		Thus, the limiting fluctuation process $Z_t$ satisfies a linear SDE of Ornstein-Uhlenbeck (OU) (see \cite{oksendal2003sde}) type with additive noise. In the Regime 1, i.e,  when $\cc=0,$ it is pure OU process while in Regime 2 it has a extra linear drift $-\ell(t).$ Since it has linear drift with deterministic coefficient and nature of noise is additive, the equation \eqref{ZtSoln} admits an unique strong solution and in particular it is a Gaussian process. This justifies referring  the Theorem \ref{CLTresult} as a central limit type theorem. A similar explanation can be given about $Q_t$ and Theorem \ref{CLTresult2}, in fact it is more straightforward since $Q_t$ is a deterministic function. 
	\end{rmkk}
	
	\section{Preliminaries}\label{section3}
	In this section, we mention the preliminaries that are required to prove our main results. Let us now start by presenting certain stochastic tools that will be used in proving our results. 
	
	\subsection{Classical Results from Literature}
	
	The first, Wald's Identity, will be used to simplify the first two moments for stopped sums of i.i.d. random variables.
	\begin{lemma}[Wald's Identity]\label{Waldlemma}\cite[Theorem 5.5.3]{gallager2013stochastic}
		Let $X_0,X_1,X_2,\ldots$ be an i.i.d. sequence of random variables. If $N$ is a stopping time with respect to the filtration generated by $\{X_i\}_{i \geq 0}$, then $$
		\E\left[\sum_{i=0}^{N} X_i\right] = \E[N+1]\E[X_i].
		$$
		Furthermore, suppose that $\E[X_i] = 0$. Then,$$
		\E\left[\left(\sum_{i=0}^N X_i\right)^2\right] = \E[N+1]Var(X_i) = \E[N+1]\E[X_i^2].
		$$
	\end{lemma}
	We remark that our version of the identity is stated for the sequence $X_i, i \geq 0$. The analogous cited version is for $X_i, i \geq 1$, and our statement can be easily derived from it.
	
	The elementary renewal theorem gives an asymptotic rate for expected renewal times. 
	\begin{theorem}[The elementary renewal theorem]\cite[Theorem 5.6.2]{gallager2013stochastic}\label{thm:renewal}
		Let $X_1,X_2,\ldots$ be an integrable, positive sequence of i.i.d. random variables. Suppose that $N(t) = \sup\left\{k :\sum_{i=1}^k X_i \leq t\right\}$. Then, $$
		\lim_{t \to \infty} \frac{\E[N(t)]}{t} = \frac 1{\E[X_1]}.
		$$
	\end{theorem}	
	We also require Donsker's theorem for both i.i.d. random variables and renewal processes. The latter is a functional central limit theorem for the renewal times as the intensity scales linearly.
	\begin{theorem}[Donsker's theorem]\cite[Theorem 16.1]{billingsley1968convergence} \label{usualdonsker} If $X_i, i \geq 1$ are i.i.d. positive random variables with finite variance, then for any $T>0$, as random elements of the Skorokhod space $D[0,T]$ we have 
		$$
		\sum_{i=1}^{\lfloor nt \rfloor} \frac{X_i - \E[X_i]}{\sqrt{nVar(X_1)}} \Rightarrow B_t
		$$
		\textbf{weakly} as $n \to \infty$, where $(B_t)_{t \in [0,T]}$ is a Brownian motion. 
	\end{theorem}
	
	\begin{theorem}[Donsker's theorem for renewal processes]\cite[Theorem 17.3]{billingsley1968convergence}\label{donsker} Suppose that $X_i,i=1,2,\ldots$ are i.i.d. random variables with mean $\mu$ and variance $\sigma$. Let $S_m = X_1+X_2+...+X_m$ and $\tau_t = \sup\{m : X_m \leq t\}$. Then, $$
		\frac{nt/\mu - \tau_{nt}}{\sqrt{n}} \Rightarrow \frac{\sigma}{\mu^{3/2}} B_t
		$$
		\textbf{weakly} in the Skorokhod topology on cadlag paths in $\mathbb R_+$, where $B_t$ is a standard Brownian motion.
	\end{theorem}
	
	We are now ready to prove some estimates which will be used to prove our main results.
	
	\subsection{Key estimates} We begin by establishing bounds on the solutions to equations \eqref{ode1} and \eqref{randomcontrol}.  Recall the functions $x_r$ and $x^n_r$ defined in \eqref{ode1} and \eqref{randomcontrol}. Our next lemma proves that both functions grow slower than a deterministic exponential function in $t$. This is particularly surprising for $x^n_r$, since it removes any stochastic involvement in the growth bound. 
	\begin{lemma} \label{ODEbd}
		Let $L(T) = \sup_{0 \leq r \leq T} |x_r|$  and $L'(T) = \sup_{0 \leq r \leq T} |x^n_r|.$ There exists a constant $C_{ABKT} > 0$ such that
		$$
		\max\{ L(T), L'(T) \} \leq C_{ABKT} \quad \text{for all } T >0.
		$$
	\end{lemma}

	\begin{proof}
		For $ x_t$, using \eqref{contrlUsoln}
		\begin{align*}
			|x_t| \leq |x_0| + \int_0^t |(A - BK)x_s|  ds \leq |x_0| + \int_0^t |A - BK|\cdot |x_s|  ds. 
		\end{align*}
		Taking the supremum over $ 0 \leq t \leq T$
		\begin{align*}
			L(T) \leq |x_0| + |A - BK| \int_0^T L(s)  ds
		\end{align*}
		By Gronwall's lemma
		\begin{align*}
			L(T) \leq  C_{ABKT}.
		\end{align*}
		For $x^n_t$, using \eqref{randomcontrol}
		\begin{align*}
			|x^n_t| \leq |x_0| + \int_0^t |Ax^n_s - BK x^n_{\pi^n(s)}|  ds \leq |x_0| + \int_0^t \left( |A|\cdot |x^n_s| + |B|\cdot |K|\cdot |x^n_{\pi^n(s)}| \right) ds.
		\end{align*}
		Since $ \pi^n(s) \leq s $, $ |x^n_{\pi^n(s)}| \leq L'(s) $. Thus
		\begin{align*}
			L'(T) \leq |x_0| + (|A| + |B|\cdot |K|) \int_0^T L'(s)  ds.
		\end{align*}
		By Gronwall's lemma,
		\begin{align*}
			L'(T) \leq |x_0| e^{(|A| + |B| \cdot |K|) T} \leq C_{ABKT}.
		\end{align*}
		Combining constants, the result follows.
	\end{proof}
	
	In order to bound the growth of the stochastic solution $X^{\epsilon,n}_t$ in \eqref{M2withnoiseforallt}, we include the following lemma which is a standard corollary of the BDG inequality.
	
	\begin{lemma}\cite{oksendal2003sde} \label{BDGnoisebd}
		For any $p\in [1,\infty)$, there exists a $ C > 0 $ such that for all $ T \in [0, T_0] $
		\begin{align*}
			\E \left[ \sup_{0 \leq t \leq T} \left| \int_0^t dW_s \right|^{2p} \right] \leq C {T^p} .
		\end{align*}
	\end{lemma}

	
	The next estimate bounds the tail of a matrix exponential sum. It will be used in bounds involving Taylor series expansions.
	\begin{lemma}\label{Expbd}For any $r\geq 0,$ we have
		$$
		\left|\sum_{k=3}^{\infty}\frac{r^kA^k}{k!}\right| \leq \frac{r^3|A|^3}{6}e^{r|A|}.
		$$
	\end{lemma}
	\begin{proof}
		Since the norm is submultiplicative,
		\begin{align*}
			\left|\sum_{k=3}^{\infty}\frac{r^kA^k}{k!}\right|  \leq \sum_{k=3}^{\infty} \frac{r^k|A|^k}{k!} \leq r^3|A^3| \sum_{k=0}^{\infty} \frac{r^{k}|A|^{k}}{k! \times (k+1)(k+2)(k+3)}  \leq  \frac{r^3|A|^3}{6} e^{r|A|}.
		\end{align*}
	\end{proof}
	
	We shall repeatedly use Wald's lemma, Lemma~\ref{Waldlemma}. Our next result establishes a stopping time for which it will be used.
	
	\begin{lemma}\label{stoppingtime}
		The random time $N^n_T+1$ is a stopping time with respect to $(\xi^n_{i+1})^p, i=0,1,2,\ldots,k$ and $p \geq 1$.  
	\end{lemma}
	\begin{proof}
		In order to prove this, fix $k \in \mathbb N$, and consider the event $\{N^n_T+1 \leq k\}$. This is equivalent to the event $\{N^n_T < k\}$, since $N^n_T$ is integer-valued. On the other hand, by \eqref{Nnt} and \eqref{taunk},
		\begin{align*}
			\{N^n_T+1 \leq k\}=\{N^n_T < k\} =  \{\tau^n_k > T\} = \left\{\sum_{i=0}^{k-1} \xi_{i+1}^n > T\right\} = \left\{\sum_{i=0}^{k-1} ((\xi_{i+1}^n)^p)^{\frac 1p} > T\right\}.
		\end{align*}
		Hence, $\{N^n_T+1 \leq k\}$ depends only on $\xi^n_1, \ldots, \xi^n_k$ which are the first $k$ terms of $(\xi^n_{i+1})^p$, $i=0,1,\ldots,k-1$. This makes it a stopping time.
	\end{proof}
	
	The next estimate, which is of independent interest as a general tool, states that we can estimate the expectations of supremums of i.i.d. sets over various index sets, by bounding them above by sums whose expectations can be taken using Lemma~\ref{Waldlemma}.
	
	\begin{proposition}\label{suptosum}
		Let $X_i, i \geq 0$ be i.i.d. non-negative  random variables, and $\tau$ be a stopping time with respect to $\{X_i\}_{i \geq 0}$. Suppose that $\E[X_i^q]<\infty$ for some $q>1$. Then, $$
		\E\left[\sup_{0 \leq t \leq \tau} X_i\right] \leq \E[\tau+1]^{\frac 1q}\E[X_i^q]^{\frac 1q}. 
		$$
	\end{proposition}
	\begin{proof}
		By Jensen's inequality, $$
		\E\left[\sup_{0 \leq t \leq \tau} X_i\right] \leq \E\left[\sup_{0 \leq t \leq \tau} X_i^q\right]^{\frac 1q} \leq \E\left[\sum_{t=0}^{\tau} X_i^q\right]^{\frac 1q},
		$$
		where we used the fact that $X_i$ are non-negative. 
		
		Note that $\{X_i^q\}_{i \geq 0}$ continues to be an i.i.d. sequence which generates the same filtration as $\{X_i\}_{i \geq 0}$. Thus, $\tau$ remains a stopping time with respect to this sequence, and the proof directly follows by Lemma~\ref{Waldlemma}.
	\end{proof}
	
	During our analysis, we shall repeatedly see a family of lower order noise terms generated by random sampling. We establish the lower order growth of these terms.
	\begin{lemma}\label{Nptilde}
		Let 
		\begin{align*}
			\tilde{\NN}_p:= \int_0^T (\xi_{{N_s^{n}+1}}^n)^p e^{({\xi_{{N_s^{n}+1}}^n})\tilde{C}} ds,
		\end{align*}
		where $p>0$ and $\tilde{C}$ is a constant depending upon matrices A,B, K, time T and p. Then, we have 
		$$\E[\tilde{\NN}_p]\leq \frac{1}{n^p}\frac{\E[N_T^n]+2}{n}\E[\xi_1^{p+1} e^{\xi^n_1 \tilde{C}}].$$
	\end{lemma}
	\begin{proof}   
		We have the fact that if $\xi_i$ are i.i.d., then for any measurable function $f$, $f(\xi_i)$ are also i.i.d. Thus, we have
		\begin{align*}
			\int_0^T (\xi_{{N_s^{n}+1}}^n)^p e^{{\xi_{{N_s^{n}+1}}^n \tilde{C}}} ds
			&= \sum_{k=0}^{N_T^n-1} \int_{\tau_k^n}^{\tau_{k+1}^n}(\xi_{k+1}^n)^p e^{\xi_{k+1}^n \tilde{C}} ds+\int_{{\tau_{N_T^n}^n}}^T(\xi_{N_s^n+1}^n)^p e^{\xi_{N_s^n+1}^n \tilde{C}} ds\\
			&= \sum_{k=0}^{N_T^n-1} (\xi_{k+1}^n)^p e^{\xi_{{k+1}}^n \tilde{C}} \int_{\tau_k^n}^{\tau_{k+1}^n} ds+(\xi_{N_T^n+1}^n)^p e^{\xi_{{N_T^n+1}}^n \tilde{C}} \int_{{\tau_{{N_T^n}}^n}}^T ds
			\\& =\sum_{k=0}^{N_T^n-1} (\xi_{k+1}^n)^p e^{\xi_{{k+1}}^n \tilde{C}} (\tau_{k+1}^n-\tau_k^n)+(\xi_{N_T^n+1}^n)^p e^{\xi_{{N_T^n+1}}^n \tilde{C}} (T-\tau_{{N_T^n}}^n) \\
			&\leq \sum_{k=0}^{N_T^n} (\xi_{k+1}^n)^{p+1} e^{\xi_{{k+1}}^n \tilde{C}}\leq  \sum_{k=0}^{N_T^n+1} (\xi_{k+1}^n)^{p+1} e^{\xi^n_{{k+1}} \tilde{C}}.
		\end{align*}
		Taking expectation, and using Lemma~\ref{Waldlemma}, we get the required result. 
	\end{proof}
	
	We remark that as $n \to \infty$, by Theorem~\ref{thm:renewal} the right hand side is asymptotically equivalent to a decay rate of $\frac 1{n^p}\frac{\E[\xi_1^{p+1}]}{\E[\xi_1]}$.  Taking $\tilde{C} = 0$ above we have following  corollary.
	
	\begin{corollary}\label{Np}
		For any $p>0$, define $$
		\NN_p := \int_0^T (s - \pi^n(s))^p ds. 
		$$
		Then, we have
		\begin{equation*}
			\E[\NN_p]\leq \frac{1}{n^p}  \frac{\E[N^n_T+2]}{n} \E[\xi_1]^{p+1},
		\end{equation*}       
	\end{corollary} 
	which also decays at the rate $\frac{1}{n^p}$ same way as in Lemma \ref{Nptilde}. Observe that, in this case, we have $(s - \pi^n(s)) \leq \xi_{N^n_s + 1}^n.$

	The upgrade from distributional convergence to convergence of moments (for instance, in Theorem~\ref{donsker}) requires uniform integrability. The next two lemmas are necessary to achieve this.
	\begin{lemma}\label{unifint1}
		Let $\MM_k :=\sum_{i=1}^k (\xi_i - \E[\xi_1])= \tau^1_k - \E[\xi_1]k,~ k=1,2,\ldots$ be the discrete time random walk associated to the increments $\xi_i - \E[\xi_1]$. Then, 
		\begin{enumerate}
			\item [(a)] For any stopping time $\tau'$ with respect to $\MM_k$, $\E[\sup_{k \leq \tau'} |\MM_k|^6] \leq C_3\E[\tau'^4]^{\frac 12}\E[\tau'^2]^{\frac 12}$ where $C_3$ depends only upon $\xi_1$.
			\item [(b)] For all even numbers $p$, we have $\E[|\MM_k|^p] \leq C_{p,\xi_1} k^{p/2}$ for all $k>0$ and $C_{p,\xi_1}$ depending only upon $p$ and $\xi_1$.
		\end{enumerate}
	\end{lemma}
	\begin{proof}
		We will use the discrete time BDG inequality (see \cite[page 23]{hall1980martingale} or \cite[page 2]{wood1999rosenthal}). Observe that $\MM_k$ is a sum of independent mean zero random variables, and hence a martingale with respect to its own filtration. Therefore, applying the BDG inequality with $p=6$ for any stopping time $\tau'$, \begin{equation}\label{eq:bdg}
			\E[\sup_{k \leq \tau'} |\MM_k|^6]\leq C_3 \E[\langle \MM_k,\MM_k\rangle_{\tau'}^3]
		\end{equation}
		for a constant $C_3$ independent of $\xi_1$ and $\tau'$, where $\langle \MM_k,\MM_k\rangle_{t}$ is the quadratic variation of $\MM_k$.  Let $t_i=\MM_i-\MM_{i-1} = (\xi_i - \E[\xi_i])$ be the martingale differences. Then, the quadratic variation in discrete time (see \cite[page 23]{hall1980martingale}) is defined as 
		\begin{equation*}
			\E[\langle \MM_k,\MM_k\rangle_{k}] = \E\left(\sum_{i \leq k} t_i^2\right).
		\end{equation*}
		Therefore, we have
		\begin{equation}\label{eq:qv}
			\E\left[\langle \MM_k,\MM_k\rangle_{\tau'}^3\right] = \E\left[\left(\sum_{k \leq \tau'} (\xi_k - \E[\xi_k])^2\right)^3\right].
		\end{equation}
		We will now expand and bound this expectation. Note that by Holder's inequality,
		\begin{equation*}
			\left(\sum_{k \leq \tau'} (\xi_k - \E[\xi_k])^2\right)^3 \leq \tau'^2 \times \sum_{k \leq \tau'} (\xi_k - \E[\xi_k])^6.
		\end{equation*}
		Therefore, we have \begin{align}
			\E\left[\left(\sum_{k \leq \tau'} (\xi_k - \E[\xi_k])^2\right)^3\right]
			\leq & ~\E\left[\tau'^2 \times \sum_{k \leq \tau'} (\xi_k - \E[\xi_k])^6\right] \nonumber \\
			\leq & ~\E[\tau'^4]^{\frac 12} \times \E\left[\left(\sum_{k \leq \tau'} (\xi_k - \E[\xi_k])^6\right)^2\right]^{\frac 12},\label{next}
		\end{align}
		by C-S inequality. We are only left to bound the final term above. Let
		$$
		a = \sum_{k \leq \tau'} (\xi_k - \E[\xi_k])^6 \quad b = \sum_{k \leq \tau'} \E(\xi_k - \E[\xi_k])^6 = \tau' \times \E(\xi_k - \E[\xi_k])^6.
		$$
		Now, observe that $a^2 \leq 4b^2+4(a-b)^2$, and that $a-b = \sum_{k\leq \tau'} [(\xi_k - \E[\xi_k])^6  - \E(\xi_k - \E[\xi_k])^6]$ is a stopped sum of i.i.d. zero mean random variables, on which Wald's identity \ref{Waldlemma} can be applied. Doing so gives
		\begin{align*}
			\E\left[\left(\sum_{k \leq \tau'} (\xi_k - \E[\xi_k])^6\right)^2\right] &\leq4 \E[\tau'^2] [\E[(\xi_k - \E[\xi_k])^6]]^2 + 4 \E[\tau'] \mbox{Var}[(\xi_k - \E[\xi_k])^6]\\
			&\leq C_{\xi_1}(\E[\tau'^2]+\E[\tau'])\le \E[\tau'^2] C_{\xi_1}.
		\end{align*}
		for a constant $C_{\xi_1}$ depending only upon $\xi_1$, where we used the fact that $E[\tau']\leq \E[\tau'^2]$ since $\tau'$ is non-negative integer valued. Combining this with \eqref{eq:bdg}, \eqref{eq:qv} and \eqref{next} finishes the proof.
		
		For the other part, we first remark that for $p=2$, the identity is immediate. For $p>2$ even, recall that $\MM_k$ is a martingale. Applying the discrete BDG inequality as we did in \eqref{eq:bdg},
		\begin{equation}\label{eq:bdgDiscrete}
			\E[|\MM_k|^p]  \leq C_{p} \E[\langle \MM_k,\MM_k\rangle_{k}^{p/2}].
		\end{equation}
		To bound the right hand side, we notice that by Holder's inequality,
		\begin{equation*}
			\E[\langle \MM_k,\MM_k\rangle_{k}^{p/2}] = \E\left[\left(\sum_{i \leq k} t_i^2\right)^{p/2}\right]\leq \E\left[k^{p/2-1}\sum_{i=1}^{k} t_i^p\right]\leq k^{\frac{p}{2}}C_{p\xi_1},
		\end{equation*}
		where we used the linearity of expectation in the last line. Combining this with \eqref{eq:bdgDiscrete} completes the proof.
	\end{proof}
	The preceding result on moments has an immediate application in the next Lemma.
	\begin{lemma}\label{unifint2}
		For any even number $p$, we have $\E[|N^1_t|^p] \leq C_pt^p$ for every $t>0$ and a constant $C_{p,\xi_1}$ depending only on $p$ and $\xi_1$.
	\end{lemma}
	
	\begin{proof}
		We use tail probability representation of moments as
		\begin{align}
			\E[|N^1_{t}|^p] = &p\int_0^{\infty}x^{p-1}\mathbb P(N^1_t \geq x)dx \nonumber\\=& p \sum_{k=0}^{\infty}\int_{k}^{k+1} x^{p-1}\mathbb P(N^1_t \geq k+1)dx \nonumber\\=& \sum_{k=0}^{\infty} \mathbb P(N^1_t \geq k+1) ((k+1)^p - k^p) \nonumber\\
			\leq & C_p\sum_{k=0}^{\infty} \mathbb P(N^1_t \geq k+1)k^{p-1}= C_p\sum_{k=0}^{\infty} \mathbb P(\tau^1_{k+1} \leq t)k^{p-1}.\label{start} 
		\end{align}
		It now suffices to bound $\mathbb P(\tau^1_{k+1}\leq t)$. We do so as follows : observe that if $t < (k+1)\E[\xi_1]$ and $p$ is any natural number, then by Markov's inequality,
		\begin{align*}\mathbb P(\tau^1_{k+1} \leq t)  =& \mathbb P(\tau^1_{k+1} - (k+1)\E[\xi_1]  \leq t - (k+1)\E[\xi_1])\\ \leq & \mathbb P(|\tau^1_{k+1} - (k+1)\E[\xi_1]| \geq (k+1)\E[\xi_1]-t) \\
			\leq &\frac{\E(|\tau^1_{k+1} - (k+1)\E[\xi_1]|^{2p+2}}{((k+1)\E[\xi_1] - t)^{2p+2}} = \frac{\E(|\MM_{k+1}|^{2p+2})}{((k+1)\E[\xi_1] - t)^{2p+2}},
		\end{align*}
		where $\MM_{k} = \tau^1_{k} - k \E[\xi_1]$. By Lemma~\ref{unifint1} we have $$
		\mathbb P(\tau^1_{k+1} \leq t) \leq C\frac{k^{p+1}}{|(k+1)\E[\xi_1] - t|^{2p+2}}
		$$
		for all $k$ such that $(k+1)\E[\xi_1]>t$. In particular, if $(k+1)>\frac{2t}{\E[\xi_1]}$, we have $((k+1)\E[\xi_1] - t)\leq \frac 12 (k+1)\E[\xi_1]$ and therefore $$
		\mathbb P(\tau^1_{k+1} \leq t) \leq C_{p,\xi_1}\frac{k^{p+1}}{(k+1)^{2p+2}} \leq C_{p,\xi_1} k^{-p-1}.
		$$
		We apply the following to \eqref{start} as follows : if $t\geq (k+1)\E[\xi_1]$, then we may bound $\mathbb P(\tau^1_{k+1} \leq t) \leq 1$. Therefore, 
		\begin{align*}
			\E[|N^1_{t}|^p]  &\leq C_p \sum_{k=0}^{\lceil t/\E[\xi_1]-1 \rceil} k^{p-1} +  C_{p,\xi_1}\sum_{\lceil t/\E[\xi_1]-1 \rceil + 1}^{\infty} k^{-p-1+p-1} \\&\leq C_p\sum_{k=0}^{\lceil t/\E[\xi_1]-1 \rceil} k^{p-1} + C_{p,\xi_1} \sum_{k=0}^{\infty} k^{-2} < C_{p,\xi_1}t^p,   
		\end{align*}
		where $C_{p,\xi_1}$ is independent of $t$. This completes the proof.
	\end{proof}
	
	We are now ready to prove Theorem~\ref{LLN}.
	\section{Proof of LLN type Results}\label{section4}
	In this section,  we prove Theorem~\ref{LLN}. For this purpose, we first establish following two supporting lemmas. The theorem will subsequently be derived as a simple consequence.
	\begin{lemma}\label{Assympwithoutnoise}
		Let $ \mathtt{Y}_T = \sup_{0 \leq s \leq T} |x^n_s - x_s|^p $. Then for any $p\geq 1,$ there exists $ C_{ABKTp} > 0, $ depending on $T,p$ and the matrices $A,B$ and $K$ such that
		\begin{align*}
			\mathtt{Y}_T\leq C_{ABKTp}~ {\NN}_p,
		\end{align*}
		where $ {\NN}_p = \int_0^T (s - \pi^n(s))^p ds $ is defined as in Corollary~\ref{Np} and converges to zero.
	\end{lemma}
	
	\begin{proof}
		Let $t\in [0,T]$ be arbitrary. Subtract \eqref{ode1} from \eqref{randomcontrol} and integrating over $[0,t]$, we get
		\begin{align*}
			x^n_t - x_t = \int_0^t \left( A(x^n_s - x_s) + BK (x_s - x^n_{\pi^n(s)}) \right) ds.
		\end{align*}
		Taking the $p$th power of the absolute value of both sides, and applying first the triangle inequality and then the C-S inequality, we get
		\begin{align}\label{second}
			|x^n_t - x_t|^p \leq C_{Tp}\left[ \int_0^t |A|^p |x^n_s - x_s|^p ds + \int_0^t |B|^p \cdot|K|^p |x_s - x^n_{\pi^n(s)}|^p ds\right].
		\end{align}
		For the second term
		\begin{align*}
			|x_s - x^n_{\pi^n(s)}|^p \leq C_{Tp}\left[ |x_s - x^n_s|^p + |x^n_s - x^n_{\pi^n(s)}|^p\right] \leq C_{Tp}(Y_s + |x^n_s - x^n_{\pi^n(s)}|^p).
		\end{align*}
		Using the \eqref{randomcontrol}, we have
		\begin{align*}
			x^n_s - x^n_{\pi^n(s)} = \int_{\pi^n(s)}^s \left( Ax^n_r - BK x^n_{\pi^n(r)} \right) dr.
		\end{align*}
		Thus
		\begin{align*}
			|x^n_s - x^n_{\pi^n(s)}| \leq \int_{\pi^n(s)}^s \left( |A| |x^n_r| + |B|\cdot |K| |x^n_{\pi^n(r)}| \right) dr.
		\end{align*}
		Since $ \pi^n(r) \leq r \leq s $, by Lemma \ref{ODEbd}, $ |x^n_r| \leq L'(s) \leq C_{ABKT},~ |x^n_{\pi^n(r)}| \leq L'(s)\le C_{ABKT} $. So
		\begin{align*}
			|x^n_s - x^n_{\pi^n(s)}|\le C_{ABKT}\int_{\pi^n(s)}^s ds =  (s - \pi^n(s))C_{ABKT}.
		\end{align*}
		The second integral in \eqref{second} becomes
		\begin{align*}
			\int_0^t |B|^p\cdot |K|^p |x_s - x^n_{\pi^n(s)}|^p ds \leq C_{Tp}\left[\int_0^t |B|^p\cdot |K|^p Y_s ds + \int_0^t |B|^p \cdot|K|^p (s - \pi^n(s))^p  ds\right].
		\end{align*}
		Plugging the above estimate in equation \eqref{second} and taking superimum over $t\in[0,T]$ on both sides, we get
		\begin{align*}
			\mathtt{Y}_T \leq C_{ABKTp}\left(\int_0^T Y_s ds + \int_0^T (s - \pi^n(s))^p  ds\right),
		\end{align*}
		Finally, by Gronwall's lemma, we have
		\begin{align*}
			\mathtt{Y}_T \leq C_{ABKTp}~{\NN}_p.
		\end{align*}
		
	\end{proof}
	
	\begin{lemma}\label{Assympwithnoise} 
		For any $p\geq 1,$ there exists $ C_{ABKTp} > 0$, such that
		\begin{align*}
			\E \left[ \sup_{0 \leq t \leq T} |X_t^{\e, n} - x^n_t|^p \right] \leq C_{ABKTp}~\e^p.
		\end{align*}
	\end{lemma}
	
	\begin{proof}
		Subtract \eqref{randomcontrol} from \eqref{controlwithnoise}
		\begin{align*}
			X_t^{\e, n} - x^n_t = \int_0^t A (X_s^{\e, n} - x^n_s) ds - \int_0^t BK \left( X^{\e, n}_{\pi^n(s)} - x^n_{\pi^n(s)} \right) ds + \e \int_0^t dW_s.
		\end{align*}
		By the triangle inequality
		\begin{align*}
			|X_t^{\e, n} - x^n_t| \leq \int_0^t |A| |X_s^{\e, n} - x^n_s| ds + \int_0^t |B|\cdot |K| |X^{\e, n}_{\pi^n(s)} - x^n_{\pi^n(s)}| ds + \e \left| \int_0^t dW_s \right|.
		\end{align*}
		Define $ \mathtt{S}_t = \sup_{0 \leq s \leq t} |X^{\e, n}_s - x^n_s| $. Since $ \pi^n(s) \leq s $, $ |X^{\e, n}_{\pi^n(s)} - x^n_{\pi^n(s)}| \leq \mathtt{S}_s $. Thus
		\begin{align*}
			\mathtt{S}_T \leq \int_0^T (|A| + |B| \cdot|K|) \mathtt{S}_s ds + \e \sup_{0 \leq t \leq T} \left| \int_0^t dW_s \right|.
		\end{align*}
		Taking pth power on both the sides and using H\"{o}lder's inequality, we get
		
		\begin{align*}
			\mathtt{S}_T^p \leq C_{pT} \int_0^T (|A| + |B| \cdot|K|)^p \mathtt{S}_s^p ds + \e^p\sup_{0 \leq t \leq T} \left| \int_0^t dW_s \right|^p.
		\end{align*}
		Taking expectation and applying Lemma \ref{BDGnoisebd}, we have
		\begin{align*}
			\E[\mathtt{S}_T^p] \leq C_{ABKTp} \left(\int_0^T \E[\mathtt{S}_s^p] ds +\e^p\right).
		\end{align*} 
		By Gronwall's lemma, we obtain
		\begin{align*}
			\E[\mathtt{S}_T^p]\leq C_{ABKTp}~\e^p, 
		\end{align*}
		i.e., 
		\begin{align*}\E \left[ \sup_{0 \leq t \leq T} |X_t^{\e, n} - x^n_t|^p \right]\leq C_{ABKTp}~\e^p.\end{align*}	
	\end{proof}
	Now, we are ready to prove Theorem \ref{LLN} with the help of above two Lemmas \ref{Assympwithoutnoise} and \ref{Assympwithnoise}.
	\begin{proof}[{Proof of Theorem \ref{LLN}}]
		By the triangle inequality
		\begin{align*}
			\E \left[ \sup_{0 \leq s \leq T} |X_s^{\e, n} - x_s|^p \right] \leq C_p\left[ \E \left[ \sup_{0 \leq s \leq T} |X_s^{\e, n} - x^n_s|^p \right] + \E \left[ \sup_{0 \leq s \leq T} |x^n_s - x_s|^p \right]\right].
		\end{align*}
		From Lemma \ref{Assympwithoutnoise} and  \ref{Assympwithnoise}, we obtain	
		\begin{align*}
			\E \left[ \sup_{0 \leq s \leq T} |X_s^{\e, n} - x_s|^p \right] \leq C_{ABKTp} {\E[\NN}_p] + C_{ABKTp}~\e^p \leq C_{ABKTp}  \left(\E[\NN_p]+\e^p\right),
		\end{align*}
		which is desired result.		
	\end{proof}	
	
	We now turn our attention to the proof of second principal result of this paper, which is formulated and rigorously stated in Theorem~\ref{CLTresult}. The content of the following section is devoted entirely to this proof.
	\section{Proof of the CLT Theorem \ref{CLTresult}}\label{CLTSection}\label{section5} The main focus of this section is to prove Theorem~\ref{CLTresult} with the help of various auxiliary lemmas. Recall that by \eqref{eq:mainterm}, the key term to analyze is the rescaled fluctuation process $\displaystyle\int_0^t \frac{{X_s^{\e, n}-X_{\pi^n(s)}^{\e, \delta}}}{\e} ds.$  We first proceed by simplifying this expression, decomposing it into a sampling component and a component affected by white noise. To  this end, let $\M=\{\M_t: 0 \le t < \infty\}$ be the process defined by
	\begin{equation}\label{noiseint}
		\M_t := \int_0^t e^{-sA} dW_s =e^{-tA}W_t+\int_0^t e^{-sA}AW_s  ds.
	\end{equation}
	Note that $\M_t$ is a $\{\mathcal F_t\}$-martingale since it is a stochastic integral.
	\begin{lemma}\label{L1L2decomposition}
		For $\e,1/n \in (0,1)$, $t\in [0,T]$, we have 
		\begin{align}\label{FluctuationDecomposition}
			\frac{1}{\e}\int_0^t ({X_s^{\e, n}-X_{\pi^n(s)}^{\e, n}}) ds :=  \LL^{\e,n}_1(t)+ \LL^{\e,n}_2(t),
		\end{align}
		where
		\begin{equation*}
			\LL^{\e,n}_1(t) = \frac{1}{\e}\int_0^t \left[{e^{{(s-\pi^n(s))}A}-I}\right] \left[I-A^{-1}BK\right]X_{\pi^n(s)}^{\e,n} ds,
		\end{equation*}
		and\begin{equation*}
			~\LL^{\e,n}_2(t) = \int_0^t e^{sA}\left(\M_s-\M_{\pi^n(s)}\right)  ds.
		\end{equation*}		
	\end{lemma}
	
	\begin{proof}
		Given that $A$ is invertible, we have for any $a<b$
		$$ \int_a^b e^{-sA} ds =-\left( e^{-bA}-e^{-aA}\right)A^{-1},$$
		which gives 
		\begin{align*}
			\int_{\tau_k}^te^{(t-s)A}BK ds&
			=-e^{tA}\left( e^{-tA}-e^{-\tau_k A}\right)A^{-1}BK= \big(e^{(t- \tau_k)A}-I\big)A^{-1}BK.
		\end{align*}
		Therefore,
		\begin{align*}
			e^{(t- \tau_k)A}- \int_{\tau_k}^te^{(t-s)A}BK ds-I=\big(e^{(t- \tau_k)A}-I\big)\big(I-A^{-1}BK\big).
		\end{align*}
		Let $t\in [\tau_k^n, \tau_{k+1}^n)$ for some $k \in \mathbb N$. In this interval, $X_{\pi^n(s)}^{\e,n}=X_{\tau_k^n}^{\e,n}$ is fixed.  From equation \eqref{LSDEwRS}, we have 
		\begin{align*}
			X_t^{\e,n}=X_{\tau_k^n}^{\e,n}e^{(t-{\tau_k^n})A}-\int_{{\tau_k^n}}^t e^{(t-s)A}BKX_{\tau_k^n}^{\e,n}ds+\e\int_{{\tau_k^n}}^te^{(t-s)A}dW_s.
		\end{align*}
		Subtracting $X_{\tau_k^n}^{\e,n},$ on both sides, we get
		\begin{align}\label{E1}
			X_t^{\e,n}-X_{\tau_k^n}^{\e,n} &=  \left[e^{(t-{\tau_k^n})A}-\int_{{\tau_k^n}}^t e^{(t-s)A}BKds -I\right]X_{\tau_k^n}^{\e,n}+\e e^{tA}\int_{{\tau_k^n}}^te^{-sA}dW_s \nonumber \\
			&=\left[\left(e^{(t-{\tau_k^n})A}-I\right)\left(I-A^{-1}BK\right)\right]X_{\tau_k^n}^{\e,n}+\e e^{tA}(\M_t-\M_{\tau_k^n}).
		\end{align}
		Now, for any $t\geq 0,$  we can write
		\begin{align*}
			X_t^{\e,n}-X_{\tau_k^n(t)}^{\e,n}=\sum_{k \geq 0} \ind_{[\tau_k^n,\tau_{k+1}^n)} (t) (X_t^{\e,n}-X_{\tau_k^n(t)}^{\e,n}),~~~~
		\end{align*}
		and so
		\begin{align*}
			\int_{0}^{t}X_s^{\e,n}-X^{\e,n}_{\pi^n(s)}ds&=\sum_{k\geq 0 }\int_0^t \ind_{[\tau_k^n,\tau_{k+1}^n)} (s) (X_s^{\e,n}-X_{\tau_k^n(s)}^{\e,n})ds.
		\end{align*}
		Since $N^n_T < \infty$ almost surely for all $n$ by \eqref{nnt} and Theorem~\ref{thm:renewal}, the sum above is finite a.s. Therefore it can be exchanged with the integral, leading to
		\begin{align*}
			\sum_{k\geq 0 }\int_0^t \ind_{[\tau_k^n,\tau_{k+1}^n)} (s)&\left[\left(e^{(t-{\tau_k^n})A}-I\right)\left(I-A^{-1}BK\right)\right]X_{\tau_k^n}^{\e,n}+\e e^{tA}(\M_t-\M_{\tau_k^n})ds\\
			& \hspace{-1cm}=\int_0^t\left[\left(e^{(t-{\tau_k^n})A}-I\right)\left(I-A^{-1}BK\right)\right]X_{\pi^n(s)}^{\e,n}+\e e^{tA}(\M_t-\M_{\pi^n(s)})ds,
		\end{align*}
		by applying \eqref{E1} and thus proof is complete.
	\end{proof}	
	Now, we provide two key Propositions on $\LL^{\e,n}_1(t)$ and $\LL^{\e,n}_2(t)$.  By combining the results obtained for $\LL_1^{\e,n}$ and $\LL_2^{\e,n}$, we will then be able to establish Theorem \ref{CLTresult}.
	
	\begin{proposition}\label{L1asymptotes}
		There exists  $\e_0$ such that for and $0<\e<\e_0$, we have 
		\begin{align*}
			\E\left[\sup_{0 \le t \le T} |\LL^{\e,n}_1(t)-\ell(t)|\right] 
			\leq 	 \left[ \cc(n^{-1/2}(n^{-1}+\e)+n^{-1/4}) + M \varkappa(\e)\right]C_{ABKT\xi_1}
		\end{align*}
		where the function $\ell(t)$ is given by Definition \eqref{ellDefn}.
	\end{proposition}
	
	\begin{proposition}
		\label{noisebd}
		For  $0<\e<\e_0$, the term  $\LL^{\e,n}_2$ decays with rate $n^{-1/2}$, i.e., we have 
		\begin{align*}
			\E\left[\sup_{0 \le t \le T} |\LL^{\e,n}_2(t)|\right] \leq \frac{1}{\sqrt{n}}C_{AT\xi_1}.
		\end{align*} 	    
	\end{proposition}
	
	\begin{proof}[Proof of Theorem \ref{CLTresult}]
		The Theorem can be proved by combining the Propositions \ref{L1asymptotes} and \ref{noisebd}.
	\end{proof}

	In the following subsections \ref{subsection5.1} and \ref{subsection5.2}, we will focus on establishing Propositions \ref{L1asymptotes} and \ref{noisebd} respectively. We begin in the following subsection with the term $\LL_1^{\e,n}$, which represents the random sampling part and leads to Proposition \ref{L1asymptotes}. Since its limiting behavior cannot be derived in a straightforward manner, we decompose $\LL_1^{\e,n}$ into several auxiliary terms. We then demonstrate that this sequence converges to the function $\ell$. In the subsequent subsection, we turn to $\LL_2^{\e,n}$ and establish that it vanishes in the limit and so establish Proposition \ref{noisebd}. Accordingly, we  now proceed with a detailed analysis of the random sampling component.
	
	\subsection{The Decomposition of Random Sampling Part: Proof of Proposition \ref{L1asymptotes}}\label{subsection5.1}
	
	We decompose $\LL^{\e,n}_1$  in the following way:
	\begin{align}
		\LL^{\e,n}_1=&\frac{1}{\e n} \int_0^t \left(\frac{e^{s - \pi^n(s) A} - I}{1/n}\right) (I-A^{-1}BK) X^{\e,n}_{\pi^n(s)} ds\nonumber \\\label{G1}
		=& \frac{1}{\e n}\int_0^t \left(\frac{e^{s - \pi^n(s) A} - I}{1/n}\right) (I-A^{-1}BK) \left(X^{\e,n}_{\pi^n(s)}  - x_{\pi^n(s)}\right)ds \\\label{G2}
		+& \frac{1}{\e n}\int_0^t \left(\frac{e^{s - \pi^n(s) A} - I - (s - \pi^n(s)) A}{1/n}\right) (I-A^{-1}BK)x_{\pi^n(s)}ds \\\label{G3}
		+& \frac{1}{\e n}\int_0^t \left(\frac{(s-\pi^n(s))A}{1/n}(I-A^{-1}BK)(x_{\pi^n(s)} - x(s)) \right) ds   \\\label{G4}
		+& \frac{1}{\e n}\int_0^t \left(\frac{(s-\pi^n(s))A}{1/n} - MA\right)(I-A^{-1}BK) x(s)ds \\\label{G5}
		+& \left(\frac{1}{\e n} - \cc\right) \int_0^t M(A-BK) x(s)ds \\\label{limit}
		+& \cc\int_0^t M (A-BK)x(s)ds\nonumber\\
		=:&\sum_{i=1}^{5}G_i+ \cc\int_0^t M (A-BK)x(s)ds=\sum_{i=1}^{5}G_i+\ell(t). \nonumber 
	\end{align}
	That is,
	\begin{equation}\label{L1limit}
		\LL^{\e,n}_1(t)-\ell(t)= \sum_{i=1}^{5}G_i,  
	\end{equation}
	where the functions $G_i, i=1,2,\ldots, 5$ are defined by expressions \eqref{G1}--\eqref{G5}, respectively. We now proceed to evaluate each of the quantities 
	$G_i$ individually and, through a careful and detailed analysis, derive the following results.
	\begin{lemma}\label{G1Lemma} The term $G_1$ decays at the rate $n^{-\frac 12}(n^{-1}+\epsilon)$, i.e.,
		\begin{align*}
			\E\left[\sup_{0\leq t\leq T} \left|G_1(t)\right|\right]\leq\frac 1{\e} 
			C_{ABKT} \left(\sqrt{\E[\NN_2] + \e^2 }\right) \sqrt{\E\tilde{ \NN}_3}\leq \cc~ n^{-1/2}(n^{-1}+\e)C_{ABKT\xi_1}.
		\end{align*}
	\end{lemma}

	\begin{lemma}\label{G2estimates}
		The term $G_2(t)$ decays at the rate $n^{-1}$, i.e.,
		\begin{align*}
			\E\left[\sup_{0 \leq t \leq T} \left|G_2(t)\right|\right]\leq \frac n\e C_{ABKT}\E[\tilde{\NN}_3] \leq \cc n^{-1}C_{ABKT\xi_1}.
		\end{align*}
	\end{lemma}
	
	\begin{lemma}\label{G3Lemma} The term $G_3$ decays at the rate $n^{-1}$, i.e.,
		\begin{align*}
			\E [\sup_{0\leq t\leq T}|G_3(t)|]\leq \frac{C_{ABKT}}{\e} \E[{\NN}_2]\leq \cc n^{-1}C_{ABKT\xi_1}.
		\end{align*}
	\end{lemma}
	\begin{lemma}\label{G4Lemma}
		The term $G_4$ decays at the rate $n^{-\frac 14}$, i.e.,
		\begin{align*}
			\E \left[\sup_{0\leq t\leq T}\left|G_4(t)\right|\right]\leq \cc  n^{-1/4}C_{ABKT\xi_1 }.
		\end{align*} 
	\end{lemma}
	
	\begin{lemma}\label{G5estimates} The term 
		$G_5$ decays at the rate $\varkappa(\epsilon)$. That is, there exists a $\e_0>0$ such that for any $0<\e<\e_0$, we have 
		\begin{align*}
			\sup_{0\leq t\leq T} |G_5(t)| \leq M\varkappa(\e)C_{ABKT}.
		\end{align*}
	\end{lemma}

	\begin{proof}[Proof of Proposition \ref{L1asymptotes}] By combining the results obtained in Lemmas \ref{G1Lemma}--\ref{G5estimates}, we get the required result.
	\end{proof}    
	Of the aforementioned lemmas, Lemmas~\ref{G1Lemma}--\ref{G3Lemma}  are established by suitably adapting the analytical framework developed in \cite{dhama2020approximation}. Their proofs require careful modifications to accommodate the presence of the random sampling terms.  This extension is feasible since the quantity $M$ associated with the random sampling mechanism remains unaffected in these estimates. In contrast, the proof of Lemma~\ref{G4Lemma} is considerably more involved and is carried out through a sequence of carefully structured steps. The proof of \ref{G5estimates} follows directly from \eqref{kappa} and \eqref{contrlUsoln}.
	
	We now proceed to prove Lemmas \ref{G1Lemma}–\ref{G5estimates}. 
	Each of these lemmas provides the necessary estimate for one of the quantities $G_i$. We therefore treat these terms separately and derive the required estimates step by step.  For clarity, the proofs are organized into the distinct subsubsections, with each subsubsection devoted exclusively to the analysis of a single  $G_i$. The proof of Lemma \ref{G4Lemma}, however, is more technically involved and will therefore be proved after completing the proof of Lemmas \ref{G5estimates}.

	\subsubsection{The \texorpdfstring{$G_1$}{G1} Term}\label{G1sub}
	We now prove Lemma~\ref{G1Lemma}. In a nutshell, the key observation is that $(X^{\epsilon,n}_{\pi^n(s)} - x_{\pi^n(s)})$ decays fast enough by Theorem~\ref{LLN}, for the term to converge to $0$.
	
	\begin{proof}[\textbf{The proof of Lemma \ref{G1Lemma}}:~]
		Let  us define $$f^n(s):=\left(\frac{e^{s - \pi^n(s) A} - I}{1/n}\right).$$
		We have $s-\pi^n(s)=s-\tau_{N_s^n}^n$. Using the fact that  $\tau_{N_s^{n}}^n\leq s\leq \tau_{N_s^{n}+1}^n,$  we have $s-\tau_{N_s^{n}}^n\leq  \tau_{N_s^{n}+1}^n-\tau_{N_s^{n}}^n=\xi_{N_s^{n}+1}^n$, and therefore, we get that 
		\begin{align}\label{eq:fn}
			|f^n(s)|=n\left|\int_0^{s-\pi^n(s)}Ae^{rA} dr\right|\leq n |A|\int_0^{\xi_{N_s^{n}+1}^n}|e^{rA}|dr\leq n \xi_{N_s^{n}+1}^n |A|e^{{\xi_{N_s^{n}+1}^n |A|}}. 
		\end{align}
		Consider the term $G_1$  defined in equation \eqref{G1} and let $z_s = X^{\e,n}_{\pi^n(s)}  - x_{\pi^n(s)}$. Then, $G_1$ becomes
		$$
		G_1(t) = \frac 1{\e n} \int_0^t f^n(s) (I-A^{-1}BK)z_s ds.
		$$
		Taking the supremum over $t\in[0,T]$ of the absolute value of $G_1$, and then taking expectation, we obtain
		\begin{align*}
			\E\Big[\sup_{0 \leq t \leq T}|G_1(t)|\Big] &\leq \frac 1{\e n}\E\left[\sup_{0 \leq t \leq T}\left|\int_0^t f^n(s) (I-A^{-1}BK) z_s ds\right|\right] \\
			& \leq \frac 1{\e n} \E\left[\sup_{0 \leq t \leq T} \int_0^t |f^n(s)| \cdot |I-A^{-1}BK| \cdot |z_s| ds\right] \\
			& \leq \frac 1{\e n} \E\left[\sup_{0 \leq t \leq T} \left(|I-A^{-1}BK|\sup_{0 \leq r \leq t} |z_r| \int_0^t |f^n(s)| ds\right)\right]\\
			& \leq \frac 1{\e} |I-A^{-1}BK|\cdot|A| \E\left[\left(\sup_{0 \leq t \leq T} |z_t|\right) \int_0^T |f^n(s)| ds\right] \\
			& \leq \frac 1{\e} |I-A^{-1}BK|\cdot|A| \left[\E\left(\sup_{0 \leq t \leq T} |z_t|\right)^2\right]^{\frac 12}\left[\E\left(\int_0^T \xi^n_{N_s^{n}+1}e^{\xi^n_{N_s^{n}+1}|A|} ds\right)^2\right]^{\frac 12}
		\end{align*}
		by using  C-S inequality and \eqref{eq:fn}. Let us focus on second term of the last expression, applying C-S, we get 
		\begin{align*}
			\left(\int_0^T \xi_{N_s^{n}+1}^n e^{{\xi_{N_s^{n}+1}^n |A|}} ds\right)^2 \leq T\int_0^T \left(\xi_{N_s^{n}+1}^n\right)^2 e^{{2\xi_{N_s^{n}+1}^n |A|}} ds.
		\end{align*}
		For the first term, from Theorem~\ref{LLN}, for $p=2$, we have that 
		\begin{align*}
			\E\sup_{0\leq t\leq T}|z_t|^2 =\E\sup_{0\leq t\leq T}|X^{\e,n}_{\pi^n(s)}  - x_{\pi^n(s)}|^2\leq C_{ABKT} \left(\E (\NN_2)+\e^2 \right).
		\end{align*} 
		With above results and calculations we have
		\begin{align*}
			\frac 1{\e} |I-A^{-1}BK||A| \E\left[\left(\sup_{0 \leq t \leq T} |z_t|\right)^2\right]^{\frac 12} &\E\left[\left(\int_0^T \xi^n_{N_s^{n}+1}e^{\xi^n_{N_s^{n}+1}|A|} ds\right)^2\right]^{\frac 12} \\
			\leq & \frac 1{\e} C_{ABKT}\left(\ E[\NN_2] + \e^2 \right)^{\frac12} (\E[\tilde{ \NN}_3])^{\frac12}.
		\end{align*}
		Therefore, we have 
		\begin{align*}
			\E\left[\sup_{0\leq t\leq T} |G_1(t)|\right]\leq \frac 1{\e} 
			C_{ABKT} \left(\ E[\NN_2] + \e^2 \right)^{\frac12} (\E[\tilde{\NN}_3])^{\frac12}\le \cc n^{-1/2}(n^{-1}+\e)C_{ABKT\xi_1},
		\end{align*} by Lemma \ref{Nptilde} and its Corollary \ref{Np}.
	\end{proof}

	\subsubsection{The \texorpdfstring{$G_2$}{G2} Term}\label{G2sub}
	In this section, we prove Lemma~\ref{G2estimates} which provides an estimate for the expression  \eqref{G2}. The key idea is that the term $$
	\frac{e^{(s - \pi^n(s)) A} - I - (s - \pi^n(s)) A}{1/n},
	$$
	by virtue of being the remainder of a Taylor series, decays fast enough in $n$ by Lemma~\ref{Expbd}. 
	
	\begin{proof}[\textbf{Proof of Lemma~\ref{G2estimates}}]
		We have 
		\begin{align*}\label{G22}
			G_2(t)= \frac{1}{\e n}\int_0^t \left(\frac{e^{(s - \pi^n(s)) A} - I - (s - \pi^n(s)) A}{1/n}\right) (I-A^{-1}BK)x_{\pi^n(s)}ds
		\end{align*}
		Letting
		\begin{equation*}
			g^n(s) := \left( \frac{e^{s A} - I - sA}{1/n}\right) (I - A^{-1} BK).
		\end{equation*}
		It can be easily seen that
		\begin{align*}
			\int_0^t &\left(\frac{e^{s - \pi^n(s) A} - I - (s - \pi^n(s)) A}{1/n}\right) (I-A^{-1}BK)x_{\pi^n(s)}ds \nonumber\\&=   \sum_{k=0}^{N^n_t - 1} \left( \int_0^{\xi^n_{k+1}} g^n(s)  ds \right) x_{\tau^n_k} + \left( \int_0^{t - \pi^n(t)} g^n(s)  ds \right) x_{\pi^n(t)}\nonumber\\
			&\leq \sup_{0 \leq t \leq T}\left[ \sum_{k=0}^{N_t^n - 1} \left|\int_0^{\xi^n_{k+1}} g^n(s)ds\right| |x_{\tau^n_k}|+ \left|\int_0^{t - \pi^n(t)} g^n(s)  ds\right| |x_{\pi^n(t)}|\right].
		\end{align*}
		A direct calculation yields
		\begin{equation}\label{gn}
			\begin{aligned}
				\int_0^{r} g^n(s)  ds & = nA^{-1} \left( e^{rA} - I - r A -\frac{r^2A^2}{2}\right) (I - A^{-1} BK) \\\text{or}, \quad
				\left| \int_0^r g^n(s)  ds \right| 
				&\leq n |A^{-1}|\cdot \left| e^{rA} - I - r A -\frac{r^2A^2}{2}\right|\cdot |I - A^{-1} BK| \\
				&=n |A^{-1}|\cdot \left| \sum_{k=3}^{\infty}\frac{r^kA^k}{k!}\right|\cdot |I - A^{-1} BK|
				\\&\leq n \frac{r^3|A|^3}{6} |A^{-1} |\cdot e^{r|A|}\cdot |I - A^{-1} BK|.
			\end{aligned}
		\end{equation}
		In equation \eqref{gn}, we have used Lemma \ref{Expbd}. Let us put $ r=\xi_{k+1}^{n}$ to get 
		\begin{align}\label{gn1}
			\left|\int_0^{\xi^n_{k+1}} g^n(s)  ds \right| \leq nC_{ABK}\left[\frac{(\xi^n_{k+1})^3}{6}e^{\xi^n_{k+1}|A|} \right].
		\end{align}
		Now putting $r= t - \pi^n(t)$, we get
		\begin{equation}\label{gn2}
			\left| \int_0^{t - \pi^n(t)} g^n(s)  ds \right| 
			\leq nC_{ABK}\left[\frac{(t - \pi^n(t))^3}{6}e^{(
				t - \pi^n(t))|A|} \right].
		\end{equation}
		Using the estimates \eqref{gn1} and \eqref{gn2} 
		\begin{align}\label{gn3}
			\sup_{0\leq t \leq T}&\left|\int_0^t \left(\frac{e^{s - \pi^n(s) A} - I - (s - \pi^n(s)) A}{1/n}\right) (I-A^{-1}BK)x_{\pi^n(s)}ds\right| \nonumber\\
			\leq & ~ nC_{ABKT}\sup_{0 \leq t \leq T} \left[\sum_{k=0}^{N_t^n - 1} \frac{(\xi_{k+1}^n)^3}{6} e^{\xi_{k+1}^n|A|} + \frac{(t - \pi^n(t))^3}{6}e^{(t - \pi^n(t))|A|}\right]\sup_{0 \leq t \leq T} |x(t)|\nonumber\\
			\leq & ~ nC_{ABKT} \sum_{k=0}^{N^n_t} {(\xi_{k+1}^n)^3}e^{\xi_{k+1}^n |A|},
		\end{align}
		by Lemma \ref{ODEbd}. Now, \eqref{gn3} and  Lemma  \eqref{Nptilde} assures that
		\begin{align*}
			\E\left[\sup_{0\leq t \leq T}\left|G_2(t)\right|\right]\leq  \frac n\e C_{ABKT} \sum_{k=0}^{N^n_t} {(\xi_{k+1}^n)^3}e^{\xi_{k+1}^n |A|}\leq  \frac n\e C_{ABKT}\E[\tilde{\NN}_3]\leq \cc n^{-1} C_{ABKT\xi_1}.
		\end{align*}
	\end{proof}

	\subsubsection{The \texorpdfstring{$G_3$}{G3} Term}\label{G3sub}	
	In this section, we prove Lemma~\ref{G3Lemma} which provides an estimate for the expression \eqref{G3}. The key idea is controlling the term $|x_{\pi^n(s)} - x_s|$ using Lemma~\ref{Assympwithoutnoise}.
	\begin{proof}[\textbf{Proof of Lemma \ref{G3Lemma}}] From equation \eqref{G3}, we have
		\begin{align*}
			\E \Big[\sup_{0\leq t\leq T} |G_3(t)|\Big] & \leq \frac{|A-BK|}{\e n} \E \left[\int_0^T\frac{(s-\pi^n(s))}{\frac 1n} \cdot |x_{\pi^n(s)} - x(s))| ds\right] \\
			& \leq \frac{|A-BK|}{\e } \E \left[\sup_{0 \leq t \leq T} \left(|x_{\pi^n(t)} - x(t))|\right)\int_0^T(s-\pi^n(s)) ds\right] \\
			& \leq \frac{|A-BK|}{\e }\left( \E \left[\sup_{0 \leq t \leq T} \left(|x_{\pi^n(t)} - x(t))|^2\right)\right]\right)^{\frac 12}\left(\E \left[\left(\int_0^T(s-\pi^n(s)) ds\right)^2\right]\right)^{\frac 12}\\
			& \leq \frac{|A-BK|}{\e }\left( \E \left[\sup_{0 \leq t \leq T} \left(|x_{\pi^n(t)} - x(t))|^2\right)\right]\right)^{\frac 12}\left(\E \left[\int_0^T(s-\pi^n(s))^2 ds\right]\right)^{\frac 12}. 
		\end{align*}
		by applying the C-S inequality twice.  Now, for first term, we can use the Lemma \ref{Assympwithoutnoise} for $p=2$, which gives
		\begin{align*}
			\E\sup_{0\leq t\leq T}\left|x_{\pi^n(s)} - x(s)\right|^2\leq C_{ABKT}\E [{\NN}_2].
		\end{align*} 
		For the second term, by applying Corollary~\ref{Np}, we ultimately obtain
		\begin{align*}
			\E \left[\sup_{0\leq t\leq T}|G_3(t)|\right]\leq \frac{C_{ABKT}}{\e} \E[{\NN}_2]\le \cc n^{-1} C_{ABKT\xi_1}.
		\end{align*}
	\end{proof}
	
	\subsubsection{The \texorpdfstring{$G_5$}{G5} Term}\label{G5sub}	
	\begin{proof}[\textbf{Proof of Lemma \ref{G5estimates}}]
		We have
		\[G_5=\left(\frac{1}{\e n} - c\right) \int_0^t M(A-BK) x(s)ds.\]
		Therefore,
		\begin{align*}
			\sup_{0\leq t\leq T} |G_5(t)|
			&\leq \left(\frac{1}{\e n} - c\right)  \left[\sup_{0 \leq t \leq T} \left|\int_0^t M(A-BK) x(s) ds\right| \right ] \\
			&\leq \left(\frac{1}{\e n} - c\right)M|A-BK|  \left[  \int_0^T \left| x(s)\right|ds \right]\\
			&\leq \left(\frac{1}{\e n} - c\right)MT|A-BK|  \left[ \sup_{0 \leq t \leq T} \left| x(s)\right|ds \right].
		\end{align*}
		There exists $\e_0>0$, such that for any $0<\e<\e_0,$ we obtain 
		\begin{equation*}
			\sup_{0\leq t\leq T} |G_5(t)|\leq C_{ABKT}M\varkappa(\e),     \end{equation*} by applying Definition \ref{kappa} together with Lemma \ref{ODEbd}.
	\end{proof}
	
	Now, at last, we are left only with Lemma \ref{G4Lemma}, which we will address in the following subsubsection.
	\subsubsection{The \texorpdfstring{$G_4$}{G4} Term}\label{G4sub}
	The proof of Lemma~\ref{G4Lemma} consists of several intermediate steps. To elaborate on these steps, it will be  convenient to recall the definition of $G_4$ from \eqref{G4}:
	\begin{equation*}
		G_4(t)=\frac{1}{\e n}\int_0^t \left(\frac{(s-\pi^n(s))A}{1/n} - MA\right)(I-A^{-1}BK) x_s ds.
	\end{equation*}
	The key idea is that replacing $(I-A^{-1}BK) x_s$ by any smooth bounded function and to obtain the following lemma.
	\begin{lemma}\label{G4decomposition}
		Let $f(x_s):= f(s)$ be a smooth, and  bounded function whose second derivative grows at most exponentially in time . Then, for any finite $T<\infty$, we have
		\begin{align}
			\E\left[ \sup_{0\leq t\leq T}\left| \int_0^t \frac{(s-\pi^n(s))}{1/n} f(s)ds- \int_0^t M f(s)ds\right|\right] \leq n^{-1/4}C_{\xi_1fT}.
		\end{align}
		
	\end{lemma}

	\begin{rmkk}
		Let $x_t$ denote the solution of system \eqref{ode1} with initial condition $x_0$. Then one example of such an $f(s)=f(x_s)$ is $ f(x_s) = (I-A^{-1}BK)x_s = (I -A^{-1}BK)e^{(A-BK)s}x_0 := 
		f(s).$
		In what follows, we apply Taylor’s theorem to the mapping $s \mapsto f(x_s)$. 
		For notational simplicity, we treat $f(x_s)$ as a function depending only on time,  that is, we write $f(x_s)=f(s)$, with the understanding that the time dependence  arises through the trajectory $x_s$. This convention will be used throughout without further mention.
		
	\end{rmkk}
	
	The proof of Lemma \ref{G4decomposition} is based on a suitable decomposition of the integral\\
	$\displaystyle\int_0^t n{(s-\pi^n(s))} f(s)ds$ into several terms that will be treated separately. In place of working with the full expression at this stage, we focus on outlining the main steps of the decomposition argument and the techniques involved. Each term in the decomposition is estimated separately and finally by combining all the estimates obtained  for the individual terms, we conclude the proof of Lemma \ref{G4decomposition} and consequently proof of Lemma \ref{G4Lemma}.
	
	\textbf{Sketch of the   Proof of Lemma \ref{G4decomposition}.}~~We now present the sketch of the proof. Before explaining the argument, we point out  that the first step involves a Taylor expansion. By the second order mean value theorem, for each $n\in \mathbb N, k \in \mathbb Z_+$ and $r>0$ there exists $\eta_k \in [\tau^n_k,\tau^n_{k+1}]$ such that
	\begin{equation}
		\int_0^{r} nsf(\tau^n_k+s)ds = n\frac{r^2}{2}f(\tau^n_k) + \frac{nr^3}{6}(2f'(\eta_k)+\tilde{\eta}_k f''(\eta_k)),\label{eq:taylorexpl}
	\end{equation}
	where $\tilde{\eta}_k=\eta_k-\tau_k^n, f' = \frac{df}{ds}$ and $f'' = \frac{d^2f}{ds^2}$. The integral ~ $\displaystyle\int_0^t n{(s-\pi^n(s))} f(s)ds$, for any $t\in [0,T]$ will asymptotically simplify in the following way.
	Below, $\eta, \eta_k \in [0,T]$ are random variables arising from remainder terms in the respective Taylor expansions, while terms that asymptotically converge to $0$ are indicated in square brackets, along with the corresponding lemma stating this.  In the second line below, we used \eqref{eq:taylorexpl} for $r = \xi^n_k$ and $t-\pi^n(t)$.
	
	\begin{align*}
		&\int_0^t \frac{(s-\pi^n(s))}{\frac 1n} f(s)ds \nonumber \\
		=& \sum_{k=0}^{N^n_t-1} \int_{0}^{\xi_{k+1}^n} nsf(\tau_k^n+s)ds + \int_0^{t-\pi^n(t)} nsf(\pi^n(t)+s)ds \nonumber \\
		=& \sum_{k=0}^{N^n_t-1} \frac{n(\xi^n_{k+1})^2}{2}f(\tau_k^n) +\frac{n(t-\pi^n(t))^2}{2}f(\pi^n(t))  \nonumber\\
		+&\left[\frac{n(\xi^n_{k+1})^3}{6}(2f'(\eta_k)+\tilde{\eta}_k  f''(\eta_k))+\frac{n(t-\pi^n(t))^3}{6}(2f'(\eta)+(\eta - \pi^n(t)) f''(\eta))\right]_{\to 0,~ \text{by Lemma~\ref{Remainder1}}}\\ 
		\approx& \sum_{k=0}^{N^n_t-1} \frac{n}{2}(\xi_{k+1}^n)^2 f(\tau_k^n) + \left[\frac{n}{2}(t-\pi_n(t))^2 f(\pi^n(t))\right]_{\to 0,~ \text{by Lemma~\ref{Remainder2}}}\\
		\approx &\sum_{k=0}^{N^n_t-1} \frac{n}{2}(\xi_{k+1}^n)^2 f(\tau_k^n) \\
		= & \sum_{k=0}^{nt/\E[\xi_1]-1} \frac{n}{2}(\xi_{k+1}^n)^2 f(\tau_k^n) + \left[\sum_{k \in (nt/\E[\xi_1]-1, N^n_t-1)}  \frac{n}{2}(\xi_{k+1}^n)^2 f(\tau_k^n)\right]_{\to 0,~ \text{by Lemma~\ref{Remainder3}}} \nonumber \\
		\approx & \sum_{k=0}^{nt/\mathbb {E}[\xi]-1} \frac{n}{2}(\xi_{k+1}^n)^2 f(\tau_k^n)\\
		= &  \sum_{k=0}^{nt/\E[\xi_1]-1} \frac{n}{2}(\xi_{k+1}^n)^2 f\left(\frac{k\E[\xi_1]}{n}\right) \\&\hspace{2.5cm}+ \left[\sum_{k=0}^{\lceil nt/\E[\xi_1] \rceil-1} \frac{n}{2}(\xi^n_{k+1})^2 \left(f(\tau^n_k) - f\left(\frac{k\E[\xi_1]}{n}\right)\right)\right]_{\to 0,~ \text{by Lemma~\ref{Remainder4}}}  \nonumber\\
		\approx & \sum_{k=0}^{nt/\E[\xi_1]-1} \frac{n}{2}(\xi_{k+1}^n)^2 f\left(\frac{k\E[\xi_1]}{n}\right)\\
		= & \sum_{k=0}^{nt/\E[\xi_1]-1} \frac{n}{2}\E(\xi_{k+1}^n)^2 f\left(\frac{\E[\xi_1]k}{n}\right)  \\&\hspace{2.3cm}+\left[\sum_{k=0}^{\lceil nt/\E[\xi_1]\rceil- 1} \frac{n}{2} \left((\xi^n_{k+1})^2 -\E[(\xi^n_{k+1})^2]\right)f\left(\frac{k\E[\xi_1]}{n}\right)\right]_{\to 0,~ \text{by Lemma~\ref{Remainder5}}}\nonumber\\
		\approx &  \sum_{k=0}^{nt/\E[\xi_1]-1} \frac{n}{2}\E(\xi_{k+1}^n)^2 f\left(\frac{k\E[\xi_1]}{n}\right)\nonumber\\
		\approx &\int_0^{t} Mf(s)ds~\text{ by Lemma~\ref{Remainder6}.}
	\end{align*}

	The next step is to do detailed examination of the decomposition introduced above. Our objective is to derive precise estimates for each individual term and to demonstrate that the contributions appearing in big brackets are negligible in the asymptotic regimes under consideration. This will be accomplished through a sequence of auxiliary results, namely Lemmas~ \ref{Remainder1}-\ref{Remainder6}. Once these estimates are established and combined, the proof of Lemma~\ref{G4decomposition} will follow as a direct and immediate consequence. 
	
	We now proceed by analyzing those bracket terms one by one. The first contribution is treated in Lemma~\ref{Remainder1}. Fundamentally, the asymptotic decay of this term is due to the appearance of the cubes of the inter-renewal times $(\xi^n_{k+1})^3$. It also arises as a Taylor remainder, which naturally leads to the following result.
	\begin{lemma}\label{Remainder1}
		Let $t \leq T$, $\eta_k \in (\tau^n_k, \tau^n_{k+1})$ for each $k \in 0,1,\ldots,N_t^n-1$, $\eta \in (t-\pi^n(t), t)$. Then,
		\begin{align*}
			\E\Bigg[\sup_{0 \leq t \leq T} \Bigg|\sum_{k=0}^{N^n_t - 1}\frac{n(\xi^n_{k+1})^3}{6}(2&f'(\eta_k)+\tilde{\eta}_k f''(\eta_k)) 
			\nonumber\\&+ \frac{n(t-\pi^n(t))^3}{6}(2f'(\eta) + (\eta -\pi^n(t))f''(\eta))\Bigg|\Bigg]
			\leq \frac{C_{\xi_1 f T}}{n}. 
		\end{align*}
	\end{lemma}

	\begin{proof}
		First, note that $\eta_k,\eta \in [0,T]$. Therefore, $\max\{|2f'(\eta_k) + \tilde{\eta}_k f''(\eta_k)|, |2f'(\eta) + (\eta - \pi^n(t)) f''(\eta)|\} \leq C_{fT}$ for some constant $C_{fT}$ depending only on $f$ and $T$. Using this, for any $t \in [0,T]$,
		\begin{align}
			&\left|\sum_{k=0}^{N^n_t - 1}\frac{n(\xi^n_{k+1})^3}{6}(2f'(\eta_k)+\tilde{\eta}_k f''(\eta_k)) + \frac{n(t-\pi^n(t))^3}{6}(2f'(\eta) + (\eta -\pi^n(t))f''(\eta))\right| \nonumber \\
			&\leq C_{fT}\left(\sum_{k=0}^{N^n_t - 1}\frac{n(\xi^n_{k+1})^3}{6} + \frac{n(t-\pi^n(t))^3}{6}\right)\leq C_{fT} \sum_{k=0}^{N^n_t}\frac{n (\xi^n_{k+1})^3}{6} \leq C_{fT} \sum_{k=0}^{N^n_T+1} \frac{n (\xi^n_{k+1})^3}{6}.\label{rem11}
		\end{align}    
		By Lemma~\ref{Waldlemma} and the fact that $\xi^n_k = \frac{\xi_k}{n}$, \begin{equation}
			\E\left[\sum_{k=0}^{N_T^n+1} \frac{n(\xi^n_{k+1})^3}{6}\right] = \frac{\E[\xi_1^3]}{n}\frac{\E[N_T^n]+2}{6n}.\label{rem12}
		\end{equation}
		By Theorem~\ref{thm:renewal}, the second term converges to $\frac{1}{\E[\xi_1]}$, while the first term converges to $0$. When combined with \eqref{rem11} and \eqref{rem12}, this completes the proof. 
	\end{proof}	
	
	The next bound is essentially proved using the idea that when $n$ is large, every inter-renewal interval is, on average, of size $\frac {\E[\xi_1]}{n}$. We will use Proposition~\ref{suptosum}.
	\begin{lemma}\label{Remainder2}
		We have 
		$$
		\E\left[\sup_{ 0 \leq t \leq T} \frac{n}{2} \left|\left(t-\pi^n(t)\right)^2 f(\pi^n(t))\right|\right] \leq \frac{C_{\xi_1 f T}}{n^{1/2}}. 
		$$

	\end{lemma}
	\begin{proof}
		Since $t \leq T$, $\pi^n(t) \leq T$. Therefore, there is a constant $C_{fT}>0$ such that \begin{equation}\label{rem21}
			\sup_{0 \leq t \leq T} \frac{n}{2} \left(t-\pi^n(t)\right)^2 |f(\pi^n(t))| \leq C_{fT}\frac{n}{2} \sup_{0 \leq t \leq T}\left(t-\pi^n(t)\right)^2.
		\end{equation}
		By \eqref{nnt}, $$
		\pi^n(t) = \pi^n(t) = \frac 1n \tau^1_{N^n_t} = \frac 1n \tau^1_{N^1_{nt}} = \frac 1n\pi^1(nt).
		$$
		Now, for any $t \in [0,T]$, $$
		\frac{n}{2} \left(t-\pi^n(t)\right)^2 = \frac{(nt - n\pi^n(t))^2}{2n} = \frac{(nt - \pi^1(nt))^2}{2n}.
		$$
		By the above equality, the C-S inequality and Proposition~\ref{suptosum} with $X_i=\xi_i^2$ and $q=2$,
		\begin{multline*}
			\E\left[\sup_{0 \leq t \leq T} \frac n2 (t-\pi^n(t))^2\right] = \E\left[\sup_{0 \leq t \leq T} \frac{(nt-\pi^1(nt))^2}{2n}\right] \\ \leq \frac 1{2n} \E\left[\sup_{k \leq N^1_{nT} + 1} (\xi_{k})^2 \right] \leq \frac 1{2n}\sqrt{\E[\xi_1^4]}\sqrt{\E[N^1_{nT}]+2}.
		\end{multline*}
		By Theorem~\ref{thm:renewal}, $\sqrt{\E[N^1_{nT}]+2}$ grows at the rate $\sqrt{n}$, which ensures that the right hand side decays to $0$ at the rate $\frac 1{\sqrt{n}}$. Combining the above equation with \eqref{rem21}, the result follows.
	\end{proof}
	
	Lemmas~\ref{Remainder3} and~\ref{Remainder4} are slightly harder to prove and to make them easier to treat we introduce the following auxiliary results : Lemma \ref{unifintrem34} - Lemma \ref{lem:usc}. The first of these establishes the uniform integrability condition.
	
	\begin{lemma}\label{unifintrem34}
		The following bounds hold for the renewal process $N^1_t$ and arrival process $\tau^n_k :$
		\begin{enumerate}
			\item [(i)] We have \begin{equation}\label{unifint}
				\E\left[\sup_{0 \leq t \leq T} \left|\frac{nt/\E[\xi_1] - N^1_{nt}}{\sqrt{n}}\right|^6\right] \leq C_{T\xi_1}. 
			\end{equation}
			In particular, $\left\{\sup_{0 \leq t \leq T} \left|\frac{nt/\E[\xi_1] - N^1_{nt}}{\sqrt{n}}\right|^4 : n \in \mathbb N\right\}$ is uniformly integrable.
			\item[(ii)] We have 
			\begin{equation}\label{unifint3}
				\E\left[\left(\frac{\E[\xi_1]}{n}\sum_{k=0}^{\lceil nT/\E[\xi_1] \rceil-1}\left|\frac{\tau_k- k\E[\xi_1]}{Var(\xi_1)\sqrt{n}}\right|\right)^{6}\right] < C_{T\xi_1}.
			\end{equation}
			In particular, $\displaystyle\left\{\left(\frac{\E[\xi_1]}{n}\sum_{k=0}^{\lceil nT/\E[\xi_1] \rceil-1}\left|\frac{\tau_k- k\E[\xi_1]}{Var(\xi_1)\sqrt{n}}\right|\right)^{4}  : n \in \mathbb N\right\}$ is uniformly integrable.
		\end{enumerate}
	\end{lemma}
	\begin{proof}
		We first prove part (i). Observe that 
		\begin{align*}
			\left|nt/\E[\xi_1] - N^1_{nt}\right|& \leq \left|nt/\E[\xi_1] - \tau_{N^1_{nt}}/\E[\xi_1]\right| + \left|\tau_{N^1_{nt}}/\E[\xi_1] - N^1_{nt}\right| \\
			&\leq \frac 1{\E[\xi_1]}(|\xi_{N^1_{nt}+1}| + |\tau_{N^1_{nt}} - \E[\xi_1]N^1_{nt}|).
		\end{align*}
		By the triangle inequality, and the inequality $(a+b)^6 \leq 64a^6+64b^6$ we have $$
		\sup_{0 \leq t \leq T}\left|nt/\E[\xi_1] - N^1_{nt}\right|^6 \leq \frac{64}{\E[\xi_1]^6}\left(\sup_{0 \leq k \leq N^1_{nT}+1}\left|\xi_k\right|^6 + \sup_{0 \leq k \leq N^1_{nT}}\left|\tau_{k} - k\E[\xi_1]\right|^6\right).
		$$
		Taking expectations over both sides and dividing by $n^{3}$, it suffices to show that \begin{gather}
			n^{-3}\E\sup_{0 \leq k \leq N^1_{nT}+1}\left|\xi_k\right|^6 < C_{T \xi_1} \label{partone}\\
			n^{-3}\E \sup_{0 \leq k \leq N^1_{nT}}\left|\tau_{k} - k\E[\xi_1]\right|^6 < C_{T\xi_1},\label{parttwo}
		\end{gather}
		for a constant $C_{T\xi_1}$ independent of $n$. For the first, we use Proposition~\ref{suptosum} with $q=1$ to get
		$$
		n^{-3}\E\sup_{0 \leq k \leq N^1_{nT}+1}\left|\xi_k\right|^6 \leq n^{-3}\E[N^1_{nT}+2] \E[(\xi_k)^6] = n^{-2}\E[(N^1_{nT}+2)/nT] \E[(\xi_k)^6].
		$$
		By Theorem~\ref{thm:renewal}, this term decays to $0$ at the rate $n^{-2}$. For the second term, we use Lemma \ref{unifint1} (a) on $M_k = \tau_{k} - k\E[\xi_1]$ and Lemma \ref{unifint2} to see that $$
		n^{-3}\E \sup_{0 \leq k \leq N^1_{nT}}\left|M_k\right|^6 \leq n^{-3}\E[(N^1_{nt})^4]^{\frac 12} \E[(N^1_{nt})^2]^{\frac 12} \leq C_{T\xi_1}.
		$$
		for some constant $C$ independent of $n$. Thus, \eqref{partone} and \eqref{parttwo} are proved, showing that \eqref{unifint} holds.
		
		Next, we prove part (ii). Recalling that $M_k = \tau^1_k - k \E[\xi_1]$ we have $$
		\E\left[\frac{\E[\xi_1]^{6}}{n^{6}}\left(\sum_{k=0}^{\lceil nT/\E[\xi_1] \rceil-1}\left|\frac{\tau_k- k\E[\xi_1]}{Var(\xi_1)\sqrt{n}}\right|\right)^{6}\right] = \frac {\E[\xi_1]^{6}}{Var(\xi_1)^6n^{9}} \E\left[\left(\sum_{k=0}^{\lceil nT/\E[\xi_1]\rceil - 1} |M_k|\right)^{6}\right]. 
		$$
		It is sufficient to focus on the expectation. By Holder's inequality, $$
		\left(\sum_{k=0}^{\lceil nT/\E[\xi_1]\rceil - 1} |M_k|\right)^{6} \leq (\lceil nT/\E[\xi_1]\rceil)^5\left(\sum_{k=0}^{\lceil nT/\E[\xi_1]\rceil - 1} |M_k|^6\right).
		$$
		By the Lemma \ref{unifint1} (b), the expectation of the left hand side is bounded by a constant times $n^5 \sum_{k=0}^{\lceil nT/\E[\xi_1]\rceil - 1} k^3$, which is at most a constant times $n^9$. Therefore, \eqref{unifint3} follows.
	\end{proof}
	
	Next, we include the proof of Lemma~\ref{lem:assistance} which will be used in both Lemmas~\ref{Remainder3} and ~\ref{Remainder4}.
	
	\begin{lemma}\label{lem:assistance} Let $\alpha\geq 0$ be a constant, then there exists a constant $C>0$ independent of $n$ such that  $$
		\E\left[e^{2\alpha\tau^n_{\lceil nT/\E[\xi_1]\rceil}}\right]^{\frac 12} \leq C_{\alpha T\xi_1}.
		$$
	\end{lemma}
	\begin{proof}
		Recall that $\tau^n_{\lceil nT/E[\xi_1]\rceil} = \frac 1n\sum_{k=0}^{\lceil nT/E[\xi_1]-1 \rceil} \xi_k$ is a sum of $\lceil \frac{nT}{\E[\xi_1]}\rceil $ i.i.d. copies of $\frac{\xi_k}{n}$. By Jensen's inequality, $\E[X^{\frac 1n}] \leq [\E X]^{\frac 1n}$ for all $n \geq 1$. Thus, we have
		\begin{align}
			\E\left[e^{2\alpha\tau^n_{\lceil nT/\E[\xi_1]\rceil}}\right]^{\frac 12} = \E\left[e^{\frac{2\alpha}{n}\xi_1}\right]^{\frac 12\lceil\frac {nT}{\E[\xi_1]}\rceil}& \leq \E[e^{2\alpha \xi_1}]^{\frac 1{2n}\lceil\frac {nT}{\E[\xi_1]}\rceil} \leq \E[e^{2 \alpha \xi_1}]^{ \left(\frac{T}{2\E[\xi_1]}+\frac 12\right)}.
		\end{align} 
		This completes the proof.
	\end{proof}
	
	Finally, we include the proof of the rather technical Lemma~\ref{lem:usc} which will be used in Lemma \ref{Remainder4}. Let $D([0, T])$ be the Skorokhod space, i.e., $$ D([0,T]) = \{x:[0,T] \to \R \;|\; x \text{ is right continuous and has left limits}\}.$$ Then we have the following result.
	
	\begin{lemma}\label{lem:usc}
		Suppose that $X_n \Rightarrow X$ as random elements on $D([0, T])$. Then, if $(\sup_{t \in [0,T]} X_n(t))_{n \geq 1}$ is uniformly integrable and $\E[\sup_{t \in [0,T]} X] < \infty$, we have
		\begin{equation*}
			\limsup_{n \to \infty} \E\left[\sup_{t \in [0,T]} X_n\right] \leq \E\left[\sup_{t\in [0,T]} X\right]. 
		\end{equation*}
	\end{lemma}
	
	\begin{proof}
		We start with the claim that the functional $\displaystyle\mathcal{T}y := \sup_{t \in [0,T]} y(t)$ is upper semicontinuous (usc) on $[0,T]$. For this, it is sufficient to prove that if $y_n \to y$ in $D[0,T]$, then $\displaystyle\limsup_{n \to \infty} \mathcal{T}y_n \leq \mathcal{T}y$. Following \cite[Page 112, second paragraph]{billingsley1968convergence}, $y_n \to y$ implies the existence of a sequence of functions $\lambda_n : [0,T] \to [0,T]$ which are strictly increasing continuous bijections, such that $\displaystyle y_n(\lambda_n(t)) \to y_n(t)$ and $\lambda_n(t) \to t$ \emph{uniformly} in $t$.
		
		Let $\eta>0$ be arbitrary. By the above paragraph, there exists $N$ such that $y_n(\lambda_n(t)) < y(t)+\eta$ for all $t\in [0,T]$ and $n>N$. In particular, since $\lambda_n$ are bijections, we have $$
		\mathcal{T}y_n = \mathcal{T}(y_n\circ \lambda_n) \leq \mathcal{T} (y \circ \lambda_n)+\eta = \mathcal{T}y+\eta
		$$
		whenever $n>N$. Since $\eta$ was arbitrary, this instantly implies our claim.
		
		For an arbitrary $L > 0$, consider the functional $\mathcal{T}_L(y) =\min\{L, \mathcal{T}y\}$. Since the pointwise minimum of usc functions remains usc, it follows that $\mathcal{T}_L(y)$ is a bounded, usc function on $[0,T]$. By \cite[Problem 7, Chapter 2]{billingsley1968convergence}, it follows that \begin{equation}\label{eq:ui0}
			\limsup_{n \to \infty}\E[\mathcal{T}_L X_n] \leq \E[\mathcal{T}_L X].
		\end{equation}
		
		By the uniform integrability condition, for all $\eta>0$ there exists $K \in \N$ such that\\ $\E[|\mathcal{T}X_n| 1_{\{|\mathcal{T}X_n| > K\}}] < \eta$ for all $n \in \N$. This implies that \begin{equation}\label{eq:ui0.5}
			\E[\mathcal{T} X_n - \mathcal{T}_K X_n] = \E[\mathcal{T} X_n 1_{\{\mathcal{T}X_n > K\}}] < \eta
		\end{equation}
		for all $n\in \mathbb N$. Now, for any arbitrary $L>0$ and $n \in \mathbb N$,
		\begin{align*}
			\E[\mathcal{T}X_n] = \E\left[\mathcal{T}X_n - \mathcal{T}_L X_n\right] + \E\left[\mathcal{T}_L X_n  - \mathcal{T}_L X\right] + \E\left[\mathcal{T}_L X - \mathcal{T} X\right] + \E\left[\mathcal{T} X\right]. 
		\end{align*}
		We take the limit-superior on both sides as $n \to \infty$. For any $\eta>0$, the first term can be made smaller than $\eta$ by taking $L$ large enough as in \eqref{eq:ui0.5}. The second term has non-positive limit superior for any $L$ by \eqref{eq:ui0}. The third term is clearly non-positive. Thus, taking $L \to \infty$ we obtain 
		$\limsup_{n \to \infty} E\left[\mathcal{T} X_n\right] \leq \E\left[\mathcal{T} X\right],$ as desired.
	\end{proof}
	
	
	We are now ready to address Lemma \ref{Remainder3}, which starts with a collection of approximation results for the sum $
	\sum_{k=0}^{N_t^n - 1} \frac{n}{2} \left(\xi^n_{k+1}\right)^2 f(\tau^n_k).
	$
	Our first lemma replaces the random sum by a deterministic sum, by noting that $N^n_t-1 \approx \frac{nt}{\E[\xi_1]}-1$ by Theorem~\ref{thm:renewal}. Therefore, values of $k$ in between $N^n_t-1$ and $\frac{nt}{\E[\xi_1]}-1$ are expected to contribute negligibly to the sum since they are very few indices in number. This is what Lemma~\ref{Remainder3} establishes. Throughout the next proofs, the interval $(a,b)$ is used to denote the set $\{x \in \mathbb{R} : \min(a,b) \leq x \leq \max(a,b)\}$.
	
	\begin{lemma}\label{Remainder3}
		We have 
		\begin{align*}
			\E\left[\sup_{ 0 \leq t \leq T} \left|\sum_{k \in (nt/\E[\xi_1]-1,  N^n_t-1) \cap \mathbb{Z}} \frac{n}{2} (\xi^n_{k+1})^2 f(\tau^n_k)\right|\right]\leq C_{\xi_1 fT} n^{-\frac {1}{4}}.
		\end{align*}
	\end{lemma}
	
	\begin{proof}
		Note that if $k \in (nt/\E[\xi_1]-1, N^n_t - 1) \cap \mathbb{Z}_+$ then $k \leq \max\{\lceil nt/\E[\xi_1]\rceil -1, N^n_t-1\}$. So, we have for $t \leq T$,
		$$
		\tau^n_{k} \leq \max\{\tau^n_{\lceil nt/\E[\xi_1]\rceil}, \pi^n(t)\} \leq \max\{\tau^n_{\lceil nt/\E[\xi_1] \rceil}, T\} \leq  \tau^n_{\lceil nT/\E[\xi_1] \rceil}+ T. 
		$$
		%
		%
		%
		Since $f''$ has at most exponential growth, so does $f$. 
		Therefore, for some $\alpha>0$, $$
		|f(\tau^n_k)| \leq Ce^{\alpha\tau^n_k}  \leq C_Te^{\alpha\tau^n_{\lceil nT/\E[\xi_1] \rceil}}.
		$$
		Using this we get 
		\begin{align}
			\E&\left[\sup_{ 0 \leq t \leq T} \left|\sum_{k \in (nt/\E[\xi_1]-1, N^n_t-1) \cap \mathbb{Z}_+} \frac{n}{2} (\xi^n_{k+1})^2 f(\tau^n_k)\right|\right] \nonumber \\
			&\leq   \E\left[C_{T}e^{\alpha\tau^n_{\lceil nT/\E[\xi_1] \rceil}}\sup_{0 \leq t \leq T} \sum_{k \in (nt/\E[\xi_1]-1,N^n_t-1)\cap \mathbb{Z}_+} \frac{n}{2}(\xi^n_{k+1})^2\right] \nonumber \\
			&\leq C_{T}\E\left[e^{2\alpha\tau^n_{\lceil nT/\E[\xi_1]\rceil}}\right]^{\frac 12} \E\left[\left(\sup_{0 \leq t \leq T} \sum_{k \in (nt/\E[\xi_1]-1,N^n_t-1) \cap \mathbb{Z}_+} \frac{n}{2}(\xi^n_{k+1})^2\right)^2\right]^{\frac 12} \label{eq:main}
		\end{align}
		by the C-S inequality. The first term is bounded independent of $n$, by Lemma~\ref{lem:assistance}:
		\begin{equation}
			\E\left[e^{2\alpha\tau^n_{\lceil nT/\E[\xi_1]\rceil}}\right]^{\frac 12} \leq C_{T\alpha}.\label{eq:one}
		\end{equation}
		We now focus our attention on the other term in \eqref{eq:main}, where we begin by noting that \begin{align}
			&\E\left[\left(\sup_{0 \leq t \leq T} \sum_{k \in (nt/\E[\xi_1]-1,N^n_t-1) \cap \mathbb {Z}_+} \frac{n}{2}(\xi^n_{k+1})^2\right)^2\right]^{\frac 12}\nonumber \\
			\leq & \frac 1{2n} \E\left[\sup_{0 \leq t \leq T} \left|nt/\E[\xi_1] - N^n_t\right|^2 \sup_{k \in (nt/\E[\xi_1] - 1, N^n_t- 1) \cap \mathbb Z_+}(\xi_{k+1})^{4} \right]^{\frac 12}.\label{eq:three}
		\end{align}		
		Now, suppose that $0 \leq t \leq T$ and $k \in (nt/\E[\xi_1]-1, N^n_t-1) \cap \mathbb{Z}$. Then, observe that $$k \leq \max\{\lceil nt/\E[\xi_1]\rceil-1 , N^1_{nt}-1\} \leq \max\{\lceil nT/\E[\xi_1]\rceil-1, N^1_{nT}-1\},$$
		where we used \eqref{nnt}. Therefore, $$
		\sup_{k \in (nt/\E[\xi_1]-1, N^1_{nt}-1) \cap \mathbb {Z_+}} (\xi_{k+1})^{4} \leq \sup_{k \in (0,\max\{\lceil nT/\E[\xi_1] \rceil-1, N^1_{nT}-1\}) \cap \mathbb Z_+} (\xi_{k+1})^{4}.
		$$
		Now, splitting the supremums and applying the C-S inequality in \eqref{eq:three},
		\begin{align}
			\frac{1}{2n}\E&\left[\sup_{0 \leq t \leq T} \left|nt/\E[\xi_1] - N^1_{nt}\right|^2 \sup_{k \in (nt/\E[\xi_1] - 1, N^n_t- 1) \cap \mathbb{Z}_+}(\xi_{k+1})^{4} \right]^{\frac 12}\nonumber \\
			& \leq \frac{1}{2n}\E\left[\left(\sup_{0 \leq t \leq T} \left|nt/\E[\xi_1] - N^1_{nt}\right|^2\right)\left( \sup_{k \in (0,\max\{\lceil nT/\E[\xi_1]\rceil - 1, N^1_{nT}- 1\})}(\xi_{k+1})^{4}\right) \right]^{\frac 12}\nonumber \\
			& \leq\frac{1}{2n} \E\left[\left(\sup_{0 \leq t \leq T} \left|nt/\E[\xi_1] -N^1_{nt}\right|^{4}\right)\right]^{\frac 1{4}}\E\left[\left( \sup_{k \in (0,\max\{\lceil nT/\E[\xi_1] \rceil - 1, N^1_{nT}- 1\}) \cap \mathbb{Z}}(\xi_{k+1})^{8}\right) \right]^{\frac 1{4}}\label{eq:four}.
		\end{align}	      
		To bound the first term in \eqref{eq:four}, we note that by Donsker's Theorem~\ref{donsker} and the continuous mapping theorem \cite[Theorem 5.1]{billingsley1968convergence}, 
		$$
		\left|\frac{nt/\E[\xi_1] - N^1_{nt}}{\sqrt{n}}\right|^4 \Rightarrow \frac{Var(\xi_1)^4}{\E[\xi_1]^{12}}|B_t|^4,
		$$
		where $B_t$ is a Brownian motion (in distribution). By Lemma~\ref{unifintrem34} (i) and Lemma~\ref{lem:usc} we have 
		\begin{align}
			\limsup_{n \to \infty} \E\left[\left(\sup_{0 \leq t \leq T} \left|\frac{nt/\E[\xi_1] -N^1_{nt}}{\sqrt{n}}\right|^{4}\right)\right]^{\frac 1{4}}
			\leq \frac{Var(\xi_1)}{\E[\xi_1]^3}\E\left[\sup_{0 \leq t \leq T} |B_t|^{4}\right]^{\frac 1{4}},\label{eq:five}
		\end{align}
		Thus, the first term in \eqref{eq:four}   grows at the rate $\sqrt{n}$ as $n \to \infty$.
		
		For the other term, we use Proposition~\ref{suptosum} with $\frac{nT}{\E[\xi_1]}+N^1_{nT}+1$ and $p=1$. Indeed,
		\begin{align}
			\E\left[ \sup_{k \in (0,\max\{ nT/\E[\xi_1], N^1_{nT}- 1\}) \cap \mathbb{Z}_+}(\xi_{k+1})^{8}\right]^{\frac 1{4}}& 
			\leq\E\left[ \sup_{k \leq nT/\E[\xi_1] + N^1_{nT}+ 1} (\xi_{k+1})^{8}\right]^{\frac 1{4}} \nonumber\\
			&\leq  \E[\xi_{1}^{8}]^{\frac 1{4}} \left(\E\left[nT/\E[\xi_1] + N^1_{nT}+ 1\right]\right)^{\frac 1{4}}.\label{eq:last}
		\end{align}	
		This term grows at the rate $n^{\frac 14}$ as $n \to \infty$, by Theorem~\ref{thm:renewal}. Combining all the numbered equations \eqref{eq:main}-\eqref{eq:last}, we obtain that $$
		\E\left[\sup_{ 0 \leq t \leq T} \left|\sum_{k \in (nt/\E[\xi_1]-1, N^n_t-1) \cap \mathbb{Z}} \frac{n}{2} (\xi^n_{k+1})^2 f(\tau^n_k)\right|\right] \leq C_{\xi fT}~ n^{-1+\frac 12+\frac 14} = C_{\xi fT}~ n^{-\frac 14}
		$$
		from where the result follows.
	\end{proof}
	
	The proof of next  Lemma \ref{Remainder4} involves handling $f(\tau^n_k)$. Note that we expect $\tau^n_k \approx \frac{k\E[\xi_1]}{n}$. Thus, the mean value theorem can be used to bound $f(\tau^n_k)-f\left(\frac{k\E[\xi_1]}{n}\right)$ in terms of $\tau^n_k - k\E[\xi_1]/n$. These key ideas form the crux of the next bound.
	\begin{lemma}\label{Remainder4}
		We  have \begin{align*}
			\E\left[\sup_{0 \leq t \leq T} \left|\sum_{k=0}^{\lceil nt/\E[\xi_1] \rceil-1} \frac{n}{2}(\xi^n_{k+1})^2 \left(f(\tau^n_k) - f(k\E[\xi_1]/n)\right) \right|\right]\le C_{T\xi_1}n^{-1/4}.
		\end{align*}
	\end{lemma}

	\begin{proof}
		For every $k=0,1,\ldots,\lceil nt/\E[\xi_1] \rceil-1$ and $0 \leq t \leq T$, observe that $$
		\max\{\tau^n_k,k\E[\xi_1]/n\} \leq \max\{\tau^n_{\lceil nt/\E[\xi_1]\rceil},t\} \leq T + \tau^n_{\lceil nT/\E[\xi_1]\rceil}.
		$$
		Recall that $f'$ grows at most exponentially, because $f''$ does so. By the mean value inequality,$$
		|f(\tau^n_k) - f(k\E[\xi_1]/n)| \leq C \left|\tau^n_k - k\E[\xi_1]/n\right| e^{\alpha \tau^n_{\lceil nT/\E[\xi_1]\rceil }}
		$$
		for some $\alpha>0$. This implies that
		\begin{align*}
			\sup_{0 \leq t \leq T}&\left|\sum_{k=0}^{\lceil nt/\E[\xi_1] \rceil-1} \frac{n}{2}(\xi^n_{k+1})^2 \left(f(\tau^n_k) - f(k\E[\xi_1]/n)\right) \right| \\ &\leq C\sup_{0 \leq t \leq T}\left(\sum_{k=0}^{\lceil nt/\E[\xi_1]\rceil -1} \frac{n}{2}(\xi^n_{k+1})^2|\tau^n_k- k\E[\xi_1]/n|\right)e^{\alpha\tau^n_{\lceil nT/\E[\xi_1] \rceil}} \\&
			\leq  C\left(\sum_{k=0}^{\lceil nT/\E[\xi_1]\rceil-1} \frac{n}{2}(\xi^n_{k+1})^2|\tau^n_k- k\E[\xi_1]/n|\right)e^{\alpha\tau^n_{\lceil nT/\E[\xi_1]\rceil}}.
		\end{align*}	
		We take the expectation above, and apply the C-S inequality to separate the terms.\begin{align}
			\E&\left[\left(\sum_{k=0}^{\lceil nT/\E[\xi_1]\rceil-1} \frac{n}{2}(\xi^n_{k+1})^2|\tau^n_k- k\E[\xi_1]/n|\right)e^{\alpha\tau^n_{\lceil nT/\E[\xi_1]}\rceil}\right] \nonumber \\
			& \leq  C \E\left[\left(\sum_{k=0}^{\lceil nT/\E[\xi_1]\rceil-1} \frac{n}{2}(\xi^n_{k+1})^2|\tau^n_k- k\E[\xi_1]/n|\right)^2\right]^{\frac 12}\E\left[e^{2\alpha\tau^n_{\lceil nT/\E[\xi_1]\rceil}}\right]^{\frac 12}.\label{rem41}
		\end{align}	
		By Lemma~\ref{lem:assistance}, \begin{equation}\label{rem42}
			\E\left[e^{2\alpha\tau^n_{\lceil nT/\E[\xi_1]\rceil}}\right]^{\frac 12} \leq C_{\alpha T}
		\end{equation}
		is bounded independent of $n$ (where we note that $\alpha = |A-BK|$ for the linear case). To bound the first term in \eqref{rem41}, we remove the factor $\sup_{ k \leq \lceil nT/\E[\xi_1]\rceil+1} \xi_{k+1}^2$ from the sum.
		\begin{align}
			\E&\left[\left(\sum_{k=0}^{\lceil nT/\E[\xi_1]\rceil-1} \frac{n}{2}(\xi^n_{k+1})^2\left|\tau^n_k- k\E[\xi_1]/n\right|\right)^2\right]^\frac12 
			\nonumber \\
			&=\frac 1{2n} \E\left[\left(\sum_{k=0}^{\lceil nT/\E[\xi_1]\rceil-1}(\xi_{k+1})^2|\tau^n_k- k\E[\xi_1]/n|\right)^2\right] ^\frac12\nonumber\\
			&\leq   \frac 1{2n} \E\Bigg[\left(\sum_{k=0}^{\lceil nT/\E[\xi_1]\rceil-1}|\tau^n_k- k\E[\xi_1]/n|\right)^2\left(\sup_{k\leq \lceil nT/\E[\xi_1] \rceil-1} (\xi_{k+1})\right)^{4}\Bigg]^{\frac 12} \nonumber \\
			&\leq \frac 1{2n} \E\left[\left(\sum_{k=0}^{\lceil nT/\E[\xi_1] \rceil-1}|\tau^n_k- k\E[\xi_1]/n|\right)^{4}\right]^{\frac 1{4}} \E\left[\sup_{k\leq \lceil nT/\E[\xi_1]\rceil-1} (\xi_{k+1})^8\right]^{\frac 1{4}}\label{rem43},
		\end{align} 
		where we applied the C-S inequality in the last step. The second term is controlled by applying Proposition~\ref{suptosum} with $q=1$, $X_i=\xi_{k+1}^8$ and noticing that $\lceil nT/\E[\xi_1]+1 \rceil$ is constant: 
		\begin{equation}
			\E\left[\left(\sup_{k\leq \lceil nT/\E[\xi_1]-1 \rceil} (\xi_{k+1})\right)^{8}\right]^{\frac 1{4}} \leq \E\left[\sum_{k=0}^{ \lceil nT/\E[\xi_1]+1 \rceil} (\xi_{k+1})^{8}\right]^{\frac 1{4}}
			\leq C_{T}\E\left[(\xi_{1})^{8}\right]^{\frac 14}n^{\frac 14}.\label{rem44}
		\end{equation}		
		The first term is handled by the usual Donsker's Theorem~\ref{usualdonsker}, from which we know that $$
		\frac{\tau_{\lfloor nt \rfloor} - \lfloor nt\rfloor \E[\xi_1]}{Var(\xi_1) \sqrt{n}} \Rightarrow B_t
		$$
		on $D([0,T])$, where $B_t$ is a Brownian motion (in distribution). We will now use the "generalized" continuous mapping theorem, see \cite[Theorem 1.11.1]{wellner2013weak}. Note that every function in $D([0,T])$ is Riemann integrable since it has only countably many discontinuities. Therefore, any sequence of Riemann sums of such a function converge to the integral of the function. That is, for any $g \in D([0,T])$, $$
		\frac{\E[\xi_1]}{n}\sum_{k=0}^{\lceil nT/\E[\xi_1] \rceil} \left|g\left(\frac{k}{n}\right)\right| \to \int_0^T |g(t)| dt.
		$$
		Applying the generalized continuous mapping theorem,
		\begin{equation}
			\left(\frac{\E[\xi_1]}{n}\sum_{i=1}^{\lfloor nT/\E[\xi_1]\rfloor}\left|\frac{\tau_{k} -  k \E[\xi_1]}{Var(\xi_1) \sqrt{n}} \right|\right)^{4}\Rightarrow \left(\int_0^T |B_t| dt\right)^{4}.\label{cmp}
		\end{equation}
		To strengthen this result into  convergence of expectations, observe that \begin{equation*}\E\left[\left(\sum_{k=0}^{nT/\E[\xi_1]-1}|\tau^n_k- k\E[\xi_1]/n|\right)^{4}\right]^{\frac 1{4}} = \sqrt{n}\frac{Var(\xi_1)}{\E[\xi_1]}\E\left[\left(\frac{\E[\xi_1]}{n}\sum_{k=0}^{\lceil nT/\E[\xi_1] \rceil-1}\left|\frac{\tau_k- k\E[\xi_1]}{Var(\xi_1)\sqrt{n}}\right|\right)^{4}\right]^{\frac 1{4}}.
		\end{equation*}
		By Lemma~\ref{unifintrem34} (ii),  \eqref{cmp} can be upgraded to convergence of moments. In particular,        
		$$
		\E\left[\left(\frac{\E[\xi_1]}{n}\sum_{k=0}^{\lceil nT/\E[\xi_1] \rceil-1}\left|\frac{\tau_k- k\E[\xi_1]}{Var(\xi_1)\sqrt{n}}\right|\right)^{4}\right]^{\frac 1{4}} \to \E\left[\left(\int_0^T |B_t| dt\right)^{4}\right]^{\frac 1{4}}.
		$$
		Therefore, it follows that this particular term grows at the rate $\sqrt{n}$ as $n \to \infty$. Thus
		\begin{align*}
			\E\left[\left(\sum_{k=0}^{\lceil nT/\E[\xi_1]\rceil-1} \frac{n}{2}(\xi^n_{k+1})^2|\tau^n_k- k\E[\xi_1]/n|\right)^2\right]^\frac12 \leq n^{-1/4}C_{T\xi_1}.
		\end{align*}
		Combining the above equation with the numbered equations \eqref{rem41}-\eqref{cmp}, the result follows.
	\end{proof}
	
	We shall now resolve the final step of this transition, which is the replacement of $(\xi^n_{k+1})^2$ by its expectation $\E[(\xi^n_1)^2]$. Once this is complete, we will have a deterministic Riemann integral to work with.
	
	\begin{lemma}\label{Remainder5}
		We have, for $\sigma = Var(\xi_1^2)$, that $$
		\E\left[\sup_{0 \leq t \leq T} \left|\sum_{k=0}^{\lceil nt/\E[\xi_1]\rceil- 1} \frac{n}{2} \left((\xi^n_{k+1})^2 -\E[(\xi^n_{k+1})^2]\right)f\left(\frac{k\E[\xi_1]}{n}\right)\right|\right] 
		\leq \frac{C\sigma_1}{n}\left(1+\frac{nT}{\E[\xi_1]}\right)^{\frac{1}{2}}.
		$$
	\end{lemma}

	\begin{proof} 
		Define   \begin{equation}\label{rem51} S_{m} = \sum_{k=1}^{m} \left((\xi_{k})^2 -\E[(\xi_{k})^2]\right)f\left(\frac{(k-1)\E[\xi_1]}{n}\right)\end{equation} and $S_{0}=0$. Then $S_{m}$ is a martingale with respect to $\mathcal{F}_m=\sigma(\xi_1, \xi_2, ....\xi_{m})$. By the C-S inequality and Doob's maximal inequality, 
		\begin{equation}\label{eq:rem52} \E\left[\sup_{0 \leq m \leq  \lceil \frac {nT}{\E[\xi_1]}\rceil}|S_m|\right]\leq \sqrt{ \E\left[\sup_{0 \leq m \leq  \lceil \frac {nT}{\E[\xi_1]}\rceil}|S_m|^2\right]}\leq 2\sqrt{\E\left[|S_{\lceil \frac {nT}{\E[\xi_1]}\rceil}|^2\right]}.\end{equation}
		Let $X_k = (\xi_{k})^2 - \E[(\xi_k)^2]$. For any $1 \leq m \leq \lceil \frac{nT}{\E[\xi_1]}\rceil$, by \eqref{rem51} we have
		\begin{align}
			\E|S_m|^2&=\sum_{k=1}^{m} \left|f\left(\frac{(k-1)\E[\xi_1]}{n}\right)\right|^2\E[(X_{1}^2)]
			\leq m C_{ABKT}Var(\xi_1^2).\label{eq:rem53}
		\end{align}
		Combining equations \eqref{rem51}-\eqref{eq:rem53},
		\begin{align*}
			\frac 1{2n}\E&\left[\sup_{0 \leq t \leq T} \left|\sum_{k=0}^{\lceil nt/\E[\xi_1]\rceil- 1} \left((\xi_{k+1})^2 -\E[(\xi_{k+1})^2]\right)f\left(\frac{k\E[\xi_1]}{n}\right)\right|\right]\\
			&\hspace{1.5cm} = \frac 1{2n}\E \left[\sup_{0 \leq m \leq \lceil \frac {nT}{\E[\xi_1]}\rceil} \left|S_m\right| \right]
			\leq \frac{C_{ABKT}\sigma_1 \sqrt{\lceil \frac {nT}{\E[\xi_1]}\rceil}}{n}
			\leq \frac{C\sigma_1}{n}\left(1+\frac{nT}{\E[\xi_1]}\right)^{\frac{1}{2}},
		\end{align*}
		whence the result follows. Thus, this term will decay at rate $1/\sqrt{n}.$
	\end{proof}
	The final lemma establishes the accuracy of the Riemann integral approximation. Although the argument is standard, we include the proof to maintain completeness of the presentation.
	\begin{lemma}\label{Remainder6} As $n\to \infty, $
		we have   
		\begin{align*}
			\E \left[\sup_{0\leq t\leq T}\left| \sum_{k=0}^{nt/\E[\xi_1]} \frac{n}{2}\mathbb{E}(\xi_{k+1}^n)^2 f(\E[\xi_1]k/n)-\int_0^{t} Mf(s)ds\right|\right]\le \frac{1}{n}C_{\xi_1MT}.
		\end{align*}
		Note that the expectation is over a deterministic quantity, hence irrelevant.
	\end{lemma}
	
	\begin{proof}
		We have 
		$$  \sum_{k=0}^{nt/\E[\xi_1]} \frac{n}{2}\E(\xi_{k+1}^n)^2 f(\E[\xi_1]k/n)=\sum_{k=0}^{nt/\E[\xi_1]} \frac{\E(\xi_{1})^2 }{2n}f(\E[\xi_1]k/n).
		$$
		Let $\Delta s=\E[\xi_1]/n$ and $s_k=\E[\xi_1]k/n.$ Then $\frac{\E{[\xi_1^2]}}{2n}=\frac{\E{[\xi_1^2]}}{2\E[\xi_1]}\Delta s=M\Delta s,$ and so
		$$\sum_{k=0}^{nt/\E[\xi_1]} \frac{\E(\xi_{1})^2 }{2n}f(\E[\xi_1]k/n)=\sum_{k=0}^{nt/\E[\xi_1]} \frac{\E(\xi_{1})^2 }{2n}f(s_k)=\sum_{k=0}^{nt/\E[\xi_1]} Mf(s_k)\Delta s.$$
		Let $g(s) = \int_{s_k}^s f(r)dr \in C^2[s_k,s_{k+1}]$, with $g'' = f'$. Since $\sup_{t\in [0,T]} \|f'\| < \infty$, we have
		$$\left|f(s_k)\Delta s-\int_{s_{k}}^{s_{k+1}} f(s) ds\right|\leq \|f'\|_{L^\infty} (\Delta s)^2/2. $$
		using the mean value inequality. Summing over  $0 \leq k \leq nt/\E[\xi_1]$ we have
		\begin{align*}
			\left|\sum_{k=0}^{nt/\E[\xi_1]} Mf(s_k)\Delta s-M\int_{0}^{t} f(s) ds\right|&\leq \|f'\|_{L^\infty} M \sum_{k=0}^{nt/\E[\xi_1]} (\Delta s)^2/2\\&\leq M\|f'\|_{L^\infty}\frac{nt}{\E[\xi_1]} \left(\frac{\E[\xi_1]}{n}\right)^2\le \frac{M t\|f'\|_{L^\infty}}{n}\E[\xi_1].
		\end{align*}
		Taking supremum over time $t\in [0,T],$ we get
		\begin{equation*}
			\sup_{t\in [0,T]}\left|\sum_{k=0}^{nt/\E[\xi]} Mf(s_k)\Delta s-M\int_{0}^{t} f(s) ds\right|\le \frac{1}{n}C_{\xi_1MT},
		\end{equation*}     
		completing the proof.
	\end{proof}
	
	Finally, the proofs of Lemma~\ref{G4decomposition} and \ref{G4Lemma} are straightforward.
	
	\begin{proof}[\textbf{Proof of the Lemma \ref{G4decomposition} and \ref{G4Lemma}}] Combining results of Lemmas \ref{Remainder1} - \ref{Remainder6}, we get the Lemma  \ref{G4decomposition}. Finally, observe that for $f(t)=(A-BK)x(t)=(A-BK)e^{(A-BK)t}x_0$, all the assumptions of Lemma \ref{G4decomposition} are satisfied. Hence, Lemma \ref{G4decomposition} applies to this function $f$, with constants depending on the matrices $A,B,K$ and on the time $T$. As a result, invoking the estimate from Lemma \ref{G4decomposition}, we obtain Lemma \ref{G4Lemma} as a direct corollary.
	\end{proof}

	\subsection{The Noise Part: Proposition \ref{noisebd}}\label{subsection5.2}
	
	In this section, we prove Proposition~\ref{noisebd}. The proof follows by bounding $\LL^{\e,n}_2(t)$ using functionals of the Brownian motion $W_s$.
	
	\begin{proof}[Proof of Proposition \ref{noisebd}]
		We have 
		$$\LL^{\e,n}_2(t) = \int_0^t e^{sA}\left(\M_s-\M_{\pi^n(s)}\right)  ds.$$	
		For any $t\geq 0$, we have \[\M_t-\M_{\pi^n(t)}=\int_{\pi^n(t)}^t e^{-sA}dW_s. \] 
		From \eqref{noiseint}, we have
		\[ e^{tA}(\M_t-\M_{\pi^n(t)})=W_t-e^{(t-\pi^n(t))A} W_{\pi^n(t)}+e^{tA}\int_{\pi^n(t)}^t e^{-sA}AW_s ds.\]
		By adding and subtracting $e^{(t-\pi^n(t))A}W_t$ on right hand side, we have 
		\begin{align}\label{noisedecomposition}
			|e^{tA}(\M_t-\M_{\pi^n(t)})|\leq \left|I-e^{(t-\pi^n(t))A}\right|&\sup_{0\leq s\leq t}|W_s|+e^{(t-\pi^n(t))|A|}|W_t-W_{\pi^n(t)}| \nonumber\\
			&+e^{t|A|} \left(\int_{\pi^n(t)}^t e^{-s|A|}|A|ds\right)\sup_{0\leq s\leq t}|W_s|.
		\end{align}
		Now, we will find estimates for each term  of right hand sides of \eqref{noisedecomposition}. Let us focus on first term.
		\begin{align}\label{nd1}
			\left|I-e^{(t-\pi^n(t))A}\right|=\left| \int_0^{(t-\pi^n(t))} e^{sA}|A| ds\right|\le(t-\pi^n(t))e^{(t-\pi^n(t))|A|}|A|\le\xi_{N_t^n+1}^ne^{(\xi_{N_t^n+1}^n)|A|}|A|. 
		\end{align}
		For the next term, a simple bound holds:
		\begin{align}\label{nd2}
			e^{(t-\pi^n(t))|A|}|W_t-W_{\pi^n(t)}|\leq e^{\xi_{N_t^n+1}^n|A|}|W_t-W_{\pi^n(t)}|. 
		\end{align}
		For the last term, we have
		\begin{align}\label{nd3}
			\int_{\pi^n(t)}^t e^{-s|A|}|A|ds=e^{t|A|}-e^{\pi^n(t)|A|}&=e^{t|A|}(1-e^{-(t-\pi^n(t))|A|})\nonumber\\&\leq e^{t|A|}(t-\pi^n(t))|A|\leq \xi_{N_t^n+1}^ne^{t|A|}|A|.
		\end{align}
		Using estimates \eqref{nd1}, \eqref{nd2} and \eqref{nd3} in \eqref{noisedecomposition}, we have 
		\begin{align}\label{4.16}
			|e^{tA}(M_t-M_{\pi^n(t)})|\leq \xi_{N_t^n+1}^ne^{(\xi_{N_t^n+1}^n)|A|}|A|\sup_{0\leq s\leq t}|W_s|&+e^{\xi_{N_t^n+1}^n|A|}|W_t-W_{\pi^n(t)}|\nonumber\\& + \xi_{N_t^n+1}^ne^{2t|A|}|A|\sup_{0\leq s\leq t}|W_s|.
		\end{align}
		Integrating with respect to time $0\leq s\leq t,$  
		\begin{align*}
			|\LL^{\e,n}_2(t)|\leq |A|\int_0^t \xi_{N_t^n+1}^ne^{(\xi_{N_s^n+1}^n)|A|}\sup_{0\leq r\leq s}|W_r| &ds+\int_0^te^{\xi_{N_s^n+1}^n|A|}|W_s-W_{\pi^n(s)}|ds\\&+ e^{2t|A|}|A|\int_0^t \xi_{N_s^n+1}^n \sup_{0\leq r\leq s}|W_r|ds.
		\end{align*}
		Taking the supremum over $t\in [0,T]$ and then the expectation,
		\begin{align*}
			\E\left[\sup_{0 \le t \le T} |\LL^{\e,n}_2(t)|\right]
			&\leq |A|\int_0^T \E\left[\xi_{N_s^n+1}^ne^{(\xi_{N_s^n+1}^n)|A|}\right] \E\left[\sup_{0\leq r\leq s}|W_r| \right]ds
			\\ &+\int_0^T\E\left[e^{\xi_{N_s^n+1}^n|A|}\right]\E\left[|W_s-W_{\pi^n(s)}|\right]ds\\ &+e^{2t|A|}|A|\int_0^T \E\left[\xi_{N_s^n+1}^n\right] \E\left[\sup_{0\leq r\leq s}|W_r|\right]ds\\
			&\leq \frac{C_{AT\xi_1}}{n}\left(\int_0^T \sqrt{s} ds
			+n\int_0^T \E\sqrt{(s-\pi^n(s))} ds+ \int_0^T \sqrt{s}ds\right)\\
			&\leq  C_{AT\xi_1}\E[\NN_{1/2}].
		\end{align*}
		Since $\E[\NN_{1/2}] \leq \frac{C_{\xi_1}}{\sqrt{n}}$ by Corollary \ref{Np} and so, we have
		\begin{align*}
			\E\left[\sup_{0 \le t \le T} |\LL^{\e,n}_2(t)|\right] \leq 
			\frac{C_{AT\xi_1}}{\sqrt{n}}.
		\end{align*}
	\end{proof}


	
	\section{Generalization to the Nonlinear Case}\label{section6}
	This section is motivated by the work presented in \cite{dhama2025asymptotic}, where the author has worked on fast periodic sampling perturbed by Poisson random measure. In the present section, with the help of Lemma \ref{G4decomposition}, we extend the result of \cite{dhama2025asymptotic} to a more general sampling setting by considering a nonlinear hybrid system subject to random sampling but driven by white noise.  More precisely, we consider that	
	\begin{align}\label{GSDE}
		dY^{\e,n}_t = c\left(Y^{\e,n}_{t}, Y^{\e,n}_{{\pi^n}}\right)dt + \e \sigma(Y^{\e,n}_{t}) dW_t 
		\quad Y^{\e,n}_0=y_0,
	\end{align}		
	where the functions $c: \R^d \times \R^d \to \R^d$ and $\sigma: \R^d \to \R^{d \times d}$ are assumed to be measurable and to satisfy appropriate regularity conditions. The process $W$  represent an independent Brownian motion. Observe that equation \eqref{GSDE} can naturally be interpreted as a random perturbation of an underlying nonlinear control system.
	\begin{equation}\label{nonlineareqn1}
		\dot{y}=c'(y,u), \quad y(0)= y_0 \in \R^d
	\end{equation}where $ c':\R^d\times \R^d \to \R^d$, with a feedback control law $u=\kappa(y)$, for some appropriate function $\kappa.$
	In this case, the drift function $c$ in \eqref{GSDE} is obtained from $c'$ 
	by absorbing the feedback control into the dynamics, namely, $c(y,z) := c'(y,\kappa(z)).$ 
	The argument $z = Y^{\e,n}_{\pi^n}$ in \eqref{GSDE}
	reflects a sample and hold implementation of the feedback control, 
	where the control is updated only at the random time instants as discussed in Section 2. 
	Therefore, in equation \eqref{GSDE}, the sampling effect is through the term $Y^{\e,n}_{{\pi^n}}$. Here, we also note that the dynamics of the random sampled counterpart of equation \eqref{nonlineareqn1} can be expressed as 
	\begin{equation}\label{GRODE}
		\dot{y}^{n}_t =  c(y_t^{n},y_{{\pi^n}}^{n}), \quad y_0^n= y_0 \in \R^d
	\end{equation} which is also a fully non linear equation. Our aim in this section is to analyze  hybrid system \eqref{GSDE} and to understand how the presence of random sampling and small external noise affect the behavior of solution. Now, we  state the hypotheses that will be used in the subsequent analysis.
	\begin{hypp}[Lipschitz continuity]\label{LipCont}
		There exists a positive constant $C$ such that for any $x_1, x_2, z_1, z_2 \in \R^d,$ we have
		\begin{equation*}
			\begin{aligned}
				|c(x_1,x_2)-c(z_1,z_2)|& \le C(|x_1-z_1|+|x_2-z_2|),~~~~\text{and}\\
				|\sigma(x_1)-\sigma(x_2)|&\le C |x_1-x_2|.
			\end{aligned}
		\end{equation*}			
	\end{hypp}
	From Hypothesis \ref{LipCont}, we observe that there exists a positive constant $C$ such that for any $x, z \in \R^d,$ we have
	\begin{equation}\label{growthofnoisecoefficeients}
		\begin{aligned}
			|c(x,z)| & \le C(1+|x|+|z|),~~~~\text{and}\quad 
			|\sigma(z)| &\le C(1+|z|).
		\end{aligned}			
	\end{equation}

	\begin{hypp}[Boundedness and linear growth of derivatives]\label{Derivative}
		For the vectors $x=(x_1,\cdots, x_n), y=(y_1,\cdots,y_n)\in \R^d$ and $c:\R^d \times \R^d \to \R^d,$ define
		$$D_1 c(x,y) := \left(\frac{\partial c_i}{\partial x_k}(x,y)\right)_{1\le i,k\le n}
		\quad \text{and} \quad
		D_2 c(x,y) := \left(\frac{\partial c_i}{\partial y_k}(x,y)\right)_{1\le i,k\le n},$$  
		be the Jacobian matrices of $f$ with respect to the first and second variables, respectively.
		Second-order derivatives are defined component wise by
		$$
		D_1^2c:=D_{11} c := D_1(D_1 c), \qquad
		D_{12} c := D_1(D_2 c), \qquad
		D_2^2c:=D_{22} c := D_2(D_2 c),
		$$
		and higher-order derivatives, such as $D_2^3 f$, are understood in the same manner. We have following assumptions. We assume that\begin{equation*}
			\begin{aligned}
				&\left|D_1c(x,y)\right|\le C(1+|y|),\quad \left|D_2c(x,y)\right| \le C,\quad\left|D_1D_2c(x,y)\right|\le C, \\
				& \left|D_1^2c(x,y)\right| \le C(1+|y|), \quad  and \quad \left|D_2^2c(x,y)\right|\le C.   \end{aligned}
		\end{equation*}
	\end{hypp}

	We establish the law of large numbers and the central limit theorem in the presence of random sampling for this setting also. The proof of the LLN follows the general strategy of the linear case, but requires additional arguments to control the contributions arising from the multiplicative noise term. For the CLT, we build on the framework introduced in \cite{dhama2025asymptotic}, suitably adapting it to accommodate the random sampling terms. In particular, to deal with the main complicated fluctuation term in the CLT ( Proposition \ref{MainTermApprox} ), we use Lemma~\ref{G4decomposition}, which simplifies the computations and makes them comparatively easier than the approach in \cite{dhama2025asymptotic}.
	
	We are now ready to establish our first main result of this section. But before that, we make a small yet important remark.
	\begin{rmkk}\label{rmkk1}
		We observe that by using an argument analogous to that in Lemma \ref{ODEbd} and applying   \eqref{growthofnoisecoefficeients}, one can get that
		\begin{gather}
			\sup_{ 0 \leq t \leq T}|y_t|<C_T,\quad \quad
			\sup_{ 0 \leq t \leq T}|y_t^n|<C_T.\label{boundonYandYn}
		\end{gather}
	\end{rmkk}
	\subsection{LLN  Result for Nonlinear Case}The first main result shows that when $\e$ is small and  $n$ is large, the stochastic process $Y^{\e,n}_t $ behaves deterministically. In particular,  $Y^{\e,n}_t \to y_t$ uniformly in $L^p(\Omega),~ p \geq 1$ in all the Regime.

	\begin{theorem}[Law of Large Numbers Type Result]\label{LLNExtension}
		Let $Y_t^{\e,n}$ and $y_t$ denote the respective solutions to equations \eqref{GSDE} and \eqref{nonlineareqn1}. Then, for any fixed $0<T<\infty$ and any $1\le p< \infty$, there exists a positive constant $C$, depending  on $T$ and $\xi_1$ only, such that for any $\e>0$ and $n \in \mathbb{N},$ we have
		\begin{equation*}
			\E\left[\sup_{0\le t \le T}|Y_t^{\e,n}-y_t|^p\right]\le 
			(\e^p + n^{-p} )C_{T\xi_1}.
		\end{equation*}
	\end{theorem}
	\begin{proof} 
		In order to tackle the multiplicative noise term, we need to first establish the  moment bounds for the process $Y_{t}^{\e,n}$. This can be established  easily for any $p\geq 1$ by using standard techniques for estimating moments of solutions to stochastic differential equations (see \cite{oksendal2003sde}) and linear growth conditions \eqref{growthofnoisecoefficeients}, yielding
		\begin{align}\label{Ynbound}	 
			\E\left[\sup_{0\le s \le T}|Y_s^{\e,n}|^p\right]\le C_{pT}, \quad \E\left[\sup_{t\in [0, T]}\left|\int_0^t \sigma(Y_{s}^{\e,n}) dW_s\right|^p\right] \leq C_{pT}.
		\end{align}
		
		We shall now proceed as in the proof of Theorem~\ref{LLN} by using the triangle inequality. That is, from equations \eqref{GSDE} and \eqref{GRODE}, for any $p\geq 1$ we have
		\begin{align*}
			|Y_t^{\e,n}-y_t^n|^p\leq & 2^{p-1} \left|\int_0^tc(Y_s^{\e,n},Y_{\pi^{n}(s)}^{\e,n})-c(y_s^n,y_s^n)ds\right|^p+2^{p-1}\e^p\left|\int_0^t \sigma(Y_s^{\e,n})dW_s\right| ^p.
		\end{align*}
		Using Hypothesis \ref{LipCont}, conditions \eqref{growthofnoisecoefficeients}, taking sup over time, expectation on both the sides, and using \eqref{Ynbound}, we get
		\begin{align}
			\E\sup_{t\in [0,T]}|Y_t^{\e,n}-y_t^n|^p\leq  C_p \E\sup_{t\in [0,T]}\int_0^t (1+|Y_s^{\e,n}-y_s^n|^p+|Y_{\pi^n(s)}^{\e,n}-y_s^n|^p )ds
			+\e^pC_{pT},\label{nl4}.
		\end{align}
		Using the fact that $\pi^n(t) \leq t,$ we have $$\E \left[\sup_{ 0 \leq t \leq T}|Y_{\pi^{n}(t)}^{\e,n}-y_t^n|\right]\leq \E \left[\sup_{ 0 \leq t \leq T}|Y_{t}^{\e,n}-y_t^n|\right]. $$ Finally, applying Gronwall's inequality to \eqref{nl4} 
		\begin{align}\label{ApproxG2}
			\E\left[\sup_{ 0 \leq t \leq T}|Y_t^{\e,n}-y_t^n|^p\right]\leq \e^p C_{pT}.
		\end{align}
		Next,  using Hypothesis \ref{LipCont} we have for any $1\leq p< \infty$ that
		\begin{align*}
			|y_t^n-y_t|^p\leq& \left|\int_0^tc(y_s^n,y_{\pi^{n}(s)}^n)-c(y_s,y_s)ds\right|^p\leq C_{Tp} \int_0^t\left(|V_s|^p+(s-\pi^n(s))^pC_T\right)ds,
		\end{align*}
		where $ V_s=\sup_{0\leq r\le s}|y_s^n-y_s|$.~Taking the supremum and expectation on both the sides, then applying the Gronwall's inequality, for any $1\leq p< \infty,$ we get
		\begin{align}\label{ApproxG1}
			\E[\sup_{ 0 \leq t \leq T}|y_t^n-y_t|^p]\leq \E[\NN_p] C_T.
		\end{align}
		Combining estimates \eqref{ApproxG2} and \eqref{ApproxG1} using the triangle inequality,
		\begin{align*}
			\E\left[\sup_{ 0 \leq t \leq T}|Y_t^{\e,n}-y_t^n|^p\right]&\leq C_p\E\left[\sup_{ 0 \leq t \leq T}|Y_t^{\e,n}-y_t^n|^p\right]+C_p\E[\sup_{ 0 \leq t \leq T}|y_t^n-y_t|^p]\\&\leq C_{Tp}(\E([\NN_p])+\e^p)\le (\e^p+n^{-p})C_{Tp\xi_1}
		\end{align*}
		by Corollary \ref{Np}, for any $1\leq p< \infty$.
	\end{proof}
	

	\subsection{CLT Result for Nonlinear Case}This subsection is devoted to establishing CLT-type results for the generalized setting. Our analysis is covering Regimes 1 and 2, while the corresponding result for Regime 3 can be obtained by following the similar arguments as in the linear case.
	In Regimes 1 and 2, let us define the rescaled fluctuation process
	\begin{equation*}\label{fluctuation}
		\mathsf{Z}^{\e,n}_t := \frac{Y^{\e,n}_t - y_t}{\e}. 
	\end{equation*}	Here, we note that the coarser parameter $\e$ is used to rescale the stochastic quantity $(Y_t^{\e,n}-y_t).$ To get more insight into the rescaled process $\mathsf{Z}^{\e,n}_t,$ we get by \eqref{GSDE} and \eqref{nonlineareqn1} that
	\begin{equation}\label{fluctuation1}
		\mathsf{Z}^{\e,n}_t=\frac1\e\int_0^t \left\{c\left(Y^{\e,n}_{s}, Y^{\e,n}_{{\pi^n(s)}}\right)-c(y_s,y_s)\right\} ds + \int_0^t \sigma(Y^{\e,n}_{s})dW_s.
	\end{equation}
	Applying Taylor's theorem (see \cite{edwards2012advanced}), we obtain
	\begin{align*}
		\mathsf{Z}^{\e,n}_t=\int_0^t \left\{D_1c(y_s,y_s)+D_2c(y_s,y_s)\right\}\mathsf{Z}^{\e,n}_{s} ds + \int_0^t D_2c(y_s,y_s)&\left(\frac{Y^{\e,n}_{{\pi^n(s)}}-Y_{s}^{\e,n}}{\e} \right)ds\nonumber\\&+ \int_0^t \sigma(Y^{\e,n}_{s})dW_s+{\mathrm R}_t^{\e,n},
	\end{align*}	
	\text{where}
	\begin{align}
		{\mathrm R}_t^{\e,n}:= \int_0^t \Bigg[\frac{c\left(Y^{\e,n}_{s}, Y^{\e,n}_{{\pi^n(s)}}\right)-c(y_s,y_s)}{\e} - D_1c(y_s,&y_s)\mathsf{Z}^{\e,n}_{s} - D_2c(y_s,y_s)\mathsf{Z}^{\e,n}_{s}\nonumber \\& -D_2c(y_s,y_s)\left(\frac{Y^{\e,n}_{{\pi^n(s)}}-Y_{s}^{\e,n}}{\e} \right)\Bigg]ds. \label{Rten}
	\end{align}
	Our goal is to describe the limiting behavior of the fluctuation process $\mathsf{Z}^{\e,n}_t$ as $\e\searrow 0 $ and $n\to \infty.$ For this purpose, we define a function
	\begin{align}\label{ellg}
		\ell_g(t):={\cc M}\int_0^t D_2c(y_s,y_s)\cdot c(y_s,y_s)  ds.
	\end{align}
	Suppose, we are able to show that ${\mathrm R}^{\e,n}_t= \mathcal{O}(\e^2 + n^{-2})$, and
	\begin{align*}
		\int_0^t D_2c(y_s,y_s)\left(\frac{Y^{\e,n}_{{\pi^n(s)}}-Y_{s}^{\e,n}}{\e} \right)ds\to \ell_g(t)
	\end{align*}
	as $\e\searrow 0$ and $n\to \infty.$	 Then the process $\mathsf{Z}^{\e,n}_t$ converges to a limiting process  $\mathsf{Z}=\{\mathsf{Z}_t: t \ge 0\}$ which is uniquely defined as the solution of the stochastic differential equation given below 
	\begin{align}\label{Zsoln}
		\mathsf{Z}_t:= \!\int_0^t\{D_1c(y_s,y_s)+D_2c(y_s,y_s)\} \mathsf{Z}_{s}  ds +{\cc}M \int_0^t D_2c(y_s,y_s)\cdot c(y_s,y_s)  ds \!+ \!\!\int_0^t \sigma(y_s) dW_s.
	\end{align}
	This argument is the central focus of this subsection. We analyze this approximation in Regimes 1 and 2 in the following theorem, which  constitutes the second main result of this section.
	\begin{theorem}[Central Limit Theorem Type Result]\label{CLTExtension2}
		Let $y_t$ and $Y_t^{\e,n}$ denote the solutions of \eqref{nonlineareqn1} and \eqref{GSDE}, respectively. Furthermore, let $\mathsf{Z}^{\e,n}_t$ and $\mathsf{Z}_t$ be defined by \eqref{fluctuation1} and \eqref{Zsoln}, respectively. Suppose that we are in Regime $i \in \{1,2\}$, i.e., $\displaystyle{\lim_{\e \searrow 0, n\to \infty}1/(n\e) = \cc \in [0,\infty)}$. Then, for any fixed $T \in (0,\infty)$, there exists a positive constant $C$ independent of $\e$ and $n$, and there exists $\e_0>0$ such that for $0<\e<\e_0$, we have
		\begin{equation*}\label{FCLT}
			\E\left[\sup_{0 \le t \le T} |\mathsf{Z}^{\e,n}_t - \mathsf{Z}_t|\right] 
			\le  [\cc (n^{-1/4}+\e)+ \varkappa(\e)+ \e+n^{-1/2}]C_{MT\xi_1}.
		\end{equation*}
	\end{theorem}
	The proof of Theorem~\ref{CLTExtension2} is based on a collection of intermediate results. We begin by presenting  the necessary propositions and lemmas. Once these are stated, we combine them to conclude the proof of the theorem at the end of this subsection. The proof of those intermediate results will be provided in next subsection.
	\begin{proposition}\label{sigmadifference}		
		We have 
		\begin{gather*}
			\E\left[\sup_{0 \le t \le T}\left|\int_0^t \{\sigma(Y_{s}^{\e,n})-\sigma(y_{s})\} dW_s\right|\right] \le C_T[\E[\NN_2]+\e^2]^\frac12\le [\e+n^{-1}]C_{\xi_1 T}.
		\end{gather*}
	\end{proposition}
	
	\begin{proposition}\label{RemainderTerms}
		Let ${\mathrm R}_t^{\e,n}$ be defined as in \eqref{Rten}.
		Then, for any fixed $T>0,$ there exists $C_T>0$ such that
		\begin{equation*}
			\E\left[\sup_{0\le t \le T}|{\mathrm R}_t^{\e,n}|\right] \le C_T(\e^{4}+ \E[\NN_4])^{\frac12}\leq (\e^2+n^{-2})C_{T\xi_1}. 
		\end{equation*}
	\end{proposition}

	\begin{proposition}\label{MainTermApprox}
		Let $y_t$ and $Y_t^{\e,n}$  solve \eqref{nonlineareqn1} and \eqref{GSDE}, respectively. Then, for any fixed $T>0,$ there exists a positive constant $C_T$ such that for any $0< \e <\e_0$ 
		\begin{align*}
			\E\left[ \sup_{0 \le t \le T}\Big|\int_0^t D_2c(y_s,y_s)\frac{Y_{s}^{\e, n}-Y_{\pi^n(s)}^{\e, n}}{\e} ds - \ell_g(t)\Big| \right] \le [\cc (n^{-1/4}+\e)+ \varkappa(\e)+n^{-1/2}]C_{MT\xi_1}
		\end{align*}
		which converges to zero as $\e\searrow 0$ and $n\to \infty$ with the rate depending on the value of $\cc\in [0, \infty).$
	\end{proposition}
	We establish Proposition \ref{MainTermApprox} by decomposing the proof into the following sequence of intermediate lemmas, each addressing a key component of the argument.	
	\begin{lemma}\label{Mterms}
		Let $Y_t^{\e,n}$ be the solution of equation \eqref{GSDE}. Then, for any fixed $T>0$, $t\in[0,T],$ and $\e,n>0,$ we have 
		\begin{equation}\label{M1234}
			\int_0^t  D_2c(y_s,y_s) \frac{Y_{s}^{\e, n}-Y_{\pi^n(s)}^{\e, n}}{\e} ds = \sum_{i=1}^{3}{\J}_i^{\e,n}(t),\quad \text{where}
		\end{equation} 
		\begin{equation*}
			\begin{aligned}
				{\J}_1^{\e,n} & :=   \int_0^t  D_2c(y_s,y_s) \int_{\pi^n(s)}^s \frac{c(Y_{r}^{\e,n}, Y_{\pi^n(r)}^{\e, n})-c(Y_{\pi^n(r)}^{\e, n}, Y_{\pi^n(r)}^{\e, n})}{\e}  dr  ds\\&\hspace{3cm}
				+ \int_0^t  D_2c(y_s,y_s) \int_{\pi^n(s)}^s \frac{c(Y_{\pi^n(r)}^{\e, n}, Y_{\pi^n(r)}^{\e, n})-c(y_{\pi^n(r)}, y_{\pi^n(r)})}{\e}  dr  ds, \\
				{\J}_2^{\e,n}   & := \int_0^t  D_2c(y_s,y_s) \int_{\pi^n(s)}^s \frac{c(y_{\pi^n(r)}, y_{\pi^n(r)})}{\e}  dr   ds, \\ 
				{\J}_3^{\e,n}  & :=  \int_0^t  D_2c(y_s,y_s) \int_{\pi^n(s)}^s \sigma(Y_{r}^{\e,n}) dW_r  ds.
			\end{aligned}
		\end{equation*}
	\end{lemma}
	Now, our next step is to  show that the terms $\E[\sup_{0 \le t \le T}|{\J}_1^{\e,n}(t)|],\quad \sup_{0 \le t \le T}|{\J}_2^{\e,n}(t)-\ell_g(t)|$ and  $\E[\sup_{0 \le t \le T}|{\J}_3^{\e,n}(t)|]$   are small. 
	
	\begin{lemma}\label{M1}
		Let ${\J}_1^{\e,n}(t)$ be defined as in equation \eqref{M1234}. Then, for any fixed $T>0,$ there exists a positive constant $C_{T\xi_1}$ such that for any $ \e \in (0,\e_0),$  we have
		\begin{align*}
			\E\left[\sup_{0 \le t \le T}\left|{\J}_1^{\e,n}(t)\right|\right]  \le\cc\left[\e+n^{-1}\right]C_{T\xi_1}.
		\end{align*}	
	\end{lemma}		
	We next decompose ${\J}_2^{\e,n}(t)$ defined in \eqref{M1234} as follows.
	
	\begin{equation}\label{M2decomposition}
		\begin{aligned}
			{\J}_2^{\e,n}(t) &= {\I}_1^{\e,n}(t)+ {\I}_2^{\e,n}(t), \quad \text{where}\\
			{\I}_1^{\e,n}(t) &:= \int_0^t\{D_2c(y_s,y_s)-D_2c(Y_{\pi^n(s)}, Y_{\pi^n(s)})\}\int_{\pi^n(s)}^s \frac{c(y_{\pi^n(r)}, y_{\pi^n(r)})}{\e}  dr  ds, \\
			{\I}_2^{\e,n}(t)& := \int_0^t D_2c(y_{\pi^n(r)}, y_{\pi^n(r)})\int_{\pi^n(s)}^s \frac{c(Y_{\pi^n(s)}, Y_{\pi^n(s)})}{\e}  dr  ds.
		\end{aligned}
	\end{equation}

	\begin{lemma}\label{M2M1}
		Let ${\I}_1^{\e,n}(t)$ be defined as in equation \eqref{M2decomposition}. Then, 
		\begin{equation*}
			\E\left[\sup_{0 \le t \le T}\left|{ \I}_1^{\e,n}(t)\right|\right]\le
			\cc[\e+n^{-1}]C_{T\xi_1}.
		\end{equation*}
	\end{lemma}
	
	\begin{lemma}\label{M2M2}
		Let ${\I}_2^{\e,n}(t)$ be defined as in equation \eqref{M2decomposition} with $\ell_g(t)$ as given in \eqref{ellg}. Then, for any $T>0,$ there exists a positive constant $C_T>0$ such that for any $0<\e< \e_0$, we have 
		\begin{equation*}
			\sup_{0 \le t \le T}\left|{\I}_2^{\e,n}(t)-\ell_g(t)\right| \le [\cc (n^{-1/4}+\e)+ \varkappa(\e)]	C_{TM\xi_1}.
		\end{equation*}
	\end{lemma}
	
	\begin{lemma}\label{M3}
		Let ${\J}_3^{\e,n}(t)$ be defined as in equation \eqref{M1234}. Then, for any fixed $T>0,$ there exists a positive constant $C_T$ such that for any $\e,n>0$, we have 
		\begin{equation*}
			\E\left[\sup_{0 \le t \le T}\left|{\J}_3^{\e,n}(t)\right|\right]		
			\le n^{-1/2}C_{T\xi_1}.
		\end{equation*}
	\end{lemma}

	\begin{proof}[Proof of Theorem \ref{CLTExtension2}] 
		By combining the conclusions of Propositions~\ref{sigmadifference}–\ref{MainTermApprox}, and using the estimates established therein, we obtain the desired result.
	\end{proof}
	

	
	\subsubsection{Proofs of Propostions \ref{sigmadifference}-\ref{MainTermApprox} and Lemmas \ref{Mterms}-\ref{M3}.}
	
	\begin{proof}[Proof of Proposition \ref{sigmadifference}] Using C-S inequality, BDG inequality, Hypothesis~\ref{LipCont}  and Theorem \ref{LLNExtension}, we get
		\begin{align*}
			\E\left[\sup_{0 \le t \le T}\left|\int_0^t \{\sigma(Y_{s}^{\e,n})-\sigma(y_{s})\} dW_s\right|\right] \le    C\left[\E \sup_{0 \le t \le T }|Y_{t}^{\e,n}-y_t|^2\right]^\frac12\le& C_T \left[\E[\NN_2]+\e^2\right]^\frac12\\
			&\le [\e+n^{-1}]C_{T\xi_1}.
		\end{align*}			
	\end{proof}
	\begin{proof}[Proof of Proposition \ref{RemainderTerms}]
		The proof can be obtained by  directly following \cite[Proposition 4.4]{dhama2025asymptotic} by using Taylor's formula and Hypotheses \ref{Derivative}.
	\end{proof}
	
	\begin{proof}[Proof of Proposition \ref{MainTermApprox}] By putting together the results proved in Lemmas \ref{Mterms}–\ref{M3}, we obtain the required result.          
	\end{proof}
	
	\begin{proof}[Proof of Lemma \ref{Mterms}]
		From \eqref{GSDE} we have
		\begin{align*}			
			\int_0^t D_2c(y_s,y_s)\frac{Y_{s}^{\e, n}-Y_{\pi^n(s)}^{\e, n}}{\e} ds  =\int_0^t D_2c(y_s,y_s) &\int_{\pi^n(s)}^s \frac{c(Y_{r}^{\e,n}, Y_{\pi^n(r)}^{\e, n})}{\e}  dr  ds \\& + \int_0^t D_2c(y_s,y_s) \int_{\pi^n(s)}^s \sigma(Y_{r}^{\e,n}) dW_r  ds.
		\end{align*}
		Writing \begin{multline*}c(Y_{r}^{\e,n}, Y_{\pi^n(r)}^{\e, n})=c(Y_{r}^{\e,n}, Y_{\pi^n(r)}^{\e, n})-c(Y_{\pi^n(r)}^{\e, n}, Y_{\pi^n(r)}^{\e, n})\\+c(Y_{\pi^n(r)}^{\e, n}, Y_{\pi^n(r)}^{\e, n}) +c(y_{\pi^n(r)}, y_{\pi^n(r)})-c(y_{\pi^n(r)}, y_{\pi^n(r)}) \end{multline*} in the second term of the right hand side of the above equation, we get
		\begin{equation*}
			\begin{aligned}
				\frac 1\e\int_0^t  D_2c(y_s,y_s) \big[{Y_{s}^{\e, n}-Y_{\pi^n(s)}^{\e, n}}\big] ds &= \int_0^t D_2c(y_s,y_s) \int_{\pi^n(s)}^s \frac{c(Y_{r}^{\e,n}, Y_{\pi^n(r)}^{\e, n})-c(Y_{\pi^n(r)}^{\e, n}, Y_{\pi^n(r)}^{\e, n})}{\e}  dr  ds\\
				&  +\int_0^t D_2c(y_s,y_s) \int_{\pi^n(s)}^s \frac{c(Y_{\pi^n(r)}^{\e, n}, Y_{\pi^n(r)}^{\e, n})-c(y_{\pi^n(r)}, y_{\pi^n(r)})}{\e}  dr  ds \\
				& - \int_0^t D_2c(y_s,y_s) \int_{\pi^n(s)}^s \frac{c(y_{\pi^n(r)}, y_{\pi^n(r)})}{\e}  dr  ds\\
				&+ \int_0^t D_2c(y_s,y_s) \int_{\pi^n(s)}^s \sigma(Y_{r}^{\e,n}) dW_r  ds.
			\end{aligned}
		\end{equation*}
		The right hand side of the above equation is easily recognized as the sum of ${\J}_i^{\e,n}(t), 1 \le i \le 3.$
	\end{proof}

	\begin{proof}[Proof of Lemma \ref{M1}]
		Recalling the definition of ${\J}_1^{\e,n}(t)$ from \eqref{M1234}, we have
		\begin{align*}
			{\J}_1^{\e,n}(t) &= \int_0^t D_2c(y_s,y_s) \int_{\pi^n(s)}^s \frac{c(Y_{r}^{\e,n}, Y_{\pi^n(r)}^{\e, n})-c(Y_{\pi^n(r)}^{\e, n}, Y_{\pi^n(r)}^{\e, n})}{\e}  dr  ds\\
			& \hspace{1cm} + \int_0^t D_2c(y_s,y_s) \int_{\pi^n(s)}^s \frac{c(Y_{\pi^n(r)}^{\e, n}, Y_{\pi^n(r)}^{\e, n})-c(y_{\pi^n(r)}, y_{\pi^n(r)})}{\e}  dr  ds\\& := J_1^{\e,n}(t) + J_2^{\e,n}(t).
		\end{align*}		For $J_1^{\e,n}(t),$ using Hypothesis \ref{LipCont}, we have
		\begin{equation*}
			\begin{aligned}
				|J_1^{\e,n}(t)| 
				& \lesssim \frac{1}{\e}\int_0^t\int_{\pi^n(s)}^{s}\sup_{0\le r \le s}|Y_r^{\e,n}-y_r+y_r-y_{\pi^n(r)}+y_{\pi^n(r)}-Y_{\pi^n(r)}^{\e,n}| dr  ds\\
				& \lesssim \frac{1}{\e}\int_0^t\int_{\pi^n(s)}^{s}\sup_{0\le r \le s}|Y_r^{\e,n}-y_r| dr  ds + \frac{1}{\e}\int_0^t\int_{\pi^n(s)}^{s}\sup_{0\le r \le s}|y_r-y_{\pi^n(r)}| dr  ds.
			\end{aligned}
		\end{equation*}
		Taking supremum on both side over time and then taking expectation on both sides, we get
		\begin{align*}
			\E\left[\sup_{0\le t \le T}|J_1^{\e,n}(t)|\right] \lesssim \frac{1}{\e}\E\left[\int_0^T \sup_{0\le r \le s}\left\{|Y_r^{\e,n}-y_r|+|y_r-y_{\pi^n(r)}|\right\} (s-\pi^n(s))   ds \right].				
		\end{align*}
		Applying C-S inequality, we get
		\begin{align*}
			\E\left[\sup_{0\le t \le T}|J_1^{\e,n}(t)|\right] \lesssim & \frac{1}{\e}\left[\E\sup_{0\le r \le s}|Y_r^{\e,n}-y_r|^2\right]^\frac12\left[\E\int_0^T  (s-\pi^n(s))^2   ds \right]^\frac12  \\
			&+\frac1\e\left[\E\sup_{0\le r \le s}|y_r-y_{\pi^n(r)}|^2\right]^\frac12 \left[\E \int_0^T (s-\pi^n(s))^2  dr  ds \right]^\frac12\\
			&\leq \frac{1}{\e}\left[(\e^2+\NN_2)^{\frac12}+(\NN_2)^{\frac12}\right]\left(\E[\NN_2]\right)^{\frac12}C_T\\
			&\le \cc\left[\e+n^{-1}\right]C_{T\xi_1}.	\end{align*}		
		Similarly, for $J_2^{\e,n}(t),$ by similar calculations to those above, we obtain 
		\begin{equation*}
			\E\left[\sup_{0\le t \le T}|J_2^{\e,n}(t)|\right] \le\cc\left[\e+n^{-1}\right]C_{T\xi_1}.
		\end{equation*}
	\end{proof}		      
	Before proving Lemma \ref{M2M1} and \ref{M2M2}, we need the following  fluctuation bound which can be establish by using mean value theorem and Hypothesis  \ref{Derivative}.
	\begin{lemma}\label{MVTresult} Let the function $c$ satisfies the  Hypothesis  \ref{Derivative}. Then,  there exists a constant $K>0$, such that  
		\begin{equation*}
			\left|D_2c(Y_{\pi^n(s)}, Y_{\pi^n(s)})-D_2c(y_s,y_s)\right|\leq K|Y_{\pi^n(s)}-y_s|.
		\end{equation*}
	\end{lemma} 
	
	\ignore{\begin{proof}We will prove it as following. Define the mapping
			$$
			\Phi:\mathbb{R}^n\to\mathbb{R}^n, \qquad 
			\Phi(z):=D_{2}c(z,z).$$
			To estimate $\Phi(Y_{\pi^n(s)})-\Phi(y_s)$, we apply the multivariate Mean Value
			Theorem which says that there exists
			$\eta_s = y_s + \theta\,(Y_{\pi^n(s)}-y_s)$ for some  $\theta\in(0,1),$
			such that
			$$
			\Phi(Y_{\pi^n(s)})-\Phi(y_s)
			= D\Phi(\eta_s)\,(Y_{\pi^n(s)}-y_s),
			$$
			where $D$ is total derivative. Hence,
			\begin{align}
				\label{MVT}
				\bigl|\Phi(Y_{\pi^n(s)})-\Phi(y_s)\bigr|
				\le |D\Phi(\eta_s)|\; |Y_{\pi^n(s)}-y_s|.
			\end{align}
			Since $\Phi(z)=D_2c(z,z)$, the chain rule gives
			$$
			D\Phi(z)=D_1D_2c(z,z)+D_2D_2c(z,z).
			$$
			By Hypothesis~\ref{Derivative},
			$$
			|D_1D_2c(x,y)|\le C,
			\qquad
			|D_2D_2c(x,y)|\le C.
			$$
			Therefore,
			$$
			|D\Phi(z)|\le C+C = 2C,
			\qquad \forall\, z\in\mathbb{R}^n.
			$$	
			Substituting this into the MVT estimate \eqref{MVT}, yields
			$$
			\bigl|\Phi(Y_{\pi^n(s)})-\Phi(y_s)\bigr|
			\le 2C\, |Y_{\pi^n(s)}-y_s|,
			$$
			and therefore
			$$
			\bigl|D_2c(Y_{\pi^n(s)},Y_{\pi^n(s)}) - D_2c(y_s,y_s)\bigr|
			\le 2C\,|Y_{\pi^n(s)}-y_s|.$$					
	\end{proof}	}		
	
	
	\begin{proof}[Proof of Lemma \ref{M2M1}] We have 		
		
		$${\I}_1^{\e,n}(t) = \int_0^t\{D_2c(y_s,y_s)-D_2c(Y_{\pi^n(s)}, Y_{\pi^n(s)})\}\int_{\pi^n(s)}^s \frac{c(y_{\pi^n(r)}, y_{\pi^n(r)})}{\e}  dr  ds.	$$
		Applying C-S inequality repeatedly, we get for any $t \in [0,T]$ that
		\begin{align*}
			|{\I}_1^{\e,n}(t)|\leq& \frac1\e \left(\int_0^t |D_2c(y_s,y_s)-D_2c(Y_{\pi^n(s)}, Y_{\pi^n(s)})|^2 ds\right)^{\frac12} \nonumber\\
			&\hspace{1cm}\times\left(\int_0^t |c(y_{\pi^n(r)},y_{\pi^n(r)})|^4 ds\right)^\frac14\left(\int_0^t (s-\pi^n(s))^4ds\right)^\frac14
			\nonumber\\
			&\leq\frac1\e \left(\int_0^T |Y_{\pi^n(s)}-y_s|^2 ds\right)^{\frac12}\left(\int_0^T (1+|y_{\pi^n(r)}|^4) ds\right)^\frac14 \left(\int_0^T (s-\pi^n(s))^4ds\right)^\frac14,
		\end{align*}
		by using Lemma \ref{MVTresult} and \eqref{growthofnoisecoefficeients}. Taking the supremum over $t \in [0,T]$ on the left, and taking expectation the on both sides, using Jenson inequality, applying Theorem \ref{LLNExtension}, we get
		\begin{equation*}
			\E[\sup_{ 0 \leq t \leq T}|{\I}_1^{\e,n}(t)|]\leq \frac1\e (\e^2+\E[\NN_2])^\frac12 \left(\E[\NN_4]\right)^\frac14 C_T
			= \cc(\e+n^{-1})C_{\xi_1T}.
		\end{equation*}		
	\end{proof}	
	\begin{proof}[Proof of Lemma \ref{M2M2}] 	Recalling the definition of $\ell_g(t)$ and  ${ \I}_2^{\e,n}(t)$  from Definition \ref{ellg} and   equations\eqref{M2decomposition}, respectively, we have
		\begin{align*}
			&{\I}_2^{\e,n}(t)-\ell_g(t) \\&= \int_0^t D_2c(Y_{\pi^n(s)}, Y_{\pi^n(s)})\int_{\pi^n(s)}^s \frac{c(y_{\pi^n(r)}, y_{\pi^n(r)})}{\e}  dr  ds
			-{\cc M}\int_0^t D_2c(y_s,y_s)\cdot c(y_s,y_s)  ds\\
			&=\left(\frac{1}{n\e}-\cc\right)\int_0^tD_2c(Y_{\pi^n(s)}, Y_{\pi^n(s)})\cdot c(y_{\pi^n(r)}, y_{\pi^n(r)}) \frac{(s-\pi^n(s))}{1/n}ds\\
			&+\cc \int_0^t \left(D_2c(Y_{\pi^n(s)}, Y_{\pi^n(s)})\cdot{c(y_{\pi^n(s)}, y_{\pi^n(s)})}-D_2c(y_s,y_s)\cdot c(y_s,y_s)\right) \frac{(s-\pi^n(s))}{1/n}ds\\
			&+\cc\int_0^t \left(\frac{(s-\pi^n(s))}{1/n}-M\right)D_2c(y_s,y_s)\cdot c(y_s,y_s) ds. 
		\end{align*}
		Therefore,			\begin{align*}
			\big|{\I}_2^{\e,n}(t)&-\ell_g(t)\big| \\
			&\leq \left|\frac{1}{n\e}-\cc\right|\left|\int_0^t \left(D_2c(Y_{\pi^n(s)}, Y_{\pi^n(s)})\cdot{c(y_{\pi^n(s)}, y_{\pi^n(s)})}\right)\frac{(s-\pi^n(s))}{1/n}ds\right|
			\\& +\cc\left|\int_0^t \left(D_2c(Y_{\pi^n(s)}, Y_{\pi^n(s)})\cdot{c(y_{\pi^n(s)}, y_{\pi^n(s)})}-D_2c(y_s,y_s)\cdot c(y_s,y_s)\right) \frac{(s-\pi^n(s))}{1/n}ds\right|
			\\&+\cc \left|\int_0^t \left(\frac{(s-\pi^n(s))}{1/n}-M\right)D_2c(y_s,y_s)\cdot c(y_s,y_s)\right|\\&
			:=I_2^1+I_2^2+I_2^3.
		\end{align*}
		Noting that $\varkappa(\e)=\left|\frac{1}{n\e}-\cc\right|$, we have
		\begin{align*}
			I_2^1=&	\varkappa(\e)\left|\int_0^t \left(D_2c(Y_{\pi^n(s)}, Y_{\pi^n(s)})\cdot{c(y_{\pi^n(s)}, y_{\pi^n(s)})}\right)\frac{(s-\pi^n(s))}{1/n}ds\right|\\&\leq \varkappa(\e) K\int_0^t(1+|y_{\pi^n(s)}|)\frac{(s-\pi^n(s))}{1/n}ds,
		\end{align*}
		which gives
		\begin{align*}
			\E\left[\sup_{ 0\leq t \leq T}|I_2^1(t)|\right]&\leq \varkappa(\e) nK \left(\E\int_0^T (s-\pi^n(s))^2ds\right)^{\frac12}C_T=\varkappa(\e) nK\left(\E\NN_2\right)^{\frac12}C_T \le \varkappa(\e)C_{T\xi_1}.
		\end{align*}			
		Now, by C-S inequality, we obtain	
		\begin{align*}
			\E&\left[\sup_{ 0 \leq t \leq T}|I_2^2|\right]\\
			&\leq n\cc \left(\E\int_0^T \left(D_2c(Y_{\pi^n(s)}, Y_{\pi^n(s)})\cdot{c(y_{\pi^n(s)}, y_{\pi^n(s)})}-D_2c(y_s,y_s)\cdot c(y_s,y_s)\right)^2ds\right)^{\frac12}\\	&\hspace{9cm}\times\left(\E\int_0^T  (s-\pi^n(s))^2ds\right)^{\frac12}.
		\end{align*}
		The last integral is $\E\left(\NN_2\right)^\frac12.$ Now, let us focus on the first integral. We can write 
		\begin{align*}
			\E\int_0^T &\left|D_2c(Y_{\pi^n(s)}, Y_{\pi^n(s)})\cdot{c(y_{\pi^n(s)}, y_{\pi^n(s)})}-D_2c(y_s,y_s)\cdot c(y_s,y_s)\right|^2ds\\&\lesssim \E\int_0^T \left|D_2c(Y_{\pi^n(s)}, Y_{\pi^n(s)})\cdot\left({c(y_{\pi^n(s)}, y_{\pi^n(s)})}-c(y_s,y_s)\right)\right|^2ds\\
			&\hspace{3.5cm}+\E\int_0^T \left|\left(D_2c(Y_{\pi^n(s)}, Y_{\pi^n(s)})-D_2c(y_s,y_s)\right)\cdot c(y_s,y_s)\right|^2ds\\
			&\leq K	\E\int_0^T |y_{\pi^n(s)}-y_s|^2ds+\left[\E\int_0^T \left|Y_{\pi^n(s)}-y_s\right|^4ds \right]^{\frac12} \left[\E\int_0^T \left(1+|y_s|^4\right) ds\right]^{\frac12},						\end{align*}
		where we have used the Hypotheses \ref{LipCont},~\ref{Derivative}, and Lemma \ref{MVTresult}. Thus 
		\begin{equation*}
			\E\left[\sup_{ 0 \leq t \leq T}|I_2^2|\right]\leq  \frac1\e\left[\NN_2+(\e^4+\NN_4)^{\frac12}\right]^\frac12\left({\E[\NN_2]}\right)^\frac12C_T\le \cc [n^{-1}+\e]C_{T\xi_1}.		
		\end{equation*}
		Next, we have 
		\begin{align}\label{M22I3}
			\E\left[\sup_{ 0 \leq t \leq T}|I_2^3|\right]= 	\cc \E\sup_{ 0 \leq t \leq T}\left|\int_0^t \left(\frac{(s-\pi^n(s))}{1/n}-M\right)D_2c(y_s,y_s)\cdot c(y_s,y_s) ds \right|.\nonumber
		\end{align}
		
		By Hypothesis \ref{Derivative}, the term $D_2c(y_s,y_s)\cdot c(y_s,y_s)$ is continuous. Moreover, by \eqref{growthofnoisecoefficeients} together with estimate \eqref{boundonYandYn}, we have $|c(y_s,y_s)|\leq C_T(1+|y_s|)\leq C_T$, and from Hypothesis \ref{Derivative} we also obtain $|D_2c(y_s,y_s)|<K$. Hence the product $D_2c(y_s,y_s)\cdot c(y_s,y_s)$ is not only continuous but also bounded. Therefore, there exists a mollifier $\rho_\e(y_s)$ such that
		$
		|D_2c(y_s,y_s)\cdot c(y_s,y_s)-\rho_\e(y_s)|\leq C\e.
		$
		Now, take $f$ to be a function satisfying the assumptions of Lemma \ref{G4decomposition}. Then we have
		\begin{align*}
			\E\sup_{ 0 \leq t \leq T}&\bigg|\int_0^t \left(\frac{(s-\pi^n(s))}{1/n}-M\right)D_2c(y_s,y_s)\cdot c(y_s,y_s) ds \bigg|\\
			&\le\E\int_0^T \left| \left(\frac{(s-\pi^n(s))}{1/n}-M\right)(D_2c(y_s,y_s)\cdot c(y_s,y_s)-\rho_\e(y_s))\right| ds\\
			&\hspace{3.5cm}+\E\int_0^T \left| \left(\frac{(s-\pi^n(s))}{1/n}-M\right)(\rho_\e(y_s)-f(y_s))\right| ds\\
			&\hspace{5.5cm}+\E\sup_{ 0 \leq t \leq T}\left|\int_0^t \left(\frac{(s-\pi^n(s))}{1/n}-M\right)f(y_s) ds\right|\\
			&\leq \sup_{s\in [0,T]}(|D_2c(y_s,y_s)\cdot c(y_s,y_s)-\rho_\e(y_s)|+|\rho_\e(y_s)-f(y_s)|)\\
			&\hspace{5.7cm}\times\int_0^T\E|n(s-\pi^n(s))-M|ds+n^{-1/4}C_{\xi_1T}\\
			&\leq C\e\int_0^T\E(\xi_{N_s^n+1}+M) ds+n^{-1/4}C_{\xi_1MT}\le (\e+n^{-1/4})C_{\xi_1 M T},
		\end{align*}
		where in last second inequality, we have used Lemma \ref{G4decomposition}. So, finally, we have
		\begin{align*}
			\E\left[\sup_{ 0 \leq t \leq T}|I_2^3|\right]\le \cc(\e+n^{-1/4})C_{\xi_1 M T}. 
		\end{align*}


	\end{proof}
	
	\begin{proof}[Proof of Lemma \ref{M3}]
		For $t \in [0,T],$ recalling the definition of ${\J}_3^{\e,n}(t)$ from equation \eqref{M1234} and using the C-S inequality, we obtain
		\begin{equation*}
			|{\J}_3^{\e,n}(t)| \le \left(\int_0^t|D_2c(y_s,y_s)|^{2} ds\right)^\frac12 \left(\int_0^t \left|\int_{\pi^{n}(s)}^{s}\sigma(Y_{r}^{\e,n})  dW_r\right|^2 ds\right)^\frac12.
		\end{equation*}
		Taking supremum over $[0,T]$ followed by expectation, Jensen's inequality ($\E[X^{\frac12}]\leq [\E X]^{\frac12}$) and then using Hypothesis \ref{Derivative} for the boundedness of $D_2c(y_s,y_s)$, we have
		\begin{equation}\label{M31}
			\E\left[\sup_{0\le t \le T}|{\J}_3^{\e,n}(t)|\right]\leq 				C_T \left(\int_0^T \E  \left|\int_{\pi^{n}(s)}^{s}\sigma(Y_{r}^{\e,n}) dW_r\right|^2 ds\right)^\frac12.
		\end{equation}
		~Let $\sigma_i \in \R^d$, $1 \le i \le n$, represent the columns of the matrix $\sigma$, then using It\^{o} isometry and C-S's inequality, we get
		\begin{align*}
			\int_0^T \E \left|\int_{\pi^{n}(s)}^{s}\sigma(Y_{r}^{\e,n}) dW_r\right|^2ds
			&=\int_0^T \E \left|\sum_{i=1}^{n}\int_{\pi^{n}(s)}^{s}\sigma_i(Y_{r}^{\e,n}) dW_r^i\right|^2 ds\\
			&\lesssim \int_0^T\E \sum_{i=1}^{n}\left|\int_0^T\sigma_i(Y_{r}^{\e,n}) 1_{\{\pi^{n}(s), s\}}(r) dW_r^i \right|^2 ds \\
			&\lesssim  \int_0^T\sum_{i,j=1}^{n}\E\int_{0}^{T}\sigma_{ji}^2(Y_{r}^{\e,n})1_{\{\pi^{n}(s), s\}}dr ds\\
			&=\E \int_0^T\int_{\pi^{n}(s)}^{s}|\sigma(Y_{r}^{\e,n})|^2dr ds\\
			&\leq \left(\E\int_0^T\sup_{0\leq r\leq s}|\sigma(Y_{r}^{\e,n})|^4ds\right)^{\frac12}\left(\E\int_0^T (s-\pi^n(s))^2 ds\right)^{\frac12}.
		\end{align*}
		by C-S inequality. Now, using  \eqref{growthofnoisecoefficeients}, and then moment bound \eqref{Ynbound} and Corollary \ref{Np}, we get
		\begin{align*}
			\int_0^T \E \left|\int_{\pi^{n}(s)}^{s}\sigma(Y_{r}^{\e,n}) dW_r\right|^2ds &\leq C\left(\E\int_{0}^{T}\left( 1+ \sup_{0 \le r \le s} |Y_r^{\e,n}|^4\right)ds\right)^{\frac12}\left(\E\NN_2\right)^{\frac12}\\
			&\leq C_T\left(\E[\NN_2]\right)^{\frac12}. 
		\end{align*}
		Putting this last expression in equation \eqref{M31},  we get
		\begin{equation*}\label{M3square}
			\begin{aligned}
				\E\left[\sup_{0\le t \le T}|{\J}_3^{\e,n}(t)|\right]
				\leq C_T\left(\E[\NN_2]\right)^{\frac14}\le \frac{C_{T\xi_1}}{\sqrt{n}}. 
			\end{aligned}
		\end{equation*}
	\end{proof}			
	\ignore{\begin{rmkk} 
			For simplicity, we consider only Gaussian white noise in this work. The results can also be extended to the case where the SDE \eqref{GSDE} is driven by an additive compensated Poisson noise of the form
			$$
			\int_{0<|x|<1} F\left(Y^{\e,n}_{t-},x\right)\,\widetilde{N}(dt,dx),
			$$
			where $\widetilde{N}(dt,dx)$ denotes the compensated Poisson random measure. Assume that $F$ satisfies the usual Lipschitz and linear growth conditions. Under this additional  assumption, similar estimates can be obtained as in the white noise case. he main difference is that we now deal with jump processes. To handle them, we use  BDG type inequalities for jump processes and a careful application of the C-S type inequality. The remaining arguments follow in the same manner.
		\end{rmkk}
	}
	
	\section{Acknowledgment} The first author acknowledges the support of the Institute Postdoctoral Fellowship at Ashoka University, India. The second author acknowledges the support of the Institute Postdoctoral Fellowship at IIT Kanpur, India. He would also like to thank to Prof. Suprio Bhar (IIT Kanpur) for his support during the initial stage of this work through his CRG grant with file number : IITK/RD/AD/28487.

	\bibliographystyle{plain}    
	\bibliography{RandomSampling}
	
\end{document}